\documentclass[11pt]{article}
\usepackage{amsfonts}
\usepackage{mathrsfs}

\usepackage[latin1]{inputenc}
\usepackage{amsmath,amssymb,color}
\usepackage{latexsym}
\usepackage{epstopdf}
\usepackage{geometry}   
\usepackage[active]{srcltx}
\usepackage{epstopdf}
\usepackage[
bookmarks=true,         
bookmarksnumbered=true, 
colorlinks=true, pdfstartview=FitV, linkcolor=blue, citecolor=blue,
urlcolor=blue]{hyperref}

\newtheorem{theorem}{Theorem}[section]

\newtheorem{lemma}[theorem]{Lemma}
\newtheorem{proposition}[theorem]{Proposition}

\newtheorem{remark}[theorem]{Remark}
\numberwithin{equation}{section}
\newtheorem{example}[theorem]{Example}
\def\R{{\mathbb R}}
\def\E{{{\mathbb E}\,}}

\def\N{{\mathbb N}}

\def\I{\mathbf I}
\def\J{\mathbf J}\def\R{\mathbf R}

\def\wh{\widehat}

\def\Var{{\mathop {{\rm Var\, }}}}

\def\square{{\vcenter{\vbox{\hrule height.3pt
        \hbox{\vrule width.3pt height5pt \kern5pt
           \vrule width.3pt}
        \hrule height.3pt}}}}

\def\tlint{{- \kern-0.85em \int \kern-0.2em}}
\def\dlint{{- \kern-1.05em \int \kern-0.4em}}

\def\cF{{\cal F}}
\def\Om{{\Omega}}

\def\cF{{\cal F}}

\def\Om{{\Omega}}

\def \eref#1{\hbox{(\ref{#1})}}

\def \eref#1{\hbox{(\ref{#1})}}

\newenvironment{proof}[1][Proof]{\noindent\textit{#1.} }{\hfill \rule{0.5em}{0.5em}}

\def\nn{{\nonumber}}

\begin{document}

\title{Limit theorems for functionals of two independent Gaussian processes}
\date{\today}

\author{Jian Song, Fangjun Xu and Qian Yu \thanks{F. Xu is supported by National Natural Science Foundation of China (Grant No.11401215), Natural Science Foundation of Shanghai (16ZR1409700) and 111 Project (B14019).} }
\maketitle

\begin{abstract}
\noindent 

Under certain mild conditions, some limit theorems for functionals of two independent Gaussian processes are obtained. The results apply to general Gaussian processes including fractional Brownian motion, sub-fractional Brownian motion and bi-fractional Brownian motion. A new and interesting phenomenon is that, in comparison with the results for fractional Brownian motion,  extra randomness appears in the limiting distributions for Gaussian processes with nonstationary increments, say sub-fractional Brownian motion and bi-fractional Brownian. The results are obtained based on the method of moments, in which Fourier analysis, the chaining argument  introduced in \cite{nx1} and a paring technique are employed.

\vskip.2cm \noindent {\it Keywords:} Limit theorem, Gaussian processes, method of moments, chaining argument, paring technique.

\vskip.2cm \noindent {\it Subject Classification: Primary 60F17;
Secondary 60G15, 60G22.}
\end{abstract}

\section{Introduction}

Let $\left\{X_t= (X^1_t,\dots,
X^d_t), t\geq 0\right\} $ be a $d$-dimensional Gaussian process with component processes being independent copies of a 1-dimensional centered Gaussian process. We assume that there exist some $\alpha_1>0$ and $H\in(0,1)$ such that $\Var(X^1_t)=\alpha_1 t^{2H}$ for all $t\geq 0$. Some well known Gaussian processes  possessing this property,  say Brownian motions (Bms), fractional Brownian motions (fBms), sub-fractional Brownian motions (sub-fBms) and bi-fractional Brownian motions (bi-fBms). Let $\widetilde{X}$ be an independent copy of $X$.  When $X$ and $\widetilde{X}$ are fBms, we know that the intersection local time of $X$ and $\widetilde{X}$ does not exist if $Hd=2$ (\cite{nol, wu_xiao}), and this is called the critical case. If $X$ and $\widetilde{X}$ are fBms with $H\leq 1/2$, the following convergence in law was obtained in \cite{bx}.

\begin{theorem} \label{thm0} Suppose $Hd=2$ and $f$ is a real-valued bounded measurable function on $\R^d$ with $\int_{\R^d}|f(x)||x|^{\beta}\, dx<\infty$ for some $\beta>0$. Then, for any $t_1$ and $t_2\geq 0$,
\begin{align*}
\frac{1}{n} \int^{e^{nt_1}}_0\int^{e^{nt_2}}_0 f(B^{H,1}_u-B^{H,2}_v)\, du\, dv\overset{\mathcal{L}}{\longrightarrow} C_{f,d}\, (t_1\wedge t_2)\, N^2
 \end{align*}
as $n$ tends to infinity, where 
\begin{align*}
 C_{f,d}=\frac{d}{4}\, B(\frac{d}{4},\frac{d}{4})\, \frac{1}{(2\pi)^{\frac{d}{2}}}\int_{\R^d} f(x)\,dx
\end{align*}
with $B(\cdot,\cdot)$ being the Beta function, and $N$ is a real-valued standard normal random variable.
\end{theorem}

In this paper, we consider the asymptotic behavior of 
\begin{align} \label{variable}
\frac{1}{h(n)}\int^{e^{nt_1}}_0\int^{e^{nt_2}}_0 f(X_u-\widetilde{X}_v)\, du\, dv
\end{align}
as $n$ tends to $+\infty$, under certain mild conditions. 

The random variables in (\ref{variable}) appear in the  study of occupation times for the Gaussian random field $X_u-\widetilde{X}_v$ and their corresponding derivatives, see \cite{XChen, ghx, nol, oss, wu_xiao} and the references therein. 
It is of interest to find a normalization function $h(n)$ with proper growing speed as $n$ tends to infinity, so that \eqref{variable} converges to a non-trivial distribution. It turns out that the choice of $h(n)$ depends on $\int_{\R^d} f(x)dx$. That is, when $\int_{\R^d} f(x)\, dx\neq 0$ which corresponds to the first-order limit law, one may choose  $h(n)=n$; when $\int_{\R^d} f(x)\, dx= 0$ which corresponds to the second-order limit law, one needs to choose a normalization function  $h(n)=\sqrt n$ with slower growing speed.

To obtain the desired limit theorems for \eqref{variable}, we make the following assumptions on the Gaussian process $X$:
\begin{itemize}
\item[(A1)]  There exist constants $ \gamma_1\geq 1, \alpha_1>0$ and nonnegative decreasing functions $\phi_{1,i}(\varepsilon)$  on $[0,1/\gamma_1]$ with $\lim\limits_{\varepsilon\to 0}\phi_{1,i}(\varepsilon)=0$,  for $i=1,2$, such that
\[
0\leq t^{2H}(\alpha_1-\phi_{1,1}(h/t))\leq \mathrm{Var}(X^1_{t+h}-X^1_h)\leq t^{2H} (\alpha_1+\phi_{1,2}(h/t))
\]
for all $h\in[0, t/\gamma_1]$.

\item[(A2)]   There exist constants $\gamma_2\geq 1$, $\alpha_2>0$ and nonnegative decreasing functions $\phi_{2,i}(\varepsilon)$ on $[0,1/\gamma_2]$ with $ \lim\limits_{\varepsilon\to 0}\phi_{2,i}(\varepsilon)=0$, for $i=1,2,$ such that
\[
0\leq h^{2H}(\alpha_2-\phi_{2,1}(h/t))\leq \mathrm{Var}(X^1_{t+h}-X^1_t)\leq h^{2H}(\alpha_2+\phi_{2,2}(h/t))
\]
for all $h\in[0,t/\gamma_2]$.

\item[(B)]   Given $m\ge 1$, there exists a positive constant $\kappa$ depending on $m$, such that for any $0=s_0<s_1 <\cdots < s_m $ and $x_i \in \mathbb{R}^d$, $1\le i \le m$, we have
\[
\mathrm{Var} \Big( \sum_{i=1}^m x_i  \cdot (X_{s_i} -X_{s_{i-1}}) \Big) \ge \kappa \sum_{i=1}^m |x_i|^2 (s_i-s_{i-1})^{2H}.
\]

\item[(C1)] For any $0<t_1<t_2<t_3<t_4<\infty$ and $\gamma>1$, there exists a nonnegative decreasing function $\beta_1(\gamma)$ with $\lim\limits_{\gamma\to\infty}\beta_1(\gamma)=0$ such that,
if  $\frac{\Delta t_2}{\Delta t_4}\leq \frac{1}{\gamma}$ or $\frac{\Delta t_2}{\Delta t_4} \geq \gamma$, then 
 \[
\big|\E\big(X^1_{t_4}-X^1_{t_3}\big)\big(X^1_{t_2}-X^1_{t_1}\big)\big|\leq  \beta_1(\gamma)\, \left[\E\big(X^1_{t_4}-X^1_{t_3}\big)^2 \right]^{\frac{1}{2}}\left[\E\big(X^1_{t_2}-X^1_{t_1}\big)^2 \right]^{\frac{1}{2}},
\] 
where $\Delta t_i=t_i-t_{i-1}$ for $i=2,3,4$.

\item[(C2)] For any $0<t_1<t_2<t_3<t_4<\infty$ and $\gamma>1$, there exists a nonnegative decreasing function $\beta_2(\gamma)$ with $\lim\limits_{\gamma\to\infty}\beta_2(\gamma)=0$ such that, if $\frac{\Delta t_2}{\Delta t_3}\leq \frac{1}{\gamma}$ and $\frac{\Delta t_4}{\Delta t_3} \leq \frac{1}{\gamma}$, then
 \[
\big|\E\big(X^1_{t_4}-X^1_{t_3}\big)\big(X^1_{t_2}-X^1_{t_1}\big)\big|\leq \beta_2(\gamma)\, \left[\E\big(X^1_{t_4}-X^1_{t_3}\big)^2 \right]^{\frac{1}{2}}\left[\E\big(X^1_{t_2}-X^1_{t_1}\big)^2 \right]^{\frac{1}{2}}.
\] 
\end{itemize}

\begin{remark}
Note that the stationary increment property was used to obtain the limit laws for functionals of  fBm or fBms in the previous literatures \cite{bx, nx1, nx2, xu}.  In this work, we do not require the stationary increment property, but instead assume some weaker conditions  (A1) and (A2).  Assumption (B) characterizes the nondeterminism property of $X$, and it is satisfied if, for instance, $X$ is self-similar and has the local nondeterminism property. Assumption (C1) is required in Theorem \ref{thm1} (first-order limit law) and Theorem \ref{thm2} (second-order limit law), while Assumption (C2) is only needed in Theorem \ref{thm2}.
\end{remark}

The following are the main results of this paper.
\begin{theorem} \label{thm1} Under the assumptions $(A1)$, $(A2)$, $(B)$ and $(C1)$, we further suppose that $Hd=2$ and $f$ is a real-valued bounded measurable function on $\R^d$ with $\int_{\R^d}|f(x)||x|^{\beta}\, dx<\infty$ for some $\beta>0$. Then, for any $t_1$ and $t_2\geq 0$,
\begin{align*} 
\frac{1}{n} \int^{e^{nt_1}}_0\int^{e^{nt_2}}_0 f(X_u-\widetilde{X}_v)\, du\, dv\overset{\mathcal{L}}{\longrightarrow} C_{f,d}\, (t_1\wedge t_2)\, Z_{(\frac{\alpha_2}{\alpha_1})^{\frac{d}{4}}}\, \widetilde{Z}_{(\frac{\alpha_2}{\alpha_1})^{\frac{d}{4}}}\, N^2
 \end{align*}
as $n$ tends to infinity, where 
\begin{align*}
 C_{f,d}=\frac{d}{4}\, B(\frac{d}{4},\frac{d}{4})\, \frac{1}{(2\pi\alpha_2)^{\frac{d}{2}}}\int_{\R^d} f(x)\,dx
\end{align*}
with $B(\cdot,\cdot)$ being the Beta function, $Z_{\lambda}$ is a positive random variable with parameter $\lambda>0$ and $\E[Z^m_{\lambda}]=\frac{\Gamma(m+\lambda)}{m!\Gamma(\lambda)}$ for all $m\in\N$,  $\widetilde{Z}_{\lambda}$ is an independent copy of $Z_{\lambda}$, $N$ is a real-valued standard normal random variable independent of  $Z_{(\frac{\alpha_2}{\alpha_1})^{\frac{d}{4}}}$ and $\widetilde{Z}_{(\frac{\alpha_2}{\alpha_1})^{\frac{d}{4}}}$.
\end{theorem}

In this paper, the Fourier transform is given by, when $f\in L^1(\R^d)$,
\[
\wh f(\xi)=\int_{\R^d} f(x) e^{\iota\xi\cdot x}dx,
\] 
where $\iota=\sqrt{-1}$. 
\begin{theorem}  \label{thm2} Under the assumptions in Theorem \ref{thm1}, we further assume that  $\int_{\R^d} f(x)\, dx=0$ and $(C2)$. Then, for any $t_1$ and $t_2\geq 0$,
\begin{align} \label{thmrv}
\frac{1}{\sqrt{n}} \int^{e^{nt_1}}_0\int^{e^{nt_2}}_0 f(X_u-\widetilde{X}_v)\, du\, dv\overset{\mathcal{L}}{\longrightarrow} \, \sqrt{ D_{f,d}\, (t_1\wedge t_2)\, Z_{(\frac{\alpha_2}{\alpha_1})^{\frac{d}{4}}}\, \widetilde{Z}_{(\frac{\alpha_2}{\alpha_1})^{\frac{d}{4}}} N^2}\, \eta
 \end{align}
as $n$ tends to infinity, where 
\begin{align*}
 D_{f,d}=\frac{dB(\frac{d}{4},\frac{d}{4}) \Gamma^2(\frac{d+4}{4})}{\pi^{\frac{d}{2}}} \left(\frac{1}{(2\pi\alpha_2)^d}\int_{\R^d} |\widehat{f}(x)|^2|x|^{-d}\, dx\right)
 \end{align*} 
with $\Gamma(\cdot)$ being the Gamma function, and $\eta$ is another real-valued standard normal random variable independent of $N$, $Z_{(\frac{\alpha_2}{\alpha_1})^{\frac{d}{4}}}$ and $\widetilde{Z}_{(\frac{\alpha_2}{\alpha_1})^{\frac{d}{4}}}$.
\end{theorem}

As a byproduct, using similar arguments as in the proofs of Theorems \ref{thm1} and \ref{thm2}, we can easily obtain the following results.
\begin{theorem} \label{thm3} Under the assumptions $(A1)$, $(A2)$, $(B)$ and $(C1)$, we further suppose that $Hd=1$ and $f$ is a real-valued bounded measurable function on $\R^d$ with $\int_{\R^d}|f(x)||x|^{\beta}\, dx<\infty$ for some $\beta>0$. Then, for any $t\geq 0$,
\begin{align*}
\frac{1}{n} \int^{e^{nt}}_0f(X_u)\, du\overset{\mathcal{L}}{\longrightarrow}  \Big(\frac{1}{(2\pi\alpha_2)^{\frac{d}{2}}}\int_{\R^d} f(x)\, dx\Big)\, Z_{(\frac{\alpha_2}{\alpha_1})^{\frac{d}{2}}}\,  Z(t)
 \end{align*}
as $n$ tends to infinity, where $Z_{\lambda}$ is a positive random variable with parameter $\lambda>0$ and $\E[Z^m_{\lambda}]=\frac{\Gamma(m+\lambda)}{m!\Gamma(\lambda)}$ for all $m\in\N$,  $Z(t)$ is defined  in \cite{xu} and is independent of  $Z_{(\frac{\alpha_2}{\alpha_1})^{\frac{d}{2}}}$.
\end{theorem}

\begin{theorem}  \label{thm4} Under the assumptions in Theorem \ref{thm3}, we further assume that  $\int_{\R^d} f(x)\, dx=0$ and $(C2)$. Then, for any $t\geq 0$,
\begin{align*} 
\frac{1}{\sqrt{n}} \int^{e^{nt}}_0 f(X_u)\, du\overset{\mathcal{L}}{\longrightarrow} \, \sqrt{ \overline{D}_{f,d}\, Z_{(\frac{\alpha_2}{\alpha_1})^{\frac{d}{2}}}\, Z(t)}\, \eta
 \end{align*}
as $n$ tends to infinity, where 
\begin{align*}
\overline{D}_{f,d}=\frac{d\Gamma(\frac{d}{2})}{\pi^{\frac{d}{2}}} \left(\frac{1}{(2\pi\alpha_2)^d}\int_{\R^d} |\widehat{f}(x)|^2|x|^{-d}\, dx\right)
 \end{align*} 
with $\Gamma(\cdot)$ being the Gamma function, and $\eta$ is another real-valued standard normal random variable independent of $Z(t)$ and $Z_{(\frac{\alpha_2}{\alpha_1})^{\frac{d}{2}}}$.
\end{theorem}

\begin{remark}  \label{remark0} 
 For all $\lambda>0$, the distribution of $Z_{\lambda}$ is uniquely determined by its moments  $\E[Z^m_{\lambda}]=\frac{\Gamma(m+\lambda)}{m!\Gamma(\lambda)}, m\in\N$ (see, e.g., \cite{durrett}). In particular, $Z_{\lambda}$ follows the Beta($\lambda$, $1-\lambda$) distribution when $\lambda\in(0,1)$.  

 When $\alpha_1=\alpha_2$, for example in the fBm case, $\lambda=1$ and it is easy to see $Z_1=1$ a.s., and this is consistent with the known results in \cite{bx, nx1, nx2, xu}. When $\alpha_1\neq \alpha_2$, for example in the  sub-fBm or bi-fBm case, $\lambda\neq 1$ and  $Z_{\lambda}$ is a non-trivial random variable. Heuristically speaking, the loss of stationarity of increments introduces new random phenomenon in the limit laws.

\end{remark}

\begin{remark}  \label{remark1}
Since $f$ is bounded, one can always assume $\beta\leq 1$. Moreover, the assumption on $f$ also implies that $f\in L^p(\R^d)$ for any $p\geq 1$. When $\int_{\R^d} f(x)\, dx=0$,  $|\widehat{f}(\xi)|=|\widehat{f}(\xi)-\widehat{f}(0)|\leq c_{\alpha}|\xi|^{\alpha}$ for any $\alpha\in[0,\beta]$,  which yields the  finiteness of $\int_{\R^d} |\widehat{f}(\xi)|^2|\xi|^{-d}\, d\xi$ by Plancherel theorem.
\end{remark}

\begin{remark} When $X$ and $\widetilde{X}$ are independent copies of a $d$-dimensional fBm, denoting $p_\varepsilon(x)=\frac{1}{\varepsilon^d}p(\frac x\varepsilon)$ where $p(x)=(2\pi)^{-d/2}e^{-|x|^2/2}$, Theorem \ref{thm1} provides the exploding rate of
 $$I_{\varepsilon}(B^H, \widetilde{B}^{H}):= \int_0^T\int_0^T p_\varepsilon (B_u^H-\widetilde B_v^H) dudv $$  
 as $\varepsilon\to 0$ in the critical case $Hd=2$ (see, Theorem 1 in \cite{nol} and Remark 3.2 in \cite{wu_xiao}). Indeed, using the self-similarity of fBms and change of variables, one can get that $\int_0^T\int_0^T p_\varepsilon (B_u^H-\widetilde B_v^H) dudv$ has the same distribution as $  \int_0^{T\varepsilon^{-1/H}}\int_0^{T\varepsilon^{-1/H}} p (B_u^H-\widetilde B_v^H) dudv$, which  by Theorem \ref{thm1} explodes at the rate of  $ \log\varepsilon^{-1}$ as $\varepsilon$ tends to zero. 
\end{remark}

\begin{remark}\label{remark1.8} In fact, using similar arguments as in the proof of Lemma \ref{lemma0'},
\begin{align*}
 D_{f,d}
 &=\frac{4}{(2\pi)^\frac{d}{2}}\left(\frac{1}{(2\pi\alpha_2)^d}\int_{\R^d} |\widehat{f}(x)|^2|x|^{-d}\, dx\right) \left( \int^{+\infty}_0 e^{-\frac{1}{2}r^{2H}}\, dr\right)^2 \\
 & \qquad\qquad \qquad \qquad \times\lim_{n\to\infty}\frac{1}{n}\int^{e^{n}}_{1}\int^{e^{n}}_{1} (u^{2H}+v^{2H})^{-\frac{d}{2}}\, du\, dv.
\end{align*}
Moreover, comparing Theorem \ref{thm2} with Theorem 4 in \cite{biane}, one may obtain the following equality
\[
\int_{\R^4} |\widehat{f}(x)|^2 |x|^{-4}\, dx=-2\pi^2 \int_{\R^4}\int_{\R^4} f(x)f(y)\log|x-y|\, dx\, dy
\]
for all functions $f$ in $C^{\infty}_c(\R^4)$ with $\int_{\R^4} f(x)\, dx=0$.
\end{remark}

Limit theorems for functionals of two independent Brownian motions and their extensions were obtained in the 1980s, see \cite{legall86a, legall86b, biane} and references therein.  However, the corresponding results for fBms were not much since then. There are two main reasons. One is that the general fBm is neither a Markov process nor a semimartingale. This means that methods working for Bms probably fail for fBms. The other is that the role played by the second fBm in the limit laws is not well understood. Recently, Nualart and Xu in \cite{nx2} proved central limit theorems for functionals of two independent $d$-dimensional fractional Brownian motions in the case $Hd<2$. After that, Bi and Xu in \cite{bx} showed the first-order limit law in the critical case $Hd=2$ with $H\leq 1/2$, but it does not include the interesting case $d=3$ which may have physical relevance. The contribution of this paper is that, in the case $Hd=2$, for more general Gaussian processes other than just fBms, we obtain the first-order limit law and the  second-order limit law in Theorem \ref{thm1} and Theorem \ref{thm2}, respectively.  

Compared with the previous proofs of limit laws for fBms, we encounter some new challenges due to the lack of stationary increments property and short range dependence property,  both of which played critical roles in deriving limit laws for fBms with $H\le 1/2$. Moreover, the second Gaussian process  causes a big trouble when proving the convergence of even moments. Thanks to the methodologies developed in the recent papers \cite{nx2, bx, xu} and the introduction of some new ideas, especially the pairing technique, these issues are solved eventually.

To conclude the introduction, we briefly mention some of the innovations in this paper.

First of all, we do not assume the stationary increment property for our Gaussian processes. Instead, we propose two increment properties (A1) and (A2), which only concern the increment on a time interval whose length is significantly larger/smaller than the preceding interval.  So our results cover  several well-known Gaussian processes  besides fBms.   A surprising observation is that, compared with stationary increments,  non-stationary increments  cause extra random phenomena (see Remark \ref{remark0}).

Secondly, we characterize the type of increments of Gaussian processes that contribute to the moments of the limiting distribution in the case $Hd=2$. Roughly speaking, only increments on intervals with uncomparable lengths contribute in the first-order limit law. As for the second-order limit law, some increments on intervals far away also contribute.  The characterization of the increments in the first-order limit law is given in Assumption (C1), which is weaker than the one in Lemma 2.3 of \cite{xu} for fBm with Hurst index $H\leq 1/2$.   The characterization of the increments in the second-order limit law,  in addition to Assumption (C1), is given in Assumption (C2), which enables us to obtain the standard Gaussian random variable $\eta$ in Theorem \ref{thm2}.  Assumptions (C1) and (C2) are satisfied by fBms, sub-fBms and bi-fBms, see Lemmas \ref{fbm}, \ref{subfbm} and \ref{bifbm}.

Thirdly, the role played by the second Gaussian process in the second-order limit law in the case $Hd=2$ is clearly revealed.  For fBms with $Hd<2$, the role played by the second fBm  was explained in Lemma 3.2 and (3.22) of \cite{nx2}. It turns out that the second Gaussian process plays a similar role as the second fBm does in \cite{nx2}. However, noting that the method used in \cite{nx2}  cannot be applied directly here, we develop a new methodology, in which the key idea is to pair the second Gaussian process with the first one in a proper manner (see \textbf{Step 3} and \textbf{Step 4} in the proof of Proposition \ref{moments} for details). Moreover, this kind of paring technique indicates the relationship between the first-order limit law and the corresponding second-order limit law. We believe that our methodologies also work  well for a variety of functionals and multiparameter processes. For instance, one may use them to extend results in \cite{biane} to multiple independent Gaussian processes. In particular, the paring technique developed here could be used to obtain a functional version of the central limit theorem proved in \cite{nx2} and extension to more general Gaussian processes should also be available. This should be discussed in another paper.

The paper is outlined in the following way. After some preliminaries in Section 2, Section 3 is devoted to the proof of Theorem \ref{thm1} and Section 4 to the proof of Theorem \ref{thm2}, based on the method of moments, Fourier transform, the chaining argument  introduced in \cite{nx1} and a paring technique.

Throughout this paper, if not mentioned otherwise, the letter $c$, with or without a subscript, denotes a generic positive finite constant whose exact value is independent of $n$ and may change from line to line.  Moreover, we use $x\cdot y$ to denote the usual inner product in $\R^d$ and $B(0,r)$ the ball in $\R^d$ centered at the origin with radius $r$.

\bigskip

\section{Preliminaries}

Let $\left\{X_t= (X^1_t,\dots,
X^d_t), t\geq 0\right\} $ be a $d$-dimensional centered Gaussian process defined on some probability space $(\Om, \cF, P)$. The components
of $X$ are independent copies of a 1-dimensional centered Gaussian process. In this paper, we always assume that
$H=2/d\in(0,1)$ and that 
\begin{equation}\label{alpha0}
\Var(X^1_t)=\alpha_1 t^{2H},~~ \text{ for all } t\ge 0,
\end{equation}
where $\alpha_1>0$ is a constant that appears in Assumption (A1).
 This is a rather weak condition that is satisfied by a variety of Gaussian processes.  In particular, it is straightforward to validate the following Gaussian processes.

\begin{example} $X^1_t$ is a $1$-dimensional fBm, of which the covariance function is 
\[
\E(X^1_t X^1_s)=\frac{1}{2}(t^{2H}+s^{2H}-|t-s|^{2H}).\]
\end{example}
\begin{example} $X^1_t$ is a $1$-dimensional sub-fBm, of which the covariance function  is 
\[
\E(X^1_t X^1_s)=t^{2H}+s^{2H}-\frac{1}{2}[(t+s)^{2H}+|t-s|^{2H}].
\]
\end{example}
\begin{example} $X^1_t$ is a $1$-dimensional bi-fBm, of which the covariance function is 
\[
\E(X^1_t X^1_s)=2^{-K}[(t^{2H}+s^{2H})^{K}-|t-s|^{2HK}],
\]
where $H\in(0,1)$, $K\in(0,1]$ and $HK=2/d$.
\end{example}

It is easy to see that Assumptions (A1) and (A2) are satisfied by fBm with $$\alpha_1=\alpha_2=1, ~~  \phi_{1,1}(\varepsilon)=\phi_{1,2}(\varepsilon)=\phi_{2,1}(\varepsilon)=\phi_{2,2}(\varepsilon)\equiv0.$$ Using Taylor expansion, one can show that Assumptions (A1) and (A2) are satisfied by sub-fBm with 
$$\alpha_1=2-2^{2H-1},~~ \alpha_2=1, ~~\phi_{1,1}(\varepsilon)=\phi_{1,2}(\varepsilon)=c_1\,\varepsilon^{1\wedge 2H}, ~~\phi_{2,1}(\varepsilon)=\phi_{2,2}(\varepsilon)=c_2\,\varepsilon^{2-2H},$$ and by bi-fBms with $$\alpha_1=1,~~ \alpha_2=2^{1-K},~~ \phi_{1,1}(\varepsilon)=\phi_{1,2}(\varepsilon)=c_3\,\varepsilon^{1\wedge 2HK},~~ \phi_{2,1}(\varepsilon)=\phi_{2,2}(\varepsilon)=c_4\,\varepsilon^{2-2HK}.$$ 
Note that the constant $\alpha_1$  in \eqref{alpha0} coincide with $\alpha_1$ appearing in Assumption (A1).  Moreover, fBm, sub-fBm and bi-fBm satisfy Assumption (B) due to their self-similarity and local nondeterminism property, see \cite{berman, tudor, luan}.

In the sequel, we will show that Assumptions (C1) and (C2) are satisfied by fBm, sub-fBm, and bi-fBm.

\begin{lemma} \label{fbm}  Assumptions (C1) and (C2) are satisfied by fBms.
\end{lemma}
\begin{proof}  Let $B^H$ be a 1-dimensional fBm with Hurst index $H$. If $H\leq 1/2$, then Assumptions (C1) and (C2) follow from Lemma 2.3 in \cite{xu}. If $H>1/2$, then
\begin{align*}
\big|\E\big(B^{H}_{t_4}-B^{H}_{t_3}\big)\big(B^{H}_{t_2}-B^{H}_{t_1}\big)\big|
&=(\Delta t_4+\Delta t_3+\Delta t_2)^{2H}+(\Delta t_3)^{2H}-(\Delta t_4+\Delta t_3)^{2H}-(\Delta t_3+\Delta t_2)^{2H}.
\end{align*}
Without loss of generality, we can assume that $\Delta t_2=\theta_2 \Delta t_4$ with $0\leq\theta_2\leq \frac{1}{\gamma}<1$ and $\Delta t_3=\theta_3 \Delta t_4$ with $\theta_3\geq 0$.
Then, by $H>1/2$ and mean value theorem with $a,b\in(0,1)$,
\begin{align*}
\big|\E\big(B^{H}_{t_4}-B^{H}_{t_3}\big)\big(B^{H}_{t_2}-B^{H}_{t_1}\big)\big|
&=(\Delta t_4)^{2H}[(1+\theta_2+\theta_3)^{2H}+(\theta_3)^{2H}-(1+\theta_3)^{2H}-(\theta_2+\theta_3)^{2H}]\\
&=2H (\Delta t_4)^{2H}\theta_2 [(1+\theta_3+a\theta_2)^{2H-1}-(\theta_3+b\theta_2)^{2H-1}]\\
&\leq 2H (\Delta t_4)^{2H}\theta_2|1+(a-b)\theta_2|^{2H-1}\\
&\leq \frac{4}{\gamma^{1-H}} (\Delta t_2)^{H}(\Delta t_4)^{H}\\
&=\frac{4}{\gamma^{1-H}}\left[\E\big(B^H_{t_4}-B^H_{t_3}\big)^2\right]^{\frac{1}{2}}\left[\E\big(B^H_{t_2}-B^H_{t_1}\big)^2 \right]^{\frac{1}{2}}.
\end{align*}
where in the first inequality we use the inequality $|y^{\alpha}-x^{\alpha}|\leq |y-x|^{\alpha}$ with $\alpha\in(0,1]$.

This gives the desired inequality in Assumption (C1). We next show that Assumption (C2) is also satisfied by fBms when $H>1/2$. By mean value theorem, 
\begin{align*}
\big|\E\big(B^{H}_{t_4}-B^{H}_{t_3}\big)\big(B^{H}_{t_2}-B^{H}_{t_1}\big)\big|
&=|(\Delta t_4+\Delta t_3+\Delta t_2)^{2H}+(\Delta t_3)^{2H}-(\Delta t_4+\Delta t_3)^{2H}-(\Delta t_3+\Delta t_2)^{2H}|\\ 
&=H(2H-1)\Delta t_2(\Delta t_3+\Delta)^{2H-2}|\Delta t_4+c_1\Delta t_2-c_2\Delta t_2|\\
&\leq \Delta t_2(\Delta t_3+\Delta)^{2H-2}(\Delta t_2+\Delta t_4)\\
&\leq (\Delta t_3)^{2H-2} [(\Delta t_2)^{2}+\Delta t_2\Delta t_4],
\end{align*}
where $c_1, c_2\in (0,1)$ and $\Delta$ is a proper constant between $\Delta t_4+c_1\Delta t_2$ and $c_2\Delta t_2$. 

Similarly, 
\begin{align*}
\big|\E\big(B^{H}_{t_4}-B^{H}_{t_3}\big)\big(B^{H}_{t_2}-B^{H}_{t_1}\big)\big|
&\leq (\Delta t_3)^{2H-2} [(\Delta t_4)^{2}+\Delta t_2\Delta t_4].
\end{align*}

Since $\frac{\Delta t_2}{\Delta t_3}\leq \frac{1}{\gamma}$ and $\frac{\Delta t_4}{\Delta t_3} \leq \frac{1}{\gamma}$,
\begin{align*}
\big|\E\big(B^{H}_{t_4}-B^{H}_{t_3}\big)\big(B^{H}_{t_2}-B^{H}_{t_1}\big)\big|
&\leq (\Delta t_3)^{2H-2} [(\Delta t_2)^{2}\wedge (\Delta t_4)^{2} +\Delta t_2\Delta t_4]\\
&\leq 2(\Delta t_3)^{2H-2}\Delta t_2\Delta t_4\\
&\leq \frac{2}{\gamma^{2-2H}} (\Delta t_2)^H(\Delta t_4)^H\\
&=\frac{2}{\gamma^{2-2H}} \left[\E(B^{H}_{t_4}-B^{H}_{t_3})^2\right]^{\frac{1}{2}}\left[\E(B^{H}_{t_2}-B^{H}_{t_1})^2\right]^{\frac{1}{2}}.
\end{align*}
This completes the proof.
\end{proof}

\begin{lemma} \label{subfbm}  Assumptions (C1) and (C2) are satisfied by sub-fBms.
\end{lemma}
\begin{proof}  If $X^1_t$ is a $1$-dimensional sub-fBm, then 
\[
((2-2^{2H-1})\wedge 1) (t-s)^{2H}\leq \E\big(X^1_{t}-X^1_{s}\big)^2\leq ((2-2^{2H-1})\vee 1) (t-s)^{2H}.
\]
Let $B^H$ be a 1-dimensional fBm with Hurst index $H$. Then
\begin{align} \label{dcmpsub}
\E(X^1_{t_4}-X^1_{t_3})(X^1_{t_2}-X^1_{t_1})=& \frac{1}{2}\left[(t_1+t_4)^{2H}+(t_2+t_3)^{2H}-(t_2+t_4)^{2H}-(t_1+t_3)^{2H}\right] \nonumber\\
&\qquad\qquad+\E(B^{H}_{t_4}-B^{H}_{t_3})(B^{H}_{t_2}-B^{H}_{t_1}).
\end{align}
By Lemma \ref{fbm}, it suffices to show that the first term on the right-hand side of (\ref{dcmpsub}), i.e.,
\begin{align*}
I:&=\frac{1}{2} \left[(t_1+t_4)^{2H}+(t_2+t_3)^{2H}-(t_2+t_4)^{2H}-(t_1+t_3)^{2H}\right]
\end{align*}
satisfies Assumptions (C1) and (C2). Clearly $I=0$ if $H=1/2$. It suffices to show the case $H\neq 1/2$. 
 
For assumption (C1), if $\Delta t_2=\theta_2 \Delta t_4$ with $0\leq\theta_2\leq \frac{1}{\gamma}<1$, $\Delta t_1=\theta_1 \Delta t_4$ with $\theta_1\geq 0$ and $\Delta t_3=\theta_3 \Delta t_4$ with $\theta_3\geq 0$, then
\begin{align*}
|I|&\leq (\Delta t_4)^{2H}\Big|(2\theta_1+2\theta_2+\theta_3+1)^{2H}+(2\theta_1+\theta_2+\theta_3)^{2H}\\
&\qquad\qquad-(2\theta_1+\theta_2+\theta_3+1)^{2H}-(2\theta_1+2\theta_2+\theta_3)^{2H}\Big|.
\end{align*}
If $H<1/2$, the inequality $|y^{\alpha}-x^{\alpha}|\leq |y-x|^{\alpha}$ with $\alpha\in(0,1]$ implies that $|I|\leq (\Delta t_4)^{2H} \theta_2^{2H}$. If $H>1/2$, then by mean value theorem with $a,b\in(0,1)$,
\begin{align*}
|I|&\leq 2H (\Delta t_4)^{2H}\theta_2\Big[(2\theta_1+(1+a)\theta_2+\theta_3+1)^{2H-1}-(2\theta_1+(1+b)\theta_2+\theta_3)^{2H-1}\Big]\\
&\leq 2H (\Delta t_4)^{2H}\theta_2\, |1+(a-b)\theta_2|^{2H-1}\\
&\leq 4 (\Delta t_4)^{2H}\theta_2.
\end{align*}

If $\Delta t_4=\theta_4 \Delta t_2$ with $0\leq\theta_4\leq \frac{1}{\gamma}<1$, $\Delta t_1=\theta_1 \Delta t_2$ with $\theta_1\geq 0$ and $\Delta t_3=\theta_3 \Delta t_2$ with $\theta_3\geq 0$, then, using similar argument as above, we can show
\begin{align*}
|I|&\leq (\Delta t_2)^{2H}\Big|(2\theta_1+2+\theta_3+\theta_4)^{2H}+(2\theta_1+1+\theta_3)^{2H}\\
&\qquad\qquad-(2\theta_1+1+\theta_3+\theta_4)^{2H}-(2\theta_1+2+\theta_3)^{2H}\Big|\\
&\leq 4 (\Delta t_2)^{2H} \theta_4^{(2H\wedge 1)}
\end{align*}
for $H\neq 1/2$.  

So, for $H\neq 1/2$,
\begin{align*}
|I|\leq \frac{4}{(2-2^{2H-1})\wedge 1}(\frac{1}{\gamma^{H}}+\frac{1}{\gamma^{1-H}})\left[\E\big(X^1_{t_2}-X^1_{t_1}\big)^2 \right]^{\frac{1}{2}}\left[\E\big(X^1_{t_4}-X^1_{t_3}\big)^2\right]^{\frac{1}{2}}.
\end{align*}

Next we show that assumption (C2) is also satisfied by sub-fBms. Note that $\frac{\Delta t_2}{\Delta t_3}\leq \frac{1}{\gamma}$ and $\frac{\Delta t_4}{\Delta t_3} \leq \frac{1}{\gamma}$. If $H\neq 1/2$, mean value theorem yields that
\begin{align*}
I&=\frac{1}{2}\Big[(2\Delta t_1+\Delta t_2+\Delta t_3+\Delta t_4)^{2H}+(2\Delta t_1+2\Delta t_2+\Delta t_3)^{2H}\\
&\qquad-(2\Delta t_1+2\Delta t_2+\Delta t_3+\Delta t_4)^{2H}-(2\Delta t_1+\Delta t_2+\Delta t_3)^{2H}\Big]\\
&=H\Delta t_2[(2\Delta t_1+\Delta t_2+\Delta t_3+c_1)^{2H-1}-(2\Delta t_1+\Delta t_2+\Delta t_3+\Delta t_4+c_2)^{2H-1}]\\
&=H(2H-1)\Delta t_2(2\Delta t_1+\Delta t_2+\Delta t_3+c_3)^{2H-2}(c_1-c_2-\Delta t_4),
\end{align*}
where $c_1,c_2\in(0,\Delta t_2)$ and $c_3$ is between $c_1$ and $c_2+\Delta t_4$. 

This implies
\begin{align*}
|I|&\leq (2\Delta t_1+\Delta t_2+\Delta t_3)^{2H-2}(\Delta t_2+\Delta t_4)\Delta t_2\leq (\Delta t_3)^{2H-2}[(\Delta t_2)^{2}+\Delta t_2\Delta t_4].
\end{align*}

Similarly,  $|I|\leq(\Delta t_3)^{2H-2}[(\Delta t_4)^{2}+\Delta t_2\Delta t_4]$. Therefore,
\begin{align*}
|I|&\leq 2(\Delta t_3)^{2H-2}[\Delta t_2\Delta t_4]\\
& \leq \frac{2}{\gamma^{2-2H}}[(\Delta t_2)^{H}(\Delta t_4)^{H}]\\
&\leq \frac{2}{(2-2^{2H-1})\wedge 1}\frac{1}{\gamma^{2-2H}}\left[\E\big(X^1_{t_4}-X^1_{t_3}\big)^2\right]^{\frac{1}{2}}\left[\E\big(X^1_{t_2}-X^1_{t_1}\big)^2 \right]^{\frac{1}{2}}.
\end{align*}
The proof is completed.
\end{proof}

\begin{lemma} \label{bifbm} Assumptions (C1) and (C2) are satisfied by bi-fBms.
\end{lemma}
\begin{proof}  Note that bi-fBms are fBms when $K=1$. By Lemma \ref{fbm}, it suffices to consider the case $K<1$. If $X^1_t$ is a $1$-dimensional bi-fBm, then $(t-s)^{2HK}\leq \E\big(X^1_{t}-X^1_{s}\big)^2\leq 2^{1-K}|t-s|^{2HK}$, see \cite{yan}.  Let $B^{HK}$ be a 1-dimensional fBm with Hurst index $HK$. 
For any $0<t_1<t_2<t_3<t_4<\infty$,  we denote $s_i=t_i^{2H}$ for $i=1,2,3, 4$. Then
\begin{align} \label{dcmpbi}
\E(X^1_{t_4}-X^1_{t_3})(X^1_{t_2}-X^1_{t_1})=& 2^{-K}\left[(s_1+s_3)^K+(s_2+s_4)^K-(s_1+s_4)^K-(s_2+s_3)^K\right] \nonumber\\
&\qquad+ 2^{1-K}\E(B^{HK}_{t_4}-B^{HK}_{t_3})(B^{HK}_{t_2}-B^{HK}_{t_1}).
\end{align}

By Lemma \ref{fbm}, it suffices to show that the first term on the right-hand side of (\ref{dcmpbi}) satisfies Assumptions (C1) and (C2). Define
\begin{align*}
I:&=2^{-K}\left[(s_1+s_3)^K+(s_2+s_4)^K-(s_1+s_4)^K-(s_2+s_3)^K\right]\\
&=2^{-K}\Big[(\Delta t_1)^{2H}+(\Delta t_1+\Delta t_2+\Delta t_3)^{2H}\Big]^K+2^{-K}\Big[(\Delta t_1+\Delta t_2)^{2H}+(\Delta t_1+\Delta t_2+\Delta t_3+\Delta t_4)^{2H}\Big]^K\\
&\qquad-2^{-K}\Big[(\Delta t_1)^{2H}+(\Delta t_1+\Delta t_2+\Delta t_3+\Delta t_4)^{2H}\Big]^K-2^{-K}\Big[(\Delta t_1+\Delta t_2)^{2H}+(\Delta t_1+\Delta t_2+\Delta t_3)^{2H}\Big]^K.
\end{align*}

For Assumption (C1), if $H\leq 1/2$, then it is easy to see 
\begin{align*}
|I|
&\leq\Big[(\Delta t_1+\Delta t_2+\Delta t_3+\Delta t_4)^{2H}-(\Delta t_1+\Delta t_2+\Delta t_3)^{2H}\Big]^K\leq (\Delta t_4)^{2HK},
\end{align*}
where we use the increment property of concave function $y=x^K$ in the first inequality and $|y^{\alpha}-x^{\alpha}|\leq |y-x|^{\alpha}$ with $\alpha\in(0,1]$ in the second inequality.

Similarly, 
\begin{align*}
|I|
&\leq \Big[(\Delta t_1+\Delta t_2)^{2H}-(\Delta t_1)^{2H}\Big]^K\leq (\Delta t_2)^{2HK}.
\end{align*}
Therefore, 
\begin{align*}
|I|\leq \frac{1}{\gamma^{HK}}  (\Delta t_2)^{HK}(\Delta t_4)^{HK}\leq  \frac{1}{\gamma^{HK}} \left[\E\big(X^1_{t_4}-X^1_{t_3}\big)^2\right]^{\frac{1}{2}}\left[\E\big(X^1_{t_2}-X^1_{t_1}\big)^2 \right]^{\frac{1}{2}}.
\end{align*}

For the case $H>1/2$,  if $\Delta t_2=\theta_2 \Delta t_4$ with $0\leq\theta_2\leq \frac{1}{\gamma}<1$, $\Delta t_1=\theta_1 \Delta t_4$ with $\theta_1\geq 0$ and $\Delta t_3=\theta_3 \Delta t_4$ with $\theta_3\geq 0$, then 
\begin{align*}
|I|&\leq (\Delta t_4)^{2HK}\Big|\Big[(\theta_1)^{2H}+(\theta_1+\theta_2+\theta_3)^{2H}\Big]^K+\Big[(\theta_1+\theta_2)^{2H}+(\theta_1+\theta_2+\theta_3+1)^{2H}\Big]^K\\
&\qquad-\Big[(\theta_1)^{2H}+(\theta_1+\theta_2+\theta_3+1)^{2H}\Big]^K-\Big[(\theta_1+\theta_2)^{2H}+(\theta_1+\theta_2+\theta_3)^{2H}\Big]^K\Big|.
\end{align*}
If $\theta_1\in[0,1]$, then 
\begin{align*}
|I|
&\leq (\Delta t_4)^{2HK}\Big[\big[(\theta_1+\theta_2)^{2H}+(\theta_1+\theta_2+\theta_3)^{2H}\big]^K-\big[(\theta_1)^{2H}+(\theta_1+\theta_2+\theta_3)^{2H}\big]^K\Big]\\
&\leq (\Delta t_4)^{2HK}\big[(\theta_1+\theta_2)^{2H}-(\theta_1)^{2H}\big]^K\\
&\leq 2H(\theta_1+\theta_2)^{K(2H-1)}(\Delta t_4)^{2HK}\theta^K_2\\
&\leq 4(\Delta t_4)^{2HK}\theta^K_2, 
\end{align*}
where we use the increment property of concave function $y=x^K$ in the first inequality and $|y^{\alpha}-x^{\alpha}|\leq |y-x|^{\alpha}$ for $\alpha\in(0,1]$ in the second inequality.

If $\theta_1>1$, then by mean value theorem
\begin{align*}
|I|&\leq K(1-K)(\Delta t_4)^{2HK}\, \big[(\theta_1)^{2H}+(\theta_1+\theta_2+\theta_3)^{2H}\big]^{K-2}\big[(\theta_1+\theta_2)^{2H}-(\theta_1)^{2H}\big]\Big[(\theta_1+\theta_2)^{2H}-(\theta_1)^{2H}\\
&\qquad\qquad+(\theta_1+\theta_2+\theta_3+1)^{2H}-(\theta_1+\theta_2+\theta_3)^{2H}\Big]\\
&\leq  (\Delta t_4)^{2HK} (\theta_1+\theta_2+\theta_3)^{2H(K-2)}\big[(\theta_1+\theta_2)^{2H}-(\theta_1)^{2H}\big]\Big[(\theta_1+\theta_2+\theta_3+1)^{2H}-(\theta_1+\theta_2+\theta_3)^{2H}\Big]\\
&\leq  (\Delta t_4)^{2HK} (\theta_1+\theta_2+\theta_3)^{2H(K-2)}(2\theta_1+2\theta_2+2\theta_3)^{2(2H-1)}\theta_2\\
&\leq 4(\theta_1+\theta_2+\theta_3)^{2HK-2}(\Delta t_4)^{2HK}\theta_2\\
&\leq 4(\Delta t_4)^{2HK}\theta_2.
\end{align*}

If $\Delta t_4=\theta_4 \Delta t_2$ with $0\leq\theta_4\leq \frac{1}{\gamma}<1$, $\Delta t_1=\theta_1 \Delta t_2$ with $\theta_1\geq 0$ and $\Delta t_3=\theta_3 \Delta t_2$ with $\theta_3\geq 0$, then 
\begin{align*}
|I|&\leq (\Delta t_2)^{2HK}\bigg|\Big[(\theta_1)^{2H}+(\theta_1+1+\theta_3)^{2H}\Big]^K+\Big[(\theta_1+1)^{2H}+(\theta_1+1+\theta_3+\theta_4)^{2H}\Big]^K\\
&\qquad-\Big[(\theta_1)^{2H}+(\theta_1+1+\theta_3+\theta_4)^{2H}\Big]^K-\Big[(\theta_1+1)^{2H}+(\theta_1+1+\theta_3)^{2H}\Big]^K\bigg|\\
&\leq K(1-K)(\Delta t_2)^{2HK}\big[(\theta_1)^{2H}+(\theta_1+1+\theta_3)^{2H}\big]^{K-2}\big[(\theta_1+1+\theta_3+\theta_4)^{2H}-(\theta_1+1+\theta_3)^{2H}\big]\\
&\qquad\qquad\times\big[(\theta_1+1)^{2H}-(\theta_1)^{2H}+(\theta_1+1+\theta_3+\theta_4)^{2H}-(\theta_1+1+\theta_3)^{2H}\big]\\
&\leq 2 (\theta_1+1+\theta_3)^{2H(K-2)}(\theta_1+2+\theta_3)^{2H-1}(\theta_1+2+\theta_3)^{2H-1} (\Delta t_2)^{2HK} \theta_4\\
&\leq 8 (\theta_1+1+\theta_3)^{2HK-2}(\Delta t_2)^{2HK} \theta_4\\
&\leq 8 (\Delta t_2)^{2HK}\theta_4.
\end{align*}
So 
\begin{align*}
|I|\leq \frac{8}{\gamma^{K-HK}} (\Delta t_2)^{HK}(\Delta t_4)^{HK}\leq \frac{8}{\gamma^{K-HK}} \left[\E\big(X^1_{t_4}-X^1_{t_3}\big)^2\right]^{\frac{1}{2}}\left[\E\big(X^1_{t_2}-X^1_{t_1}\big)^2 \right]^{\frac{1}{2}}.
\end{align*}

For Assumption (C2), recall the condition $\frac{\Delta t_2}{\Delta t_3}\leq \frac{1}{\gamma}$ and $\frac{\Delta t_4}{\Delta t_3} \leq \frac{1}{\gamma}$ and then by mean value theorem,
\begin{align*}
I&=2^{-K}K\Big[(\Delta t_1+\Delta t_2+\Delta t_3+\Delta t_4)^{2H}-(\Delta t_1+\Delta t_2+\Delta t_3)^{2H}\Big]\\
&\qquad\qquad\times \Big[[(\Delta t_1+\Delta t_2)^{2H}+c_1]^{K-1}-[(\Delta t_1)^{2H}+c_2]^{K-1}\Big],
\end{align*}
where $c_1,c_2\in((\Delta t_1+\Delta t_2+\Delta t_3)^{2H}, (\Delta t_1+\Delta t_2+\Delta t_3+\Delta t_4)^{2H})$.

If $H\leq 1/2$, then, by the inequality $|y^{\alpha}-x^{\alpha}|\leq |y-x|^{\alpha}$ with $\alpha\in(0,1]$,
\begin{align*}
|I|\leq (\Delta t_4)^{2H}[(\Delta t_2)^{2H}+(\Delta t_4)^{2H}](\Delta t_3)^{2H(K-2)}.
\end{align*}

If $H>1/2$, then by mean value theorem
\begin{align*}
I&=2^{1-K}KH\Delta t_4(\Delta t_1+\Delta t_2+\Delta t_3+c_3)^{2H-1}\left[[(\Delta t_1+\Delta t_2)^{2H}+c_1]^{K-1}-[(\Delta t_1)^{2H}+c_2]^{K-1}\right]\\
&=2^{1-K}KH(K-1)\Delta t_4(\Delta t_1+\Delta t_2+\Delta t_3+c_3)^{2H-1}[(\Delta t_1+\Delta t_2)^{2H}-(\Delta t_1)^{2H}+c_1-c_2]  c_4^{K-2},
\end{align*}
where $c_3\in(0,\Delta t_4)$ and $c_4$ is between $(\Delta t_1+\Delta t_2)^{2H}+c_1$ and $(\Delta t_1)^{2H}+c_2$. 

So
\begin{align*}
|I|
&\leq  \Delta t_4(\Delta t_1+\Delta t_2+\Delta t_3+\Delta t_4)^{2H-1}|(\Delta t_1+\Delta t_2)^{2H}-(\Delta t_1)^{2H}+c_1-c_2| c_4^{K-2}\\
&\leq 2\left(1+\frac1\gamma\right)^{2(2H-1)}[\Delta t_4\Delta t_2+(\Delta t_4)^2] (\Delta t_3)^{2HK-2}\\
&\leq 8 [\Delta t_2\Delta t_4+(\Delta t_4)^2] (\Delta t_3)^{2HK-2},
\end{align*}
where we use the facts that $\Delta t_4\leq \frac{1}{\gamma}\Delta t_3$, $|c_1-c_2|\le (\Delta t_1+\Delta t_2+\Delta t_3+\Delta t_4)^{2H}-(\Delta t_1+\Delta t_2+\Delta t_3)^{2H}$, $c_4\geq (\Delta t_1+\Delta t_2+\Delta t_3)^{2H}$ and $K\in(0,1]$ in the second inequality.

Similarly,
\begin{align*}
I&=2^{-K}K[(\Delta t_1+\Delta t_2)^{2H}-(\Delta t_1)^{2H}]\\
&\qquad\qquad\times \left[[(\Delta t_1+\Delta t_2+\Delta t_3+\Delta t_4)^{2H}+c_5]^{K-1}-[(\Delta t_1+\Delta t_2+\Delta t_3)^{2H}+c_6]^{K-1}\right],
\end{align*}
where $c_5, c_6\in((\Delta t_1)^{2H}, (\Delta t_1+\Delta t_2)^{2H})$.

If $H\leq 1/2$, then it is easy to see that $|I|\leq (\Delta t_2)^{2H}[(\Delta t_2)^{2H}+(\Delta t_4)^{2H}](\Delta t_3)^{2H(K-2)}$. If $H>1/2$, then 
\begin{align*}
I&=2^{1-K}KH(K-1)\Delta t_2(\Delta t_1+c_7)^{2H-1}\\
&\qquad\qquad\times \left[[(\Delta t_1+\Delta t_2+\Delta t_3+\Delta t_4)^{2H}-(\Delta t_1+\Delta t_2+\Delta t_3)^{2H}+c_5-c_6]\right] c_8^{K-2},
\end{align*}
where $c_7\in(0,\Delta t_2)$ and $c_8$ is between $(\Delta t_1+\Delta t_2+\Delta t_3+\Delta t_4)^{2H}+c_5$ and $(\Delta t_1+\Delta t_2+\Delta t_3)^{2H}+c_6$. 

So
\begin{align*}
|I|
&\leq \Delta t_2(\Delta t_1+\Delta t_2+\Delta t_3+\Delta t_4)^{2(2H-1)} (\Delta t_2+\Delta t_4) (\Delta t_1+\Delta t_2+\Delta t_3)^{2H(K-2)}\\
&\leq 8[(\Delta t_2)^2+\Delta t_2\Delta t_4](\Delta t_3)^{2HK-2}.
\end{align*}
Therefore, for any $H\in(0,1)$ and $K\in(0,1)$,
\begin{align*}
|I|
&\leq 8[(\Delta t_2)^{4H\wedge 2}\wedge (\Delta t_4)^{4H\wedge 2} +\Delta t^{2H\wedge 1}_2\Delta t^{2H\wedge 1}_4](\Delta t_3)^{2HK-(4H\wedge 2)}\\
&\leq 16 (\Delta t_2)^{2H\wedge 1}(\Delta t_4)^{2H\wedge 1}(\Delta t_3)^{2HK-(4H\wedge 2)}\\
&\leq 16\gamma^{2HK-(4H\wedge 2)}\left[\E\big(X^1_{t_4}-X^1_{t_3}\big)^2\right]^{\frac{1}{2}}\left[\E\big(X^1_{t_2}-X^1_{t_1}\big)^2 \right]^{\frac{1}{2}}.
\end{align*}
This completes the proof.
\end{proof}

\bigskip

\section{Proof of Theorem \ref{thm1}}

In this section, we will prove Theorem \ref{thm1}. Some ideas will be borrowed from the proof of Theorem \ref{thm0} in \cite{bx}, in which the stationary increment property of fBm played a crucial role. Noting that the stationary increment property is not assumed in this article, new ideas would be introduced to obtain the desired limit law. 
 For the sake of clarity, we will spell out all the details.

For any $t_1>0$ and $t_2>0$, define
\[
F_n(t_1,t_2)=\int^{e^{nt_1}}_0\int^{e^{nt_2}}_0 f(X_u-\widetilde{X}_v)\, du\, dv.
\]
The following result shows that the limiting distribution of $\frac{1}{n}F_n(t_1,t_2)$ depends on $t_1\wedge t_2$. 
\begin{lemma}  \label{lemma3.1} 
\[
\lim_{n\to\infty}\frac{1}{n}\E\Big[|F_n(t_1,t_2)-F_n(t_1\wedge t_2,t_1\wedge t_2)|\Big]=0.
\]
\end{lemma}
\begin{proof}
Without loss of generality, we can assume $t_1\leq t_2$ and then obtain
\begin{align*}
\E\left[|F_n(t_1,t_2)-F_n(t_1,t_1)|\right]
&\leq \frac{1}{n}\; \E\left[\int^{e^{nt_1}}_0\int^{e^{nt_2}}_{e^{nt_1}} |f(X_u-\widetilde{X}_v)|\, du\, dv\right]\\
&\leq \frac{1}{n}\int^{e^{nt_1}}_0\int^{+\infty}_{e^{nt_1}}\int_{\R^d} |f(x)|(\alpha_1 u^{2H}+\alpha_1 v^{2H})^{-\frac{d}{2}}\, dx\, du\, dv\\
&\leq \frac{1}{\alpha_1^{d/2}n} \int_{\R^d} |f(x)|\, dx \int^{e^{nt_1}}_0\int^{+\infty}_{e^{nt_1}} u^{-2}\, du\, dv\\
&\leq \frac{1}{\alpha_1^{d/2}n} \int_{\R^d} |f(x)|\, dx,
\end{align*}
where in the second inequality we use the fact that the probability density function of $X_u-\widetilde{X}_v$ is less than $(2\pi)^{-d/2} (\alpha_1u^{2H}+\alpha_1 v^{2H})^{-\frac{d}{2}}$. This gives the desired result.
\end{proof}

\medskip

Now we only need to consider the limiting distribution of $\frac{1}{n}F_n(t,t)$ for $t>0$. For simplicity of notation, we write $F_n(t)$ for $\frac{1}{n}F_n(t,t)$. Using Fourier transform, $F_n(t)$ can be rewritten as 
\begin{align*}
F_n(t)=\frac{1}{(2\pi)^d n}\int^{e^{nt}}_0\int^{e^{nt}}_0\int_{\R^d} \widehat{f}(x)\, \exp\left(-\iota x\cdot (X_u-\widetilde{X}_v) \right)\, dx\, du\, dv.
\end{align*}

Let
\begin{align*} \label{gn}
G_n(t)=\frac{1}{(2\pi)^d n}  \int^{e^{nt}}_0\int^{e^{nt}}_0\int_{|x|<1} \widehat{f}(0)\, \exp\left(-\iota x\cdot (X_u-\widetilde{X}_v)\right)\, dx\, du\, dv.
\end{align*}
We show that $F_n(t)$ and $G_n(t)$ have the same limiting distribution.
\begin{lemma} \label{lemma3.2}
\[
\lim_{n\to\infty}\E\left[|F_n(t)-G_n(t)|\right]=0.
\]
\end{lemma}
\begin{proof} We first observe that
\begin{align*}
F_n(t)-G_n(t)=J_{n,1}(t)+J_{n,2}(t)+J_{n,3}(t)+J_{n,4}(t),
\end{align*}
where 
\begin{align*}
J_{n,1}(t)&=\frac{1}{n} \int_{[0,e^{nt}]^2-[1,e^{nt}]^2} f(X_u-\widetilde{X}_v)\, du\, dv,\\
J_{n,2}(t)&=\frac{1}{(2\pi)^d n}\int^{e^{nt}}_1\int^{e^{nt}}_1\int_{|x|\geq 1} \widehat{f}(x)\, \exp\left(-\iota x\cdot (X_u-\widetilde{X}_v) \right)\, dx\, du\, dv,\\
J_{n,3}(t)&=\frac{1}{(2\pi)^d n}\int^{e^{nt}}_1\int^{e^{nt}}_1\int_{|x|<1} \left(\widehat{f}(x)-\widehat{f}(0)\right) \exp\left(-\iota x\cdot (X_u-\widetilde{X}_v) \right)\, dx\, du\, dv,\\
J_{n,4}(t)&=-\frac{\widehat{f}(0)}{(2\pi)^d n}\int_{[0,e^{nt}]^2-[1,e^{nt}]^2}\int_{|x|<1} \exp\left(-\iota x\cdot (X_u-\widetilde{X}_v) \right)\, dx\, du\, dv.
\end{align*}

Since $f$ is bounded and integrable,
\begin{align*}
\E[|J_{n,1}(t)|]
&\leq \frac{1}{n} \sup_{x\in\R^d}|f(x)|  \int^1_0\int^1_0 du\, dv \\
&\qquad+\frac{1}{n} \Big(\int_{\R^d} |f(x)|\, dx\Big) \int^1_0\int^{e^{nt}}_1 (\alpha_1 u^{2H}+\alpha_1 v^{2H})^{-\frac{d}{2}}\, du\, dv \\
&\leq \frac{1}{n}\bigg(\sup_{x\in\R^d}|f(x)|+\int_{\R^d} |f(x)|\, dx\bigg).
\end{align*}
Now it suffices to show 
\[
\lim_{n\to\infty}\E[|J_{n,i}(t)|^2]=0
\]
for $i=2,3,4$.

When $i=2$,
\begin{align*}
&\E[|J_{n,2}(t)|^2]\\
&\leq \frac{1}{n^2} \int_{|x_1|\geq 1} \int_{|x_2|\geq 1}|\widehat{f}(x_1)\widehat{f}(x_2)|\\
&\qquad\times \bigg(\int_{[1,e^{nt}]^2}\exp\Big(-\frac{1}{2}\Var\left( x_2\cdot X_{u_2}+x_1\cdot X_{u_1}\right) \Big) \, du \bigg)^2\, dx\\
&\leq \frac{4}{n^2}\int_{[1,e^{nt}]^4} \int_{|x_1|\geq 1} \int_{|x_2|\geq 1}|\widehat{f}(x_1)|\widehat{f}(x_2)|1_{\{u_1\leq u_2,v_1\leq v_2\}}\\
&\qquad\times \exp\left(-\frac{1}{2}\Var\left( x_2\cdot X_{u_2}+x_1\cdot X_{u_1}\right)-\frac{1}{2}\Var\left( x_2\cdot X_{v_2}+x_1\cdot X_{v_1}\right) \right) \, dx\, du\, dv,
\end{align*}
where in the last inequality we use the Cauchy-Schwarz inequality.

By Assumption (B),
\begin{align*}
\E[|J_{n,2}(t)|^2]
&\leq  \frac{c_1}{n^2} \int_{[1,e^{nt}]^4} \int_{|x_1|\geq 1} \int_{|x_2|\geq 1}|\widehat{f}(x_2)|\, \exp\left(-\frac{\kappa}{2}\left(|x_2|^2(|u_2-u_1|^{2H}+|v_2-v_1|^{2H})\right)\right)\\
&\qquad \times \exp\left(-\frac{\kappa}{2}\left(|x_1+x_2|^2(u^{2H}_1+v^{2H}_1)\right)\right) \, dx\, du\, dv\\
&\leq \frac{c_2}{n^2}\bigg(\int_{|x_2|\geq 1}|\widehat{f}(x_2)||x_2|^{-d}\, dx_2\bigg) \bigg(\int_{[1,e^{nt}]^2} (u^{2H}_1+v^{2H}_1)^{-\frac{d}{2}}\, du_1\, dv_1\bigg)\\
&\leq \frac{c_3}{n},
\end{align*}
where the second inequality follows from integrating w.r.t to $x_1, u_2, v_2$ and Lemma \ref{lemma0}, and the last inequality is due to Lemma \ref{lemma1}.

When $i=3$, using inequality $|\widehat{f}(x)-\widehat{f}(0)|<c_{\beta} |x|^{\beta}$, Assumption (B) and Lemmas \ref{lemma0} and \ref{lemma1}, we can obtain that
\begin{align*}
\E[|J_{n,3}(t)|^2]
&\leq \frac{c_4}{n^2}\int_{[1,e^{nt}]^4} \int_{|x_1|<1} \int_{|x_2|<1}|x_1|^{\beta}|x_2|^{\beta}1_{\{u_1\leq u_2,v_1\leq v_2\}}\\
&\qquad\times \exp\left(-\frac{1}{2}\Var( x_2\cdot (X_{u_2}-\widetilde{X}_{v_2})+x_1\cdot (X_{u_1}-\widetilde{X}_{v_1})) \right)\, dx\, du\, dv\\
&\leq \frac{c_4}{n^2} \int_{[1,e^{nt}]^4} \int_{|x_1|<1} \int_{|x_2|<1}|x_2|^{\beta}\exp\left(-\frac{\kappa}{2}\left(|x_2|^2(|u_2-u_1|^{2H}+|v_2-v_1|^{2H})\right)\right)\\
&\qquad \times \exp\left(-\frac{\kappa}{2}\left(|x_1+x_2|^2(u_1^{2H}+v_1^{2H})\right)\right) \, dx\, du\, dv\\
&\leq \frac{c_5}{n^2} \bigg(\int_{|x_2|<1}|x_2|^{\beta-d} \, dx_2\bigg) \bigg(\int_{[1,e^{nt}]^2} (u_1^{2H}+v_1^{2H})^{-\frac{d}{2}}\, du_1\, dv_1\bigg)\\
&\leq \frac{c_6}{n}.
\end{align*}

When $i=4$, using similar arguments as $i=2$ and $i=3$, we can get
\begin{align*}
\E[|J_{n,4}(t)|^2]
&\leq \frac{c_7}{n^2}\int_{[0,1]^2\times [1,e^{nt}]^2} \int_{|x_1|<1} \int_{|x_2|<1} 1_{\{u_1\leq u_2,v_1\leq v_2\}}\\
&\qquad\times \exp\left(-\frac{1}{2}\Var( x_2\cdot (X_{u_2}-\widetilde{X}_{v_2})+x_1\cdot (X_{u_1}-\widetilde{X}_{v_1})) \right)\, dx\, du\, dv\\
&\leq \frac{c_7}{n^2} \int_{[0,1]^2\times[1,e^{nt}]^2} \int_{|x_1|<1} \int_{|x_2|<1} \exp\left(-\frac{\kappa}{2}\left(|x_2|^2((u_2-u_1)^{2H}+(v_2-v_1)^{2H})\right)\right)\\
&\qquad \times \exp\left(-\frac{\kappa}{2}\left(|x_1+x_2|^2(u_1^{2H}+v_1^{2H})\right)\right) \, dx\, du\, dv\\
&\leq \frac{c_8}{n^2}  \int_{|x|<1}  \bigg(\int^{e^{nt}}_0\exp\left(-\frac{\kappa}{2} |x|^2u^{2H}\right) du\bigg)^2 dx\\
&\leq \frac{c_9}{n}.
\end{align*}
This concludes the proof.
\end{proof}

For the simplicity of notation, we set 
\begin{equation}\label{gnt'}
\overline{G}_n(t)=\frac{1}{n}  \int^{e^{nt}}_0\int^{e^{nt}}_0 \int_{B(0,1)} \exp\left(-\iota x\cdot (X_u-\widetilde{X}_v)\right)\, dx \, du\, dv.
\end{equation}
Note that 
\begin{align} \label{G}
G_n(t)=\frac{\widehat{f}(0)}{(2\pi)^d}\, \overline{G}_n(t).
\end{align}
So the limiting distribution of $G_n(t)$ can be easily deduced from that of $\overline{G}_n(t)$.

We next give the limiting distribution of $\overline{G}_n(t)$.
\begin{proposition}  \label{prop}  Assume  the same conditions as in Theorem \ref{thm1}. Then, for any $t>0$,
\begin{equation*}  \label{e1}
\overline{G}_n(t)\overset{\mathcal{L}}{\longrightarrow} (\frac{2\pi}{\alpha_2})^{\frac{d}{2}}\, \frac{d}{4}\, B(\frac{d}{4},\frac{d}{4})\, t\, \, Z_{(\frac{\alpha_2}{\alpha_1})^{\frac{d}{4}}}\, \widetilde{Z}_{(\frac{\alpha_2}{\alpha_1})^{\frac{d}{4}}}\, N^2
\end{equation*}
as $n$ tends to infinity, where $B(\cdot,\cdot)$ is the Beta function, $Z_{\lambda}$ is a positive random variable with parameter $\lambda>0$ and $\E[Z^m_{\lambda}]=\frac{\Gamma(m+\lambda)}{m!\Gamma(\lambda)}$ for all $m\in\N$,  $\widetilde{Z}_{\lambda}$ is an independent copy of $Z_{\lambda}$, $N$ is a real-valued standard normal random variable independent of  $Z_{\lambda}$ and $\widetilde{Z}_{\lambda}$.
\end{proposition}
\begin{proof}  The proof is split into five steps for easier reading.

\noindent \textbf{Step 1.} We first show tightness. Let $\I^n_m$ be the $m$-th moment of $\overline{G}_n(t)$. Then
\begin{align*}
\I^n_m
&=\frac{1}{n^m} \int_{B^m(0,1)} \bigg(\int_{[0,e^{nt}]^m} \exp\Big(-\frac{1}{2}\Var\big(\sum\limits^m_{i=1}x_i\cdot X_{u_i}\big)\Big)\, du\bigg)^2\, dx. 
\end{align*}

Define
\[
I_n(x)=\int_{D_m} \exp\bigg(-\frac{1}{2}\Var\Big(\sum^m_{i=1} x_i\cdot X_{u_i}\Big)\bigg)\, du
\]
and
\[
I^{\sigma}_n(x)=\int_{D_m} \exp\bigg(-\frac{1}{2}\Var\Big(\sum^m_{i=1} x_{\sigma(i)}\cdot X_{u_i}\Big)\bigg)\, du
\]
for any $\sigma\in\mathscr{P}_m$, where $\mathscr{P}_m$ is the set consisting of all permutations of $\{1,2,\cdots,m\}$ and 
\[
D_{m}=\left\{0<u_1<\cdots<u_m<e^{nt}\right\}.
\] 
Then
\begin{align*}
\I^n_m
&=\frac{m!}{n^m} \sum_{\sigma\in\mathscr{P}_m}\int_{B^m(0,1)} I_n(x)\, I^{\sigma}_n(x)\,dx.
\end{align*}
Applying Cauchy-Schwarz inequality and Assumption (B),
\begin{align*}
\I^n_m
&\leq \frac{m!}{n^m} \sum_{\sigma\in \mathscr{P}_m} \Big(\int_{B^m(0,1)} \big(I_n(x)\big)^2\,dx  \Big)^{1/2}\Big(\int_{B^m(0,1)} \big(I^{\sigma}_n(x)\big)^2  dx  \Big)^{1/2}\\
&= \frac{(m!)^2}{n^m} \int_{B^m(0,1)} \left(I_n(x)\right)^2\,dx \\
&\leq \frac{(m!)^2}{n^m} \int_{B^m(0,1)} \Bigg(\int_{D_m} \exp\bigg(-\frac{\kappa}{2}\Big(\sum\limits^m_{i=1} |\sum\limits^m_{j=i}x_j|^2 (u_i-u_{i-1})^{2H} \Big)\bigg) du\Bigg)^2 dx.
\end{align*}

For $i=1,2\cdots,m$, we make the change of variables 
\begin{align} \label{change}
y_i=\sum\limits^m_{j=i}x_j\;\; \text{and}\;\; w_i=u_i-u_{i-1} 
\end{align}
with the convention $u_0=0$ and then obtain
\begin{align} \label{moment}  \nonumber
\I^n_m
&\leq \frac{(m!)^2}{n^m} \int_{B^m(0,m)} \Bigg(\int_{[0,e^{nt}]^{m}}  \exp\bigg(-\frac{\kappa}{2}\Big(\sum\limits^m_{i=1} |y_i|^2 w_i^{2H}\Big)\bigg) dw\Bigg)^2 dy\\ \nonumber
&= (m!)^2\left(\frac{1}{n}\int_{|y_1|<me^{nHt}}\left(\int^1_0  \exp\left(-\frac{\kappa}{2}  |y_1|^2 w_1^{2H} \right) dw_1\right)^2\, dy_1\right)^m\\ \nonumber
&\leq (m!)^2\left(\frac{c_1}{n}\int_{|y_1|<me^{nHt}} \left(1\wedge  |y_1|^{-d} \right)\, dy_1\right)^m\\ 
&\leq c_{m,H,t},
\end{align}
where the second inequality follows from Lemma \ref{lemma0}, and $c_{m,H,t}$ is a finite positive constant depending only on $m$, $H$ and $t$.

\medskip
\noindent \textbf{Step 2.} We show that $\I^n_m$ is asymptotically equal to $\I^n_{m,\gamma}$ defined in (\ref{s3gamma}) as $n\to\infty$. 

For any positive constant $\gamma>1$, let
\[
I_{n,\gamma}(x)=\int_{D_{m,\gamma}} \exp\bigg(-\frac{1}{2}\Var\Big(\sum^m_{i=1} x_i\cdot X_{u_i}\Big)\bigg)\, du
\]
and
\[
I^{\sigma}_{n,\gamma} (x)=\int_{D_{m,\gamma}} \exp\bigg(-\frac{1}{2}\Var\Big(\sum^m_{i=1} x_{\sigma(i)}\cdot X_{u_i}\Big)\bigg)\, du,
\]
where 
 \begin{equation*}\label{dmg}
D_{m,\gamma}=D_m-\bigcup_{1\leq k\neq \ell\leq m}\bigg\{\Delta u_\ell/\gamma<\Delta u_k<\gamma \Delta u_\ell\bigg\}
\end{equation*}
and $\Delta u_k=u_k-u_{k-1}$ with the convention $u_0=0$.  

Set
\begin{align} \label{s3gamma}
\I^n_{m,\gamma}
&=\frac{m!}{n^m} \sum_{\sigma\in \mathscr{P}_m}\int_{B^m(0,1)} I_{n,\gamma}(x) \, I^{\sigma}_{n,\gamma} (x)\, dx.
\end{align}
Then
\begin{align*}
\I^n_m-\I^n_{m,\gamma}
&=\frac{m!}{n^m} \sum_{\sigma\in \mathscr{P}_m}\int_{B^m(0,1)} \left[ \left(I_n(x)-I_{n,\gamma}(x)\right)I^{\sigma}_n(x)+ \left(I^{\sigma}_n(x)-I^{\sigma}_{n,\gamma}(x)\right)I_{n,\gamma}(x)\right] dx\\
&\leq \frac{2m!}{n^m} \sum_{\sigma\in \mathscr{P}_m}\int_{B^m(0,1)} \left[ \left(I_n(x)-I_{n,\gamma}(x)\right)I^{\sigma}_n(x)\right] dx.
\end{align*}

Using Cauchy-Schwarz inequality and then inequality (\ref{moment}), 
\begin{align} \label{eq1-1}
0\leq \I^n_m-\I^n_{m,\gamma}
&\leq c_1 \bigg(\frac{1}{n^m}\int_{B^m(0,1)}  \left(I_n(x)-I_{n,\gamma}(x)\right)^2\, dx\bigg)^{1/2}.
\end{align}
Note that
\begin{align} \label{eq1-2}  \nonumber
&\int_{B^m(0,1)}  \left(I_n(x)-I_{n,\gamma}(x)\right)^2\, dx\\  \nonumber
&= \int_{B^m(0,1)} \Bigg(\int_{D_m-D_{m,\gamma}} \exp\bigg(-\frac{1}{2}\Var\Big(\sum\limits^m_{i=1} \big(\sum\limits^m_{j=i}x_j\big)\cdot \big(X_{u_i}-X_{u_{i-1}}\big)\Big)\bigg) du\Bigg)^2 dx\\
&\leq   \int_{B^m(0,m)} \bigg(\int_{D_m-D_{m,\gamma}} \exp\Big(-\frac{\kappa}{2} \sum\limits^m_{i=1} |y_j|^2 w_i^{2H}  \Big) dw\bigg)^2 dy,
\end{align}
where in the last inequality we use the change of variables in (\ref{change}) and Assumption (B).  

Noting that 
\[
D_m-D_{m,\gamma}=D_m\bigcap \Big(\bigcup_{1\le k\neq \ell \le m}\big\{w_\ell/\gamma<w_k<\gamma w_\ell\big\}\Big),
\] 
we obtain by Lemmas \ref{lemma0} and \ref{lemma1},
\begin{align}  \label{eq1-3}   \nonumber
&\int_{B^m(0,1)}  \left(I_n(x)-I_{n,\gamma}(x)\right)^2\, dx\\   \nonumber
&\leq c_2\sum_{1\leq k\neq \ell\leq m} n^{m-2}\int_{|y_k|<m}\int_{|y_\ell|<m} \bigg(\int^{e^{nt}}_0 \int^{\gamma w_\ell}_{w_\ell/\gamma} \exp\left(-\frac{\kappa}{2}\left(|y_k|^2 w_k^{2H}+|y_\ell|^2 w_\ell^{2H}\right)\right) dw_k\, dw_\ell\bigg)^2 dy_k\, dy_{\ell}\\   \nonumber
&= c_2\sum_{1\leq k\neq \ell\leq m} n^{m-2}\int_{|y_k|<m}\int_{|y_\ell|<m} \int^{e^{nt}}_0\int^{e^{nt}}_0\int^{\gamma w_\ell}_{w_\ell/\gamma}\int^{\gamma\tau_\ell}_{\tau_\ell/\gamma}\\  \nonumber
&\quad \times \exp\left(-\frac{\kappa}{2}\left(|y_k|^2\left(w_k^{2H}+\tau_k^{2H}\right)+|y_\ell|^2\left(w_\ell^{2H}+\tau_\ell^{2H}\right)\right)\right) dw_k\, d\tau_k\, dw_\ell\, d\tau_\ell \, dy_k\, dy_{\ell}\\   \nonumber
&\leq c_3\sum_{1\leq k\neq \ell\leq m} n^{m-2} \int^{e^{nt}}_0\int^{e^{nt}}_0\int^{\gamma w_\ell}_{w_\ell/\gamma}\int^{\gamma\tau_\ell}_{\tau_\ell/\gamma} \left(w_k^{2H}+\tau_k^{2H}\right)^{-\frac{d}{2}} \Big(1\wedge \left(w_\ell^{2H}+\tau_\ell^{2H}\right)^{-\frac{d}{2}}\Big)  dw_k\,  d\tau_k\, dw_\ell\, d\tau_{\ell}\\   \nonumber
&\leq c_4\sum^m_{\ell=1}\,  (\ln \gamma)\,  n^{m-2} \int^{e^{nt}}_0\int^{e^{nt}}_0 \Big(1\wedge \left(w_\ell^{2H}+\tau_\ell^{2H}\right)^{-\frac{d}{2}}\Big)  \, dw_\ell\, d\tau_{\ell}\\ 
&\leq c_5\, (\ln \gamma)\, n^{m-1},
\end{align}
where $n^{m-2}$ is from the estimation 
\[
\left(\int_{|y|<m} \int_0^{e^{nt}}\int_0^{e^{nt}} \exp(-\frac\kappa2|y|^2(u^{2H}+v^{2H})) dudv dy \right)^{m-2}\le c_6\,n^{m-2}
\] 
and $\ln \gamma$ is due to Lemma \ref{lemma0'}.

Combining inequalities (\ref{eq1-1}), (\ref{eq1-2}) and (\ref{eq1-3}) gives
\begin{align}  \label{inm}
0\leq \I^n_m-\I^n_{m,\gamma}\leq c_7\, \sqrt{\frac{\ln \gamma}{n}} .
\end{align}

\medskip

\noindent \textbf{Step 3.} We establish the relationships  among  $\I^n_{m,\gamma}$, $\overline{\I}^n_{m,\gamma}(a, b)$ in (\ref{inmgb}), $\J_m^n(a_1, a_2)$ in (\ref{jnmga0}), $\J_{m, \gamma, 1}^n(a_1, a_2)$ in (\ref{jnmga}) and $\J_{m, \gamma,2}^n(a_1, a_2)$ in (\ref{jnmga2}), which are given in (\ref{jm1})-(\ref{jm2}) and (\ref{eq3-1})-(\ref{eq3-3}).

For any $a_1, a_2>0$, define
\begin{align}   \label{jnmga0}
\J^n_{m}(a_1,a_2)
&=\frac{m!}{n^m} \sum_{\sigma\in \mathscr{P}_m}\int_{B^m(0,a_1)} \int_{[0,a_2e^{nt}]^{2m}} \nn\\
&\quad \times \exp\bigg(-\sum\limits^m_{i=1} |y_i|^2u_i^{2H}-\sum\limits^m_{i=1}\Big|\sum\limits^m_{j=i} y_{\sigma(j)}-y_{\sigma(j)+1}\Big|^2 v_i^{2H}\Big) du\, dv\, dy,
\end{align}
\begin{align} \label{jnmga}
\J^n_{m,\gamma,1}(a_1,a_2)
&=\frac{m!}{n^m} \sum_{\sigma\in \mathscr{P}_m}\int_{B^m(0,a_1)} \int_{[0,a_2e^{nt}]^{2m}-O_{m,\gamma}} \nn \\
&\quad \times \exp\bigg(-\sum\limits^m_{i=1} |y_i|^2 u_i^{2H}-\sum\limits^m_{i=1}\Big|\sum\limits^m_{j=i} y_{\sigma(j)}-y_{\sigma(j)+1}\Big|^2 v_i^{2H}\bigg) du\, dv\, dy
\end{align}
and 
\begin{align} \label{jnmga2}
\J^n_{m,\gamma,2}(a_1,a_2)
&=\frac{m!}{n^m} \sum_{\sigma\in \mathscr{P}_m}\int_{B^m_{\gamma}(0,a_1)} \int_{[0,a_2e^{nt}]^{2m}} \nn\\
&\quad \times \exp\bigg(-\sum\limits^m_{i=1} |y_i|^2u_i^{2H}-\sum\limits^m_{i=1}\Big|\sum\limits^m_{j=i} y_{\sigma(j)}-y_{\sigma(j)+1}\Big|^2 v_i^{2H}\bigg) du\, dv\, dy,
\end{align}
where 
\[
O_{m,\gamma}=\bigcup_{1\leq k\neq \ell\leq m}\Big\{u_\ell/\gamma<u_k<\gamma u_\ell\; \text{or}\;  v_\ell/\gamma<  v_k<\gamma  v_\ell \Big\}
\] 
and 
\begin{equation}\label{bgm}
B^m_{\gamma}(0,a_1)=\Big\{y_i\in\R^d: |y_i|<a_1, i=1,2,\cdots,m\Big\}-\bigcup_{1\leq i\neq j\leq m}\Big\{|y_j|/\gamma<|y_i|<\gamma |y_j|\Big\}.
\end{equation}

Using similar arguments as in obtaining (\ref{inm}), we get
\begin{align} \label{jm1}
0\leq \J^n_{m}(a_1,a_2)-\J^n_{m,\gamma,1}(a_1,a_2)\leq c_8  \sqrt{\frac{\ln \gamma}{n}}
\end{align}
and
\begin{align} \label{jm2}
0\leq \J^n_{m}(a_1,a_2)-\J^n_{m,\gamma,2}(a_1,a_2)\leq c_9  \sqrt{\frac{\ln \gamma}{n}}.
\end{align}

For any $a, b>0$ and $\sigma_1,\sigma_2\in\mathscr{P}_m$, we define
\begin{align} \label{inmgb}
\overline{\I}^n_{m,\gamma}(a, b)
&=\frac{m!}{n^m}\sum_{\sigma\in \mathscr{P}_m} \int_{B^m(0,a)} \int_{D_{m,\gamma}\times D_{m,\gamma}}  \exp\bigg(-b\sum\limits^m_{i=1} |y_i|^2\E(X^1_{r_i}-X^1_{r_{i-1}})^2\bigg) \nn\\
&\qquad\qquad\qquad \times \exp\bigg(-b\sum\limits^m_{i=1}\Big|\sum\limits^m_{j=i} y_{\sigma(j)}-y_{\sigma(j)+1}\Big|^2 \E(X^1_{s_i}-X^1_{s_{i-1}})^2\bigg) dr\, ds\, dy,
\end{align}
\begin{align} \label{inmgbs'}
\overline{\I}_{m,\gamma}^n(a, b,\sigma_1,\sigma_2)=&\frac{m!}{n^m}\sum_{\sigma\in \mathscr{P}_m} \int_{B^m(0,a)} \int_{\widehat{O}_{m,\gamma}^{\sigma_1}\times \widehat{O}_{m,\gamma}^{\sigma_2}}  \exp\bigg(-b\sum\limits^m_{i=1} |y_i|^2\E(X^1_{r_i}-X^1_{r_{i-1}})^2\bigg) \nn\\
&\qquad \times \exp\bigg(-b\sum\limits^m_{i=1}\Big|\sum\limits^m_{j=i} y_{\sigma(j)}-y_{\sigma(j)+1}\Big|^2 \E(X^1_{s_i}-X^1_{s_{i-1}})^2\bigg) dr\, ds\, dy
\end{align}
and
\begin{align*} \label{inmgbs}
\overline{\I}^{n,\sigma_1,\sigma_2}_{m,\gamma}(a,b)
&=\frac{m!}{n^m} \sum_{\sigma\in \mathscr{P}_m}\int_{B^m(0,a)} \int_{\overline{O}_{m,\gamma}(b)\times \overline{O}_{m,\gamma}(b)} \exp\bigg(-\sum\limits^m_{i=1} |y_{\sigma_1(i)}|^2 u_{\sigma_1(i)}^{2H}\bigg) \nn\\
&\qquad\qquad\qquad \times \exp\bigg(-\sum\limits^m_{i=1}\Big|\sum\limits^m_{j={\sigma_2(i)}} y_{\sigma(j)}-y_{\sigma(j)+1}\Big|^2  v_{\sigma_2(i)}^{2H}\bigg) du\, dv\, dy,
\end{align*}
where 
\[
\widehat{O}^\sigma_{m,\gamma}=\left\{\sum^m_{i=1} \Delta r_i<e^{nt}, 0<\Delta r_{\sigma(i)}<\frac{\Delta r_{\sigma(i+1)}}{\gamma}\, \text{with}\, \Delta r_i=r_i-r_{i-1} \text{ for }\, i=1,\cdots,m-1\right\}
\]
and
\[
\overline{O}_{m,\gamma}(b)=\left\{\sum^m_{i=1} u_i<be^{nt}, 0<u_i<\frac{u_{i+1}}{\gamma}\; \text{ for }\; i=1,\cdots,m-1\right\}.
\]

It is easy to see that
\begin{align} \label{inmgbs0}
\overline{\I}^n_{m,\gamma}(a, b)=\sum_{\sigma_1,\sigma_2\in\mathscr{P}_m} \overline{\I}_{m,\gamma}^n(a, b,\sigma_1,\sigma_2)
\end{align}
and
\begin{align} \label{inmgbs1}
\overline{\I}^{n,\sigma_1,\sigma_2}_{m,\gamma}(a,b)
&=\frac{m!}{n^m} \sum_{\sigma\in \mathscr{P}_m}\int_{B^m(0,a)} \int_{\overline{O}_{m,\gamma}(b)\times \overline{O}_{m,\gamma}(b)} \exp\bigg(-\sum\limits^m_{i=1} |y_{\sigma_1(\sigma^{-1}_2(i))}|^2 u_{\sigma_1(\sigma^{-1}_2(i))}^{2H}\bigg) \nn\\
&\qquad\qquad\qquad \times \exp\bigg(-\sum\limits^m_{i=1}\Big|\sum\limits^m_{j=i} y_{\sigma(j)}-y_{\sigma(j)+1}\Big|^2  v_i^{2H}\bigg) du\, dv\, dy\nn\\
&=\frac{m!}{n^m} \sum_{\sigma\in \mathscr{P}_m}\int_{B^m(0,a)} \int_{\overline{O}_{m,\gamma}(b)\times \overline{O}_{m,\gamma}(b)} \exp\bigg(-\sum\limits^m_{i=1} |y_i|^2 u_i^{2H}\bigg) \nn\\
&\qquad\qquad\qquad \times \exp\bigg(-\sum\limits^m_{i=1}\Big|\sum\limits^m_{j=i} y_{\sigma(j)}-y_{\sigma(j)+1}\Big|^2  v_i^{2H}\bigg) du\, dv\, dy
\end{align}
for all $\sigma_1,\sigma_2\in\mathscr{P}_m$.

Now we compare $\overline{\I}^n_{m,\gamma}(a,b,\sigma_1,\sigma_2)$ with $\overline{\I}^{n,\sigma_1, \sigma_2}_{m,\gamma}(a,b)$. We take the following notations: 
\begin{align}
&\overline{\alpha}_1(\gamma)=\alpha_1-\phi_{1,1}(\gamma), &
\underline{\alpha}_1(\gamma)=\alpha_1+\phi_{1,2}(\gamma) ,\notag\\
&\overline{\alpha}_2(\gamma)=\alpha_2-\phi_{2,1}(\gamma),& \underline{\alpha}_2(\gamma)=\alpha_2+\phi_{2,2}(\gamma),\notag\\ &\overline{\alpha}(\gamma)=\overline{\alpha}_1(\gamma)\wedge \overline{\alpha}_2(\gamma), &\underline{\alpha}(\gamma)=\underline{\alpha}_1(\gamma)\vee \underline{\alpha}_2(\gamma).\label{eq-alpha}
\end{align}

For $\gamma$ sufficiently large in comparison with $\gamma_1$ and $ \gamma_2$, we can use Assumptions (A1)-(A2) to get upper and lower bounds for $\E(X^1_{r_i}-X^1_{r_{i-1}})^2$ and $\E(X^1_{s_i}-X^1_{s_{i-1}})^2$ in $\overline{\I}^n_{m,\gamma}(a,b,\sigma_1,\sigma_2)$, which are constant multiples of $(r_i-r_{i-1})^{2H}=:u_i^{2H}$ and $(s_i-s_{i-1})^{2H}=: v_i^{2H}$, respectively. For the lower bound case, the constants in front of $u_i^{2H}$ and $v_i^{2H}$ are $\overline\alpha_1(\gamma) $ or $\overline\alpha_2(\gamma)$ depending on the ratio of $u_i/r_{i-1}$ or $ v_i/s_{i-1}$, respectively. For the upper bound case, the constants in front of $ u_i^{2H}$ and $v_i^{2H}$ are $\underline{\alpha}_1(\gamma)$ or $\underline{\alpha}_2(\gamma)$ depending on the ratio of $u_i/r_{i-1}$ or $ v_i/s_{i-1}$, respectively. 

Using lower bounds for $\E(X^1_{r_i}-X^1_{r_{i-1}})^2$ and $\E(X^1_{s_i}-X^1_{s_{i-1}})^2$  in \eqref{inmgbs'}, then applying change of variables to $u_i=r_i-r_{i-1}$ and $v_i=s_i-s_{i-1}$,  we have
\begin{align}\label{eq-3.18}
\overline{\I}^n_{m,\gamma}(a,b,\sigma_1, \sigma_2)
&\leq  b^{-\frac{md}{2}}  (\overline{\alpha}_1(\gamma))^{-\frac{d}{4}(|\sigma_1|+|\sigma_2|)}(\overline{\alpha}_2(\gamma))^{-\frac{d}{4}(m-|\sigma_1|+m-|\sigma_2|)} \overline{\I}^{n,\sigma_1,\sigma_2}_{m,\gamma}(a, (b \overline{\alpha}(\gamma))^{\frac{1}{2H}}),
\end{align}
where 
\begin{align}\label{|sigma|}
&|\sigma_1|=\#\{ u_k\leq r_{k-1}/\gamma:\,  u_{\sigma_1(i)}\leq  u_{\sigma_1(i+1)}/{\gamma}\, \text{ for }\, i=1,2,\cdots, m-1\},\nn\\
&|\sigma_2|=\#\{v_k\leq s_{k-1}/\gamma:\, v_{\sigma_2(i)}\leq  v_{\sigma_2(i+1)}/{\gamma}\, \text{ for }\, i=1,2,\cdots, m-1\}.
\end{align}

On the other hand, recalling $\J^n_{m,\gamma,1}(a_1,a_2)$ defined in (\ref{jnmga}) and noting that $\overline{O}_{m,\gamma}(b)$ is one of the $m!$ partitions of  
\[
\Big\{\sum_{i=1}^m u_i<be^{nt},\, u_i/u_j\notin (1/\gamma, \gamma) \text{ for }\; i\neq j \Big\},
\] 
we have
\begin{equation}\label{ilej}
\frac{1}{(m!)^2}\J^n_{m,\gamma,1}(a,b/m)\le \overline{\I}^{n,\sigma_1,\sigma_2}_{m,\gamma}(a,b) \le \frac{1}{(m!)^2}\J^n_{m,\gamma,1}(a,b).
\end{equation}
Next, by \eqref{inmgbs0}, \eqref{inmgbs1}, \eqref{eq-3.18} and (\ref{ilej}), we obtain
\begin{align} \label{eq3-1}
\overline{\I}^n_{m,\gamma}(a,b) &\leq \sum_{\sigma_1,\sigma_2\in\mathscr{P}_m} b^{-\frac{md}{2}}  (\overline{\alpha}_1(\gamma))^{-\frac{d}{4}(m-|\sigma_1|+m-|\sigma_2|)}(\overline{\alpha}_2(\gamma))^{-\frac{d}{4}(|\sigma_1|+|\sigma_2|)} \frac{1}{(m!)^2}\J^n_{m,\gamma,1}(a,(b \overline{\alpha}(\gamma))^{\frac{1}{2H}})\nn\\
&= (b\overline{\alpha}_2(\gamma))^{-\frac{md}{2}}  \left[ \prod^m_{i=1} \left( \left(\frac{\overline{\alpha}_2(\gamma)}{\overline{\alpha}_1(\gamma)}\right)^{\frac{d}{4}}+(i-1)\right)\right]^2  \frac{1}{(m!)^2}\J^n_{m,\gamma,1}(a,(b \overline{\alpha}(\gamma))^{\frac{1}{2H}})\nn\\
&= (b\overline{\alpha}_2(\gamma))^{-\frac{md}{2}} \left[ \frac{\Gamma(m+(\frac{\overline{\alpha}_2(\gamma)}{\overline{\alpha}_1(\gamma)})^{\frac{d}{4}})}{\Gamma((\frac{\overline{\alpha}_2(\gamma)}{\overline{\alpha}_1(\gamma)})^{\frac{d}{4}})} \right]^2  \frac{1}{(m!)^2}\J^n_{m,\gamma,1}(a,(b \overline{\alpha}(\gamma))^{\frac{1}{2H}})\nn\\
&=(b\overline{\alpha}_2(\gamma))^{-\frac{md}{2}} \left[ \frac{\Gamma(m+(\frac{\overline{\alpha}_2(\gamma)}{\overline{\alpha}_1(\gamma)})^{\frac{d}{4}})}{m!\Gamma((\frac{\overline{\alpha}_2(\gamma)}{\overline{\alpha}_1(\gamma)})^{\frac{d}{4}})} \right]^2  \J^n_{m,\gamma,1}(a, (b \overline{\alpha}(\gamma))^{\frac{1}{2H}}),
\end{align}
where the first equality follows from Lemma \ref{lemma4}.

Similarly, for the lower bound of $\overline{\I}^n_{m,\gamma}(a,b)$, we have
\begin{align} \label{eq3-2}
\overline{\I}^n_{m,\gamma}(a,b) &\geq \sum_{\sigma_1,\sigma_2\in\mathscr{P}_m} b^{-\frac{md}{2}}  (\underline{\alpha}_1(\gamma))^{-\frac{d}{4}(m-|\sigma_1|+m-|\sigma_2|)}(\underline{\alpha}_2(\gamma))^{-\frac{d}{4}(|\sigma_1|+|\sigma_2|)} \frac{1}{(m!)^2}\J^n_{m,\gamma,1}(a,(b \underline{\alpha}(\gamma))^{\frac{1}{2H}}/m)\nn\\
&= b^{-\frac{md}{2}} \left[ \prod^m_{i=1} \left( \underline{\alpha}_1(\gamma)^{-\frac{d}{4}}+(i-1)\underline{\alpha}_2(\gamma)^{-\frac{d}{4}}\right)\right]^2  \frac{1}{(m!)^2}\J^n_{m,\gamma,1}(a,(b \underline{\alpha}(\gamma))^{\frac{1}{2H}}/m)\nn\\
&= (b\underline{\alpha}_2(\gamma))^{-\frac{md}{2}} \left[ \frac{\Gamma(m+(\frac{\underline{\alpha}_2(\gamma)}{\underline{\alpha}_1(\gamma)})^{\frac{d}{4}})}{\Gamma((\frac{\underline{\alpha}_2(\gamma)}{\underline{\alpha}_1(\gamma)})^{\frac{d}{4}})} \right]^2 \frac{1}{(m!)^2} \J^n_{m,\gamma,1}(a, (b \underline{\alpha}(\gamma))^{\frac{1}{2H}}/m)\nn\\
&= (b\underline{\alpha}_2(\gamma))^{-\frac{md}{2}} \left[ \frac{\Gamma(m+(\frac{\underline{\alpha}_2(\gamma)}{\underline{\alpha}_1(\gamma)})^{\frac{d}{4}})}{m!\Gamma((\frac{\underline{\alpha}_2(\gamma)}{\underline{\alpha}_1(\gamma)})^{\frac{d}{4}})} \right]^2 \J^n_{m,\gamma,1}(a, (b \underline{\alpha}(\gamma))^{\frac{1}{2H}}/m).
\end{align}

Finally in this step, we provide the relationship between $\I_{m,\gamma}^n$ in (\ref{s3gamma}) and $\overline{\I}_{m,\gamma}^n(a,b)$ in (\ref{inmgb}).  When $\gamma$ is large enough, Assumption (C1) yields
\begin{align} \label{eq3-3}
\overline{\I}^n_{m,\gamma}(1/m,\underline{b}(\gamma))\leq \I^n_{m,\gamma}\leq \overline{\I}^n_{m,\gamma}(m,\overline{b}(\gamma)),
\end{align}
where $\overline{b}(\gamma)=\frac{1}{2}-m\beta_1(\gamma)$ and $\underline{b}(\gamma)=\frac{1}{2}+m\beta_1(\gamma)$ with $\beta_1(\gamma)$ given in Assumption (C1).

\medskip
\noindent \textbf{Step 4.}  We obtain estimates for $\I^n_{m}$.  For positive numbers $a_1, a_2, b_1$ and $b_2$, define
\begin{align*} \label{def1} \nonumber
\R^n_{m}(a_1,a_2,b_1,b_2)
&= \frac{m!}{n^m} \sum_{\sigma\in \mathscr{P}_m}\int_{B^m(0,a_1)} \int_{[0,a_2e^{nt}]^{2m}} \\
&\qquad \times \exp\bigg(-b_1\sum\limits^m_{i=1} |y_i|^2 u_i^{2H}-b_2\sum\limits^m_{i=1}\sup_{j\in A^{\sigma}_i} |y_j|^2 v_i^{2H}\bigg) du\, dv\, dy
\end{align*}
and
\begin{align*}
\R^n_{m,\gamma}(a_1,a_2,b_1,b_2)
&= \frac{m!}{n^m} \sum_{\sigma\in \mathscr{P}_m}\int_{B^m_{\gamma}(0,a_1)} \int_{[0,a_2e^{nt}]^{2m}} \\
&\qquad \times \exp\bigg(-b_1\sum\limits^m_{i=1} |y_i|^2 u_i^{2H}-b_2\sum\limits^m_{i=1}\sup_{j\in A^{\sigma}_i} |y_j|^2  v_i^{2H}\bigg) du\, dv\, dy,
\end{align*}
where 
\[
A^{\sigma}_i=\big\{\sigma(i), \cdots, \sigma(m)\big\}\Delta \big\{\sigma(i)+1, \cdots, \sigma(m)+1\big\}
\]
with $\Delta$ being the symmetric difference operator for two sets. 

Using similar arguments for obtaining (\ref{inm}), we get
\begin{align} \label{rm}
\R^n_{m}(a_1,a_2,b_1,b_2)-\R^n_{m,\gamma}(a_1,a_2,b_1,b_2)\leq c_{10}\sqrt{\frac{\ln \gamma}{n}}.
\end{align}
Note that when $\gamma$ is sufficiently large, for $(y_1,\cdots, y_m)$ in the set $B_\gamma^m(0,1)$ defined in \eqref{bgm},
\begin{equation}\label{iodd}
(1-m/\gamma)\sup_{j\in A_i^\sigma}|y_j|\le \bigg|\sum_{j=i}^m y_{\sigma(j)}-y_{\sigma(j)+1}\bigg|=\bigg|\sum_{j\in A_i^\sigma} y_j\bigg| \le (1+m/\gamma) \sup_{j\in A_i^\sigma}|y_j|.
\end{equation}

Thanks to (\ref{inm}), (\ref{jm1}), (\ref{jm2}), (\ref{eq3-1}), (\ref{eq3-2}), (\ref{eq3-3}) and \eqref{iodd}, we get that, when $\gamma$ is  sufficiently large,
\begin{align*}
\I^n_m
&\leq c_{11}  \sqrt{\frac{\ln \gamma}{n}} +(\overline{b}(\gamma)\overline{\alpha}_2(\gamma))^{-\frac{md}{2}} \left[ \frac{\Gamma(m+(\frac{\overline{\alpha}_2(\gamma)}{\overline{\alpha}_1(\gamma)})^{\frac{d}{4}})}{m! \Gamma((\frac{\overline{\alpha}_2(\gamma)}{\overline{\alpha}_1(\gamma)})^{\frac{d}{4}})} \right]^2 \J^n_{m,\gamma,2}(m,(\overline{b}(\gamma)\overline{\alpha}(\gamma))^{\frac{1}{2H}})\\
&\leq c_{11}\sqrt{\frac{\ln \gamma}{n}}+(\overline{b}(\gamma)\overline{\alpha}_2(\gamma))^{-\frac{md}{2}}\left[ \frac{\Gamma(m+(\frac{\overline{\alpha}_2(\gamma)}{\overline{\alpha}_1(\gamma)})^{\frac{d}{4}})}{m!\Gamma((\frac{\overline{\alpha}_2(\gamma)}{\overline{\alpha}_1(\gamma)})^{\frac{d}{4}})} \right]^2 \R^n_{m}(m,(\overline{b}(\gamma)\overline{\alpha}(\gamma))^{\frac{1}{2H}},1,1-m/\gamma)
\end{align*}
and
\begin{align*}
\I^n_m
&\geq -c_{12} \sqrt{\frac{\ln \gamma}{n}} +(\underline{b}(\gamma)\underline{\alpha}_2(\gamma))^{-\frac{md}{2}} \left[ \frac{\Gamma(m+(\frac{\underline{\alpha}_2(\gamma)}{\underline{\alpha}_1(\gamma)})^{\frac{d}{4}})}{m!\Gamma((\frac{\underline{\alpha}_2(\gamma)}{\underline{\alpha}_1(\gamma)})^{\frac{d}{4}})} \right]^2\J^n_{m,\gamma,2}(1/m,(\underline{b}(\gamma)\underline{\alpha}(\gamma))^{\frac{1}{2H}}/ m)\\
&\geq -c_{12} \sqrt{\frac{\ln \gamma}{n}} +(\underline{b}(\gamma)\underline{\alpha}_2(\gamma))^{-\frac{md}{2}}\left[ \frac{\Gamma(m+(\frac{\underline{\alpha}_2(\gamma)}{\underline{\alpha}_1(\gamma)})^{\frac{d}{4}})}{m!\Gamma((\frac{\underline{\alpha}_2(\gamma)}{\underline{\alpha}_1(\gamma)})^{\frac{d}{4}})} \right]^2 \R^n_{m,\gamma}(1/m,(\underline{b}(\gamma)\underline{\alpha}(\gamma))^{\frac{1}{2H}}/ m,1,1+m/\gamma)\\
&\geq -c_{13} \sqrt{\frac{\ln \gamma}{n}} +(\underline{b}(\gamma)\underline{\alpha}_2(\gamma))^{-\frac{md}{2}}\left[ \frac{\Gamma(m+(\frac{\underline{\alpha}_2(\gamma)}{\underline{\alpha}_1(\gamma)})^{\frac{d}{4}})}{m!\Gamma((\frac{\underline{\alpha}_2(\gamma)}{\underline{\alpha}_1(\gamma)})^{\frac{d}{4}})} \right]^2 \R^n_{m}(1/m,(\underline{b}(\gamma)\underline{\alpha}(\gamma))^{\frac{1}{2H}}/ m,1,1+m/\gamma),
\end{align*}
where we use (\ref{rm}) in the last inequalities.

\medskip
\noindent \textbf{Step 5.} We obtain the limit of $\I^n_{m}$, which is also the limit of the $m$-th moment of $\overline{G}_n(t)$ defined in \eqref{gnt'}. Using Lemma 3.4 in  \cite{bx}, 
\begin{align*}
\limsup_{n\to\infty}\I^n_m
&\leq \left(\frac{2}{\alpha_2}\right)^{\frac{md}{2}} \left[ \frac{\Gamma(m+(\frac{\alpha_2}{\alpha_1})^{\frac{d}{4}})}{m!\Gamma((\frac{\alpha_2}{\alpha_1})^{\frac{d}{4}})} \right]^2 \limsup_{\gamma\to\infty}\limsup_{n\to\infty} \R^n_{m}(m, (\overline{b}(\gamma)\overline{\alpha}(\gamma))^{\frac{1}{2H}},1,1-m/\gamma)\\
&=\left(\frac{2}{\alpha_2}\right)^{\frac{md}{2}} \left[ \frac{\Gamma(m+(\frac{\alpha_2}{\alpha_1})^{\frac{d}{4}})}{m!\Gamma((\frac{\alpha_2}{\alpha_1})^{\frac{d}{4}})} \right]^2\bigg(\frac{2t\, \pi^{\frac{d}{2}} \Gamma^2(\frac{d+4}{4})}{\Gamma(\frac{d+2}{2})}\bigg)^m\, (2m-1)!!\\
&=\left(\frac{2\pi}{\alpha_2}\right)^{\frac{md}{2}} \left[ \frac{\Gamma(m+(\frac{\alpha_2}{\alpha_1})^{\frac{d}{4}})}{m!\Gamma((\frac{\alpha_2}{\alpha_1})^{\frac{d}{4}})} \right]^2 \bigg(\frac{d}{4} B(\frac{d}{4},\frac{d}{4})\bigg)^m\, (2m-1)!!\, t^m
\end{align*}
and
\begin{align*}
\liminf_{n\to\infty}\I^n_m
&\geq \left(\frac{2}{\alpha_2}\right)^{\frac{md}{2}} \left[ \frac{\Gamma(m+(\frac{\alpha_2}{\alpha_1})^{\frac{d}{4}})}{m!\Gamma((\frac{\alpha_2}{\alpha_1})^{\frac{d}{4}})} \right]^2 \liminf_{\gamma\to\infty}\liminf_{n\to\infty} \R^n_{m}(1/m,(\underline{b}(\gamma)\underline{\alpha}(\gamma))^{\frac{1}{2H}}/ m,1,1+m/\gamma)\\
&=\left(\frac{2}{\alpha_2}\right)^{\frac{md}{2}} \left[ \frac{\Gamma(m+(\frac{\alpha_2}{\alpha_1})^{\frac{d}{4}})}{m!\Gamma((\frac{\alpha_2}{\alpha_1})^{\frac{d}{4}})} \right]^2 \bigg(\frac{ 2t\, \pi^{\frac{d}{2}}\Gamma^2(\frac{d+4}{4})}{\Gamma(\frac{d+2}{2})}\bigg)^m\, (2m-1)!!\\
&=\left(\frac{2\pi}{\alpha_2}\right)^{\frac{md}{2}} \left[ \frac{\Gamma(m+(\frac{\alpha_2}{\alpha_1})^{\frac{d}{4}})}{m!\Gamma((\frac{\alpha_2}{\alpha_1})^{\frac{d}{4}})} \right]^2 \bigg(\frac{d}{4} B(\frac{d}{4},\frac{d}{4})\bigg)^m\, (2m-1)!!\, t^m.
\end{align*}
Therefore, 
\[
\lim_{n\to\infty}\I^n_m
=\left(\frac{2\pi}{\alpha_2}\right)^{\frac{md}{2}} \left[ \frac{\Gamma(m+(\frac{\alpha_2}{\alpha_1})^{\frac{d}{4}})}{m!\Gamma((\frac{\alpha_2}{\alpha_1})^{\frac{d}{4}})} \right]^2 \left(\frac{d}{4} B(\frac{d}{4},\frac{d}{4})\right)^m\, (2m-1)!!\, t^m.
\]
This completes the proof.
\end{proof}

\bigskip

{\medskip \noindent \textbf{Proof of Theorem \ref{thm1}.}} The desired result follows directly from Lemmas \ref{lemma3.1} and \ref{lemma3.2},  equality (\ref{G})  and Proposition \ref{prop}.\hfill  {\scriptsize $\blacksquare$} 

\bigskip

\section{Proof of Theorem \ref{thm2}}

In this section, we will prove Theorem \ref{thm2}. To make notations simpler, we will abuse some notations from Section 3. We use $F_n(t_1,t_2)$ to denote the  left-hand side of (\ref{thmrv}), i.e., 
\[
F_n(t_1,t_2):=\frac{1}{\sqrt{n}} \int^{e^{nt_1}}_0\int^{e^{nt_2}}_0 f(X_u-\widetilde{X}_v)\, du\, dv.
\]
To obtain the limiting distribution of $F_n(t_1,t_2)$, we first show that $F_n(t_1, t_2)$ has the same limiting distribution as $F_n$ defined in (\ref{fnt}) with $t=t_1\wedge t_2$, through Lemmas \ref{lem1} and \ref{lem2}. Then, we prove that the $m$-th moment of $F_n$ is asymptotically equal to $\I^n_m$ in (\ref{Inm}) by Lemma \ref{lem3}. Finally, we obtain the limit of the $m$-th moment
\begin{align*} \label{mnt}
\lim\limits_{n\to\infty}\I^n_m=\E\left[\sqrt{ D_{f,d}\, (t_1\wedge t_2)\, N^2}\, \eta\right]^m
\end{align*}
in Proposition \ref{moments} and thus complete the proof of Theorem \ref{thm2}. 

\medskip

The following result shows that the limiting distribution of $F_n(t_1,t_2)$ depends only on $t_1\wedge t_2$. 
\begin{lemma}  \label{lem1} 
\[
\lim_{n\to\infty}\E\Big[|F_n(t_1,t_2)-F_n(t_1\wedge t_2,t_1\wedge t_2)|\Big]=0.
\]
\end{lemma}
\begin{proof} This follows easily from the proof of Lemma \ref{lemma3.1}.
\end{proof}

\begin{lemma} \label{lem2}  Let 
\[
J_{n}(t)=\frac{1}{\sqrt{n}} \int_{[0,e^{nt}]^2-[1,e^{nt}]^2} f(X_u-\widetilde{X}_v)\, du\, dv.
\] 
Then
\[
\lim_{n\to\infty}\E[|J_{n}(t)|]=0.
\]
\end{lemma}
\begin{proof}  This follows easily from the proof of Lemma \ref{lemma3.2}.
\end{proof}

\medskip

By Lemmas \ref{lem1} and \ref{lem2}, we only need to consider the weak convergence of 
\begin{align} \label{fnt}
F_n=\frac{1}{\sqrt{n}}\int^{e^{nt}}_1\int^{e^{nt}}_1 f(X_u-\widetilde{X}_v)\, du\, dv,
\end{align}
for which we will compute the $m$-th moments of $F_n$ for all $m\in \mathbb N$. Throughout this section, we will fix the order  $m$ of the moment  and let $\mathscr P$ denote the set consisting of all permutations of $\{1,2, \dots, m\}$. 

Define
\begin{align} \label{Inm}
\I^n_m=\frac{m!}{n^{\frac{m}{2}}} \sum_{\sigma \in \mathscr{P}} \E\Bigg[ \int_{D^{n}_m}\int_{D^n_{m}} \prod^m_{i=1} f(X_{u_i}-\widetilde{X}_{v_{\sigma(i)}}) \, du\, dv\Bigg],
\end{align}
where
\begin{align} \label{dmn}
D^n_{m}=\Big\{ u\in [1,e^{nt}]^m: u_1<\cdots <u_m,\; u_{i+1}-u_i\geq n^{-m},\, i=1,2,\dots, m-1\Big\}.
\end{align}
The following lemma indicates that the $m$-th moment of $F_n$ is  asymptotically equal to $\I_m^n.$

\medskip

\begin{lemma} \label{lem3} 
\[
\lim_{n\to\infty} |\E[F_n^m]-\I^n_m|=0.
\]
\end{lemma}
\begin{proof}  Note that
\begin{align} \label{dm}
|\E[F_n^m]-\I^n_m|
&\leq \frac{m!}{n^{\frac{m}{2}}}\left|\sum_{\sigma \in \mathscr{P}} \E\Big[ \int_{\widetilde{D}^n_{m}\times \widetilde{D}^n_{m} -\overline{D}^n_m\times \overline{D}^n_m} \prod^m_{i=1} f(X_{u_i}-\widetilde{X}_{v_{\sigma(i)}}) \, du\, dv\Big]\right|  \nonumber\\
&\qquad\qquad+\frac{m!}{n^{\frac{m}{2}}} \left|\sum_{\sigma \in \mathscr{P}} \E\Big[ \int_{\overline{D}^n_m\times \overline{D}^n_m-D^{n}_{m}\times D^{n}_{m}} \prod^m_{i=1} f(X_{u_i}-\widetilde{X}_{v_{\sigma(i)}}) \, du\, dv\Big]\right|,
\end{align}
where  $\widetilde{D}^n_{m}=\Big\{ u\in [1,e^{nt}]^m: u_1<\cdots <u_m \Big\}$ and 
\begin{align*}
\overline{D}^n_{m}=\widetilde{D}^n_{m}\cap \Big\{u_{i+1}-u_i\geq e^{-2mnt},\, i=1,2,\dots, m-1\Big\}.
\end{align*}

Since $f$ is bounded, the first term on the right-hand side of (\ref{dm}) is less than a constant multiple of $n^{-\frac{m}{2}}$. As for the second term, using Fourier transform, we get
\begin{align}
&\left|\E\left[\int_{\overline{D}^n_m\times \overline{D}^n_m-D^{n}_{m}\times D^{n}_{m}} \prod^m_{i=1} f(X_{u_i}-\widetilde{X}_{v_{\sigma(i)}})\, du\, dv \right] \right|\nn\\
&\leq \frac{1}{(2\pi)^{md}} \int_{\R^{md}}  \int_{\overline{D}^n_m\times \overline{D}^n_m-D^{n}_{m}\times D^{n}_{m}}  \prod^m_{i=1} |\widehat{f}(x_i)|\, \exp\Big(-\frac{1}{2}\Var\big(\sum\limits^m_{i=1} x_i\cdot X_{u_i}\big)\Big)\nn\\
&\qquad\qquad\qquad\qquad\qquad\qquad\qquad\qquad\times\exp\Big(-\frac{1}{2}\Var\big(\sum\limits^m_{i=1} x_i\cdot \widetilde{X}_{v_{\sigma(i)}}\big)\Big)\, du\, dv\, dx.\label{eq-4.4}
\end{align}
For $k=1,2,\cdots,m-1$, set $$\overline{D}^n_{m,k}=\overline{D}^n_{m}\bigcap \Big\{u_{k+1}-u_{k}< n^{-m} \Big\}.$$ Then 
\[
\overline{D}^n_m\times \overline{D}^n_m-D^{n}_{m}\times D^{n}_{m}\subset\bigcup^{m-1}_{k=1} \left((\overline{D}^n_m\times \overline{D}^{n}_{m,k})\bigcup (\overline{D}^{n}_{m,k}\times \overline{D}^n_m)\right).
\]  

Applying Cauchy-Schwarz inequality to the right-hand side of \eqref{eq-4.4} on the domains $\overline{D}^n_m\times \overline{D}^{n}_{m,k}$ and $\overline{D}^{n}_{m,k}\times \overline{D}^n_m$, respectively, we have
\begin{align*}
&\int_{\R^{md}}  \int_{\overline{D}^n_m\times \overline{D}^{n}_{m,k}}  \prod^m_{i=1} |\widehat{f}(x_i)|\, \exp\Big(-\frac{1}{2}\Var\big(\sum\limits^m_{i=1} x_i\cdot X_{u_i}\big)-\frac{1}{2}\Var\big(\sum\limits^m_{i=1} x_i\cdot \widetilde{X}_{v_{\sigma(i)}}\big)\Big)  \, du\, dv\, dx\\
&\leq \left[ \int_{\R^{md}} \int_{\overline{D}^n_m\times \overline{D}^n_m} \exp\Big(-\frac{1}{2}\Var\big(\sum\limits^m_{i=1} x_i\cdot X_{u_i}\big)-\frac{1}{2}\Var\big(\sum\limits^m_{i=1} x_i\cdot X_{v_i}\big)\Big) \, du\, dv\, dx\right]^{\frac{1}{2}}\\
&\qquad \times\left[ \int_{\R^{md}} \int_{\overline{D}^{n}_{m,k}\times \overline{D}^{n}_{m,k}}  \prod^m_{i=1} \big|\widehat{f}(x_i)\big|^2\, \exp\Big(-\frac{1}{2}\Var\big(\sum\limits^m_{i=1} x_i\cdot \widetilde{X}_{u_{\sigma(i)}}\big)-\frac{1}{2}\Var\big(\sum\limits^m_{i=1} x_i\cdot \widetilde{X}_{v_{\sigma(i)}}\big)\Big) \, du\, dv\, dx\right]^{\frac{1}{2}}
\end{align*}
and 
\begin{align*}
&\int_{\R^{md}}  \int_{\overline{D}^n_{m,k}\times \overline{D}^{n}_{m}}  \prod^m_{i=1} |\widehat{f}(x_i)|\, \exp\Big(-\frac{1}{2}\Var\big(\sum\limits^m_{i=1} x_i\cdot X_{u_i}\big)-\frac{1}{2}\Var\big(\sum\limits^m_{i=1} x_i\cdot \widetilde{X}_{v_{\sigma(i)}}\big)\Big)  \, du\, dv\, dx\\
&\leq \left[ \int_{\R^{md}} \int_{\overline{D}^{n}_{m,k}\times \overline{D}^{n}_{m,k}}  \prod^m_{i=1} \big|\widehat{f}(x_i)\big|^2\, \exp\Big(-\frac{1}{2}\Var\big(\sum\limits^m_{i=1} x_i\cdot X_{u_i}\big)-\frac{1}{2}\Var\big(\sum\limits^m_{i=1} x_i\cdot X_{v_i}\big)\Big) \, du\, dv\, dx\right]^{\frac{1}{2}} \\
&\qquad \times \left[ \int_{\R^{md}} \int_{\overline{D}^n_m\times \overline{D}^n_m} \exp\Big(-\frac{1}{2}\Var\big(\sum\limits^m_{i=1} x_i\cdot \widetilde{X}_{u_i}\big)-\frac{1}{2}\Var\big(\sum\limits^m_{i=1} x_i\cdot \widetilde{X}_{v_i}\big)\Big) \, du\, dv\, dx\right]^{\frac{1}{2}}.
\end{align*}
By Fubini's theorem and Assumption (B), 
\begin{align*}
&\int_{\R^{md}} \int_{\overline{D}^n_m\times \overline{D}^n_m} \exp\Big(-\frac{1}{2}\Var\big(\sum\limits^m_{i=1} x_i\cdot X_{u_i}\big)-\frac{1}{2}\Var\big(\sum\limits^m_{i=1} x_i\cdot X_{v_i}\big)\Big) \, du\, dv\, dx\\
&=\int_{\overline{D}^n_m\times \overline{D}^n_m} \int_{\R^{md}} \exp\Big(-\frac{1}{2}\Var\big(\sum\limits^m_{i=1} x_i\cdot (X_{u_i}-\widetilde{X}_{v_i})\big)\Big)\, dx \, du\, dv\\
&\leq c_1 \int_{\overline{D}^n_m \times \overline{D}^n_m}  \prod^m_{i=1} [(u_i-u_{i-1})^{2H}+(v_i-v_{i-1})^{2H}]^{-\frac{d}{2}} dudv\\
&\leq c_2 \left(\int^{e^{nt}}_{e^{-2mnt}}\int^{e^{nt}}_{e^{-2mnt}} (r^{2H}+s^{2H})^{-\frac{d}{2}} dr\, ds \right)^m\\
&\leq c_3\, n^m,
\end{align*}
where the last inequality follows from Lemma \ref{lemma1}. 

On the other hand,
\begin{align*}
&\int_{\R^{md}} \int_{\overline{D}^{n}_{m,k}\times \overline{D}^{n}_{m,k}}  \prod^m_{i=1} \big|\widehat{f}(x_i)\big|^2\, \exp\Big(-\frac{1}{2}\Var\big(\sum\limits^m_{i=1} x_i\cdot \widetilde{X}_{u_{\sigma(i)}}\big)-\frac{1}{2}\Var\big(\sum\limits^m_{i=1} x_i\cdot \widetilde{X}_{v_{\sigma(i)}}\big)\Big) \, du\, dv\, dx\\
=&(2\pi)^{md}\, \E\left[\int_{\overline{D}^{n}_{m,k}\times \overline{D}^{n}_{m,k}}\prod^m_{i=1} U_f(X_{u_{\sigma(i)}}-\widetilde{X}_{v_{\sigma(i)}}) \, du\, dv\right],
\end{align*}
where $U_f$ is the inverse Fourier transform of $ \big|\widehat{f}\, \big|^2$. 

Therefore, combining all the inequalities/equality after \eqref{eq-4.4}, we have, 
\begin{align} \label{eq-4.4.0}
&\frac{m!}{n^{\frac{m}{2}}} \left| \E\left[\int_{\overline{D}^n_m\times \overline{D}^n_m-D^{n}_{m}\times D^{n}_{m}} \prod^m_{i=1} f(X_{u_i}-\widetilde{X}_{v_{\sigma(i)}})\, du\, dv \right] \right|\nn\\
&\qquad \leq  c_4\,\sum^{m-1}_{k=1} \left[\E\left[\int_{\overline{D}^{n}_{m,k}\times \overline{D}^{n}_{m,k}}\prod^m_{i=1} U_f(X_{u_i}-\widetilde{X}_{v_i}) \, du\, dv\right]\right]^{\frac{1}{2}}\nn\\
&\qquad\qquad \leq  c_5\,n^{-m}\left[\E\left[\int_{\overline{D}^{n}_{m-1}\times \overline{D}^{n}_{m-1}}\prod^{m-1}_{i=1} U_f(X_{u_i}-\widetilde{X}_{v_i}) \, du\, dv\right]\right]^{\frac{1}{2}},
\end{align}
where the last inequality follows from the boundedness of $U_f$ and the definition of $\overline{D}^n_{m,k}$.

Using Fourier transform, the boundedness of $|\widehat{U_f}|$, Assumption (B) and Lemma \ref{lemma1},  we get that the right-hand side of (\ref{eq-4.4.0}) is less than
\begin{align*}
& c_6\,n^{-m} \left[\int_{\overline{D}^n_{m-1}\times \overline{D}^n_{m-1}} \int_{\R^{(m-1)d}} \exp\Big(-\frac{1}{2}\Var\big(\sum\limits^{m-1}_{i=1} x_i\cdot (X_{u_i}-\widetilde{X}_{v_i})\big)\Big)\, dx \, du\, dv\right]^{\frac{1}{2}} \\
 &\leq c_7\, n^{-\frac{m+1}{2}}.
\end{align*}
This gives the desired result.
\end{proof}

\medskip

 Now we  represent $\I_m^n$ given  in \eqref{Inm} using Fourier transform.  For $t\ge 0 $ and $\sigma\in\mathscr{P}$, set
 \[
I_{nt}(x)= \int_{D^{n}_{m} } \exp\Big( -\frac 12 \mathrm{Var} \big( \sum_{i=1}^{m} x_i \cdot X_{u_i}  \big)\Big)\, du
 \]
and  
\[
I_{nt}^\sigma(x)= \int_{D^{n}_{m}} \exp \Big(-\frac 12 \mathrm{Var} \big( \sum_{i=1}^{m} x_i \cdot \widetilde{X}_{v_{\sigma(i)}}  \big)\Big)\, dv,
\]
where $D_m^n$ is defined in \eqref{dmn}. Then by Fourier transform, 
\begin{equation} \label{w2}
\I^n_m=  \frac{m!}{((2\pi)^d\sqrt{n})^{m}}  \sum_{\sigma \in \mathscr{P}}  \int_{\R^{md}}  \prod _{i=1}^m  \widehat{f} (x_i) \,I_{nt}(x)\,  I_{nt}^\sigma(x)\, dx.
\end{equation}

By the preceding lemmas in this section, to prove Theorem \ref{thm2},  it suffices to compute $\lim\limits_{n\to\infty} \I^n_m$. To do this, we will use  Assumption (B) and adapt the chaining argument from  \cite{nx1}  to obtain some estimates in Lemma \ref{chain}, which is crucial to the calculation of  $\lim\limits_{n\to\infty} \I^n_m$ in Proposition \ref{moments}. For better readability, we split the rest of this section into four parts. 
\bigskip

\noindent
{\bf (I)}  {\bf Symmetrization of $|\I_m^n|$ via Cauchy-Schwarz inequality.}  In this part, we will obtain an upper bound for $|\I_m^n|$, see \eqref{cs}. To this goal,  we will first apply the  Cauchy-Schwarz inequality to the integral in (\ref{w2}) and then use Assumption (B) for the variance. Note that this kind of procedure will be used frequently  for similar integrals in the sequel.

For the integral on the right-hand side of \eqref{w2}, the Cauchy-Schwarz inequality yields
 \begin{align*}   \notag
&\bigg|  \int_{\R^{md}}  \prod _{i=1}^m  \widehat{f} (x_i) \,
I_{nt}(x)  I_{nt}^\sigma(x)\, dx  \bigg|  \\
&\qquad\qquad \le 
\bigg[\int_{\R^{md}}  \prod _{i=1}^m  |\widehat{f} (x_i)  |  
\big(I_{nt}(x)\big)^2\,  dx     \bigg]^{\frac 12}  
 \bigg[\int_{\R^{md}}   \prod _{i=1}^m  |\widehat{f} (x_i) | 
\big(I_{nt}^\sigma(x) \big)^2\,  dx \bigg]^{\frac 12}.
\end{align*}
 Taking into account that $\prod _{i=1}^m| \widehat{f} (x_i) |  $ is symmetric  in terms of  $x_i$s,  the second factor on the right-hand side of the above inequality does not depend on $\sigma$ and hence
 \[
 \bigg|  \int_{\R^{md}}   \prod _{i=1}^m  \widehat{f} (x_i)  
I_{nt}(x)  I_{nt}^\sigma(x)\, dx    \bigg| 
\le
 \int_{\R^{md}}  \prod _{i=1}^m  |\widehat{f} (x_i)  |  
\big(I_{nt }(x)\big)^2\, dx.
\]
Substituting this estimate into  (\ref{w2}) yields
\[
|\I^n_m| \le   \frac{(m!)^2}{((2\pi)^d\sqrt{n})^{m}}  \int_{\R^{md}}   \prod _{i=1}^m | \widehat{f} (x_i) |   
\big(I_{nt}(x) \big)^2\,  dx.
\]
Making the change of variables $y_i =\sum\limits_{j=i}^m x_j$ (with the convention $y_{m+1} =0$), we can write
 \begin{align*}
|\I^n_m| 
&\le  \frac{(m!)^2}{(2\pi)^{md}}\,  n^{-\frac{m}{2}}    \int_{\R^{md}} \int_{D^n_m \times D^n_m} 
  \prod _{i=1}^m| \widehat{f} (y_i-y_{i+1} ) |   \\
& \qquad  \times 
\exp\bigg( -\frac 12 \mathrm{Var} \Big( \sum_{i=1}^{m} y_i \cdot (X_{u_i} -X_{u_{i-1}}) \Big)  -\frac 12 \mathrm{Var} \Big( \sum_{i=1}^{m} y_i \cdot (\widetilde{X}_{v_i} -\widetilde{X}_{v_{i-1}}) \Big) \bigg) \,du\, dv\, dy.
\end{align*}
Applying Assumption (B) and making the change of variables $s_1 =u_1$, $r_1=v_1$, $s_i =u_i-u_{i-1}$,  and  $r_i =v_i- v_{i-1}$, for $2\le i  \le m$, we obtain
 \begin{align} \label{cs}
 |\I^n_m| 
 &\le  \frac{(m!)^2}{(2\pi)^{md}}\, n^{-\frac{m}{2}}    \int_{\R^{md}} \int_{[n^{-m},e^{nt}]^{2m}} 
 \prod _{i=1}^m| \widehat{f} (y_i-y_{i+1} ) |  \nonumber\\
& \qquad\qquad\qquad \times 
\exp\bigg( -\frac {\kappa} 2   \sum_{i=1}^{m} |y_i|^{2} (s_i^{2H} + r_i^{2H})\bigg) \,ds\, dr\, dy.
\end{align}

\noindent
{\bf (II)}   {\bf Chaining argument.} 
In this part, we apply the chaining argument introduced in \cite{nx1} to the integral on the right-hand side of \eqref{cs}. The main idea is to replace each product $\widehat{f}(y_{2i-1}-y_{2i})\widehat{f}(y_{2i}-y_{2i+1})$ by   
$\widehat{f}(-y_{2i})\widehat{f}(y_{2i} )= |\widehat{f}(y_{2i})|^2 $, noting that, by the assumption $\int_{\R^d}|f(x)||x|^{\beta}\, dx<\infty$ for some $\beta>0$,  the differences  $\widehat{f}(y_{2i-1}-y_{2i})- \widehat{f}(-y_{2i})$ and
$\widehat{f}(y_{2i}-y_{2i+1})- \widehat{f}(y_{2i})$ are bounded by constant multiples of $|y_{2i-1}|^\alpha$ and $|y_{2i+1}|^\alpha$, respectively, for any
$\alpha\in[0,\beta]$.   Making these substitutions for $\prod _{i=1}^m| \widehat{f} (y_i-y_{i+1})$  recursively, we get
 \begin{align*}
\prod _{i=1}^m| \widehat{f} (y_i-y_{i+1} ) | & =  | \widehat{f} (y_1-y_{2} )  -\widehat{f} (-y_{2} ) +
  \widehat{f} (-y_{2}) | | \widehat{f} (y_2-y_{3} )  -\widehat{f} (y_{2 }) +
  \widehat{f} (y_{2}) |  \\
  & \times | \widehat{f} (y_3-y_{4} )  -\widehat{f} (-y_{4 }) +
 \widehat{f} (-y_{4}) | | \widehat{f} (y_4-y_{5} )  -\widehat{f} (y_{4 }) +
\widehat{f} (y_{4}) |  \times \cdots \\
  & = \prod _{i=1}^m \Big| \widehat{f} (y_i-y_{i+1} )  -\widehat{f} \big((-1)^iy_{2 \lfloor \frac {i+1}2 \rfloor }\big) +
  \widehat{f} \big((-1)^i y_{2 \lfloor \frac {i+1}2 \rfloor } \big) \Big|,
  \end{align*}
  where $ \lfloor \frac {i+1}2 \rfloor  $ denotes the integer part of $\frac {i+1}2$ and $y_{m+1}=0$ by convention. 
  
  Noting that $|\widehat{f}(y)|=|\widehat{f}(-y)|$, we have
  \[
    \prod _{i=1}^m| \widehat{f} (y_i-y_{i+1} ) |
   \le \sum_{k=1}^m I_k,
   \]
  where
  \[
  I_k=\Big( \prod _{j=1}^{k-1} \big| \widehat{f}(y_{2 \lfloor \frac {j+1}2 \rfloor }) \big| \Big)~
  \Big| \widehat{f} (y_k-y_{k+1} )  -\widehat{f}(y_{2 \lfloor \frac {k+1}2 \rfloor })\Big| ~ \prod_{j=k+1}^m \big| \widehat{f} (y_j-y_{j+1} )\big|
  \]
    for $k=1, 2, \dots, m-1$,
  and
  \[
  I_m= \Big( \prod _{j=1}^{m-1} \big| \widehat{f}(y_{2 \lfloor \frac {j+1}2 \rfloor }) \big| \Big) \big|\widehat{f} (y_m)\big|.
  \]
 In this way, by \eqref{cs} we obtain the decomposition
\begin{align*} \label{eq3}
 \big| \I^n_m \big|\le \frac{(m!)^2}{(2\pi)^{md}}\,    \sum_{k=1}^m A_{k,m},
 \end{align*}
where
\[
A_{k,m}=n^{-\frac{m}{2}}  \int_{\R^{md}} \int_{[n^{-m},e^{nt}]^{2m}} 
I_k\, \exp\Big( -\frac {\kappa} 2   \sum_{i=1}^{m} |y_i|^{2}( s_i^{2H}+r_i^{2H}) \Big) \, ds\,dr\, dy.
\]

\noindent
{\bf (III)}   {\bf Some crucial estimates.}  We fix a  constant $\lambda\in (0, 1/2)$. The estimation of each term $A_{k,m}$, for $k=1, \dots, m$, is given below.
\begin{lemma} \label{chain}  There exists a positive constant $c$ such that 
\begin{enumerate}
\item[(i)] $A_{k,m}\leq c\, n^{-\lambda}$, for $k=1,2,\dots,m-1$,
\item[(ii)] $A_{m,m}\leq c\, n^{-\frac{1}{2}}$  if $m$ is odd,  and  $A_{m,m}\leq c$ if $m$ is even.
\end{enumerate}
\end{lemma} 

\begin{proof}  To prove part (i), we first consider the case when $k$ is odd. By the assumption on $f$, we can obtain $|\widehat{f}(y)|\leq c_{\alpha}(|y|^{\alpha}\wedge 1)$ for any $\alpha\in[0,\beta]$. So $A_{k,m} $ is less than a constant multiple of
\begin{align*}
&n^{-\frac{m}{2}}\int_{\R^{md}}\int_{ [n^{-m}, e^{nt}]^{2m}} |y_{k}|^{\alpha} \prod^{\lfloor\frac{m}{2}\rfloor}_{j=\frac{k+1}{2}}  (|y_{2j}|^{\alpha}+|y_{2j+1}|^{\alpha})   \prod^{\frac{k-1}{2}}_{j=1}|\widehat{f}(y_{2j})|^2   \\
&\qquad\qquad\qquad \qquad \qquad \times    \exp\Big(-\frac{\kappa}{2} \sum\limits^m_{i=1} |y_i|^2(s^{2H}_i+r_i^{2H})\Big)\, ds\, dr\, dy.
 \end{align*}

Integrating with respect to the $y_i$, $s_i$ and $r_i$ for  $i\leq k-1$ gives, by Lemmas \ref{lemma1} and \ref{lemma3},
\begin{align*}
A_{k,m}
&\leq c_1\, n^{-\frac{m-(k-1)}{2}}\int_{\R^{(m-k+1)d}}\int_{[n^{-m}, e^{nt}]^{2(m-k+1)}}  |y_k|^{\alpha}\prod^{\lfloor\frac{m}{2}\rfloor}_{j=\frac{k+1}{2}}  (|y_{2j}|^{\alpha}+|y_{2j+1}|^{\alpha})\\
&\qquad\qquad\times \exp\Big(-\frac{\kappa}{2}\sum\limits^m_{i=k} |y_i|^2\, (s^{2H}_i+r^{2H}_i)\Big)\, d\overline{s}\, d\overline{r}\, d\overline{y},
\end{align*}
where $d\overline{s}=ds_k\cdots ds_{m}$,  $d\overline{r}=dr_k\cdots dr_{m}$ and $d\overline{y}=dy_k\cdots dy_{m}$. 

By Lemmas \ref{lemma1} and \ref{lemma2},
\begin{align*}
A_{k,m}
&\leq c_2\, n^{-\frac{m-k+1}{2}+(\lfloor\frac{m-k+1}{2}\rfloor+1)(mH\alpha)+(m-k-\lfloor\frac{m-k+1}{2}\rfloor)}\\
&= c_2\, n^{\frac{m}{2}-\lfloor\frac{m}{2}\rfloor-1+(\lfloor\frac{m-k+1}{2}\rfloor+1)(mH\alpha)}.
\end{align*}
Choosing $\alpha$ small enough such that 
\[
\frac{m}{2}-\lfloor\frac{m}{2}\rfloor-1+(\lfloor\frac{m-k+1}{2}\rfloor+1)(mH\alpha)=-\lambda
\] 
gives
\begin{equation} \label{chaino}
A_{k,m}\leq c_2\, n^{-\lambda}.
\end{equation}

We next consider the case when $k$ is even. By Assumption (B), $A_{k,m}$ is less than a constant multiple of
\begin{align*}
&n^{-\frac{m}{2}}\int_{\R^{md}}\int_{[n^{-m}, e^{nt}]^{2m}} \big|\widehat{f}(-y_k)\big|\big|\widehat{f}(y_k-y_{k+1})-\widehat{f}(y_k)\big| \prod^m_{i=k+1}  |\widehat{f}(y_i-y_{i+1})|  \\
&\qquad\qquad\qquad \times  \prod^{\frac{k-2}{2}}_{j=1}|\widehat{f}(y_{2j})|^2 \,\exp\Big(-\frac{\kappa}{2}\sum\limits^m_{i=1} |y_i|^2\, (s^{2H}_i+r^{2H}_i)\Big)\, ds\, dr\, dy.
\end{align*}
Using similar arguments as in the odd case,
\begin{align*}
A_{k,m}
&\leq c_3\, n^{-\frac{m}{2}}\int_{\R^{md}}\int_{[n^{-m}, e^{nt}]^{2m}} |y_k|^{\alpha}|y_{k+1}|^{\alpha} \prod^{\lfloor\frac{m}{2}\rfloor}_{j=\frac{k+2}{2}}  (|y_{2j}|^{\alpha}+|y_{2j+1}|^{\alpha}) \\
&\qquad\qquad\qquad \times \Big(\prod^{\frac{k-2}{2}}_{j=1}|\widehat{f}(y_{2j})|^2\Big)\,\exp\Big(-\frac{\kappa}{2}\sum\limits^m_{i=1} |y_i|^2\, (s^{2H}_i+r^{2H}_i)\Big)\, ds\, dr\, dy\\
&\leq c_4\, n^{-\frac{m-(k-2)}{2}+(\lfloor\frac{m-k}{2}\rfloor+2)(mH\alpha)+(m-k-\lfloor\frac{m-k}{2}\rfloor)}\\
&= c_4\, n^{\frac{m}{2}-\lfloor\frac{m}{2}\rfloor-1+(\lfloor\frac{m-k}{2}\rfloor+2)(mH\alpha)}.
\end{align*}
Choosing $\alpha$ small enough such that 
\[
\frac{m}{2}-\lfloor\frac{m}{2}\rfloor-1+(\lfloor\frac{m-k}{2}\rfloor+2)(mH\alpha)=-\lambda
\] 
gives
\begin{equation} \label{chaine}
A_{k,m}\leq c_4\, n^{-\lambda}.
\end{equation}
Combining (\ref{chaino}) and (\ref{chaine}) gives the desired estimates in part (i).  

Finally, we show part (ii). If $m$ is odd, by Lemmas \ref{lemma1} and \ref{lemma3},
\begin{align*}
A_{m,m}
&=n^{-\frac{m}{2}}\int_{\R^{md}}\int_{[n^{-m}, e^{nt}]^{2m}} |\widehat{f}(y_m)|\prod^{\frac{m-1}{2}}_{j=1}|\widehat{f}(y_{2j})|^2 \,\exp\Big(-\frac{\kappa}{2}\sum\limits^m_{i=1} |y_i|^2\, (s^{2H}_i+r^{2H}_i)\Big)\, ds\, dr\, dy\\
&\leq c_5\, n^{-\frac{1}{2}}\int_{\R^{d}}\int_{[n^{-m}, e^{nt}]^{2}} |\widehat{f}(y_m)|\,\exp\Big(-\frac{\kappa}{2}|y_m|^2\, (s^{2H}_m+r^{2H}_m)\Big)\, ds_m\, dr_m\, dy_m\\
&\leq c_6\, n^{-\frac{1}{2}}\int_{\R^{d}} |\widehat{f}(y_m)||y_m|^{-d}\, dy_m\\
&\leq c_7\, n^{-\frac{1}{2}},
\end{align*}
where the last second inequality follows from Lemma \ref{lemma3} and the last inequality can follow easily from Remark \ref{remark1}.

If $m$ is even, then by Lemma \ref{lemma3},
\begin{align*}
A_{m,m}
&=n^{-\frac{m}{2}}\int_{\R^{md}}\int_{[n^{-m}, e^{nt}]^{2m}} \prod^{\frac{m}{2}}_{j=1}|\widehat{f}(y_{2j})|^2 \,\exp\Big(-\frac{\kappa}{2}\sum\limits^m_{i=1} |y_i|^2\, (s^{2H}_i+r^{2H}_i)\Big)\, ds\, dr\, dy\\
&\leq c_8.
\end{align*}
This completes the proof.
\end{proof}

\medskip
 
\noindent
{\bf (IV)} {\bf Convergence of moments.}  In this final part, we show the convergence of $\I_m^n$ given in \eqref{w2}, and then prove Theorem \ref{thm2}.  Recall that 
\begin{align*}
\I^n_m
&=\frac{m!}{((2\pi)^d\sqrt{n})^m}\sum_{\sigma\in\mathscr{P}} \int_{D^n_m\times D^n_m} \int_{\R^{md}} \prod^m_{i=1}\widehat{f}(x_i)\\
&\qquad\qquad\qquad\times \exp\Big(-\frac{1}{2}\Var\big(\sum\limits^m_{i=1} x_i\cdot (X_{u_i}-\widetilde{X}_{v_{\sigma(i)}})\big)\Big)\, dx\, du\, dv.
\end{align*}

\begin{proposition} \label{moments} If $m$ is odd, then $\lim\limits_{n\to\infty} \I^n_m=0$. If $m$ is even, then
\[
\lim_{n\to\infty} \I^n_m=\left[ \frac{\Gamma(\frac{m}{2}+(\frac{\alpha_2}{\alpha_1})^{\frac{d}{4}})}{(\frac{m}{2})!\Gamma((\frac{\alpha_2}{\alpha_1})^{\frac{d}{4}})} \right]^2\, (D_{f,d} \, t)^{m/2} ((m-1)!!)^2.
\]
\end{proposition}

\begin{proof} The convergence of odd moments follows easily from Lemma  \ref{chain}. So we only need to show the convergence of even moments, which will be done in five steps.

\noindent \textbf{Step 1.} We show that $\I^n_m$ is asymptotically equal to $\widetilde{\I}^n_{m,\gamma}$ defined in (\ref{gamma}).  Let 
\[
\widetilde{O}_{m}=D^n_m\cap \Big\{n^2<\Delta u_{2i-1}<e^{nt}/m,\, n^{-1}<\Delta u_{2i}<n,\, i=1,2,\cdots m/2 \Big\}, 
\]
where $\Delta u_{k}=u_{k}-u_{k-1}$ for $k=1,2,\cdots, m$ with the convention $u_0=0$.

Set
\begin{align*}
\widetilde{\I}^n_{m}
&=\frac{m!}{((2\pi)^d\sqrt{n})^m}\sum_{\sigma\in\mathscr{P}}   \int_{\widetilde{O}_{m}\times \widetilde{O}_{m}} \int_{\R^{md}} \prod^m_{i=1}\widehat{f}(x_i)  \nonumber  \\  
&\qquad\qquad\qquad \times \exp\Big(-\frac{1}{2}\Var\big(\sum\limits^m_{i=1} x_i\cdot (X_{u_i}-\widetilde{X}_{v_{\sigma(i)}})\big)\Big)\, dx\, du\, dv.
\end{align*}
Then, using similar arguments as in {\bf (I)}  {\bf Symmetrization of $|\I_m^n|$ via Cauchy-Schwarz inequality},
\begin{align*} 
|\I^n_m-\widetilde{\I}^n_{m}|
&\leq c_1\, n^{-\frac{m}{2}}\sum_{\sigma\in\mathscr{P}} \int_{\R^{md}} \int_{D^n_m \times D^n_m -\widetilde{O}_{m}\times \widetilde{O}_{m}}  \prod^m_{i=1}|\widehat{f}(x_i)| \\
&\qquad\qquad\qquad \times \exp\Big(-\frac{1}{2}\Var\big(\sum\limits^m_{i=1} x_i\cdot (X_{u_i}-\widetilde{X}_{v_{\sigma(i)}})\big)\Big)\, du\, dv\, dx \\
&\leq c_2\, n^{-\frac{m}{2}} \int_{\R^{md}}  \int_{D^n_m \times D^n_m -\widetilde{O}_{m}\times \widetilde{O}_{m}}  \prod^m_{i=1}|\widehat{f}(y_i-y_{i+1})|  \\  
&\qquad\qquad\qquad \times \exp\Big(-\frac{\kappa}{2}\sum\limits^m_{i=1} |y_i|^2\big[(\Delta u_i)^{2H}+(\Delta v_i)^{2H}\big]\Big)\, du\, dv\, dy.
\end{align*}
Thanks to Lemma \ref{chain}, 
\begin{align} \label{limsup1}
\limsup\limits_{n\to\infty}|\I^n_m-\widetilde{\I}^n_{m}|
&\leq c_3 \limsup\limits_{n\to\infty}   n^{-\frac{m}{2}}  \int_{D^n_m \times D^n_m -\widetilde{O}_{m}\times \widetilde{O}_{m}} \int_{\R^{md}} \prod^{m/2}_{k=1}|\widehat{f}(y_{2k})|^2  \nonumber  \\  
&\qquad\qquad\qquad \times \exp\Big(-\frac{\kappa}{2}\sum\limits^m_{i=1} |y_i|^2\big[(\Delta u_i)^{2H}+(\Delta v_i)^{2H}\big]\Big)\, dy\, du\, dv.
\end{align}

For $\ell=1,2,\cdots, m$, define
\[
E^n_{m,\ell}= \left\{
        \begin{array}{ll}
           D^n_m\cap \{n^{-m}\leq \Delta u_{\ell}\leq n^2\; \text{ or }\; e^{nt}/m\leq \Delta u_{\ell}\leq e^{nt}\}, & \text{if}\; \ell\; \text{is odd}; \\ 
           \\
          D^n_m\cap \{n^{-m}\leq \Delta u_{\ell}\leq n^{-1}\; \text{ or }\; n\leq \Delta u_{\ell}\leq e^{nt}\}, & \text{otherwise}.
        \end{array}
    \right.
    \]
Then, by 
\[
D^n_m -\widetilde{O}_{m}=\bigcup^m_{\ell=1}E^n_{m,\ell}
\]
and the symmetry of $u$ and $v$ in inequality (\ref{limsup1}), 
\begin{align*}
&\limsup\limits_{n\to\infty}|\I^n_m-\widetilde{\I}^n_{m}|\\
\leq &2c_3\sum^{m}_{\ell=1}\limsup\limits_{n\to\infty}   n^{-\frac{m}{2}}  \int_{E^n_{m,\ell} \times D^n_m} \int_{\R^{md}} \prod^{m/2}_{k=1}|\widehat{f}(y_{2k})|^2 \exp\Big(-\frac{\kappa}{2}\sum\limits^m_{i=1} |y_i|^2\big[(\Delta u_i)^{2H}+(\Delta v_i)^{2H}\big]\Big)\, dy\, du\, dv\\
\leq &c_4\sum^{m}_{\ell=1}\limsup\limits_{n\to\infty} \bigg[  n^{-\frac{m}{2}}  \int_{E^n_{m,\ell} \times E^n_{m,\ell}} \int_{\R^{md}} \prod^{m/2}_{k=1}|\widehat{f}(y_{2k})|^2 \exp\Big(-\frac{\kappa}{2}\sum\limits^m_{i=1} |y_i|^2\big[(\Delta u_i)^{2H}+(\Delta v_i)^{2H}\big]\Big)\, dy\, du\, dv\bigg]^{\frac{1}{2}},
\end{align*}
where in the last inequality we use the arguments as in {\bf (I)}  {\bf Symmetrization of $|\I_m^n|$ via Cauchy-Schwarz inequality} and Lemma \ref{chain}.

When $\ell$ is odd, integrating with respect to the $y_i$, $u_i$ and $v_i$ for all $i\neq \ell$ and using Lemmas \ref{lemma1} and \ref{lemma3}, 
\begin{align*}
&\limsup\limits_{n\to\infty} n^{-\frac{m}{2}}  \int_{E^n_{m,\ell} \times E^n_{m,\ell}} \int_{\R^{md}} \prod^{m/2}_{k=1}|\widehat{f}(y_{2k})|^2\exp\Big(-\frac{\kappa}{2}\sum\limits^m_{i=1} |y_i|^2\big[(\Delta u_i)^{2H}+(\Delta v_i)^{2H}\big]\Big)\, dy\, du\, dv\\
&\leq \limsup\limits_{n\to\infty} \frac{c_5}{n}\int_{\R^{d}}  \left(  \left(\int^{n^2}_{n^{-m}}+\int^{e^{nt}}_{e^{nt}/m}\right)\exp\Big(-\frac{\kappa}{2} |y_{\ell}|^2 (\Delta u_\ell)^{2H}\Big) d\Delta u_\ell \right)^2\, dy_\ell\\
&\leq \limsup\limits_{n\to\infty} \frac{c_6}{n} \int^{n^2}_{n^{-m}}\int^{n^2}_{n^{-m}} \int_{\R^{d}} \exp\Big(-\frac{\kappa}{2} |y_{\ell}|^2 [(\Delta u_\ell)^{2H}+(\Delta v_\ell)^{2H}]\Big) dy_\ell \, d\Delta u_\ell d\Delta v_\ell \\
&\qquad+\limsup\limits_{n\to\infty} \frac{c_6}{n} \int^{e^{nt}}_{e^{nt}/m}\int^{e^{nt}}_{e^{nt}/m} \int_{\R^{d}}\exp\Big(-\frac{\kappa}{2} |y_{\ell}|^2 [(\Delta u_\ell)^{2H}+(\Delta v_\ell)^{2H}]\Big)dy_\ell \, d\Delta u_\ell d\Delta v_\ell\\
&\leq \limsup\limits_{n\to\infty} \frac{c_7}{n} \int^{n^2}_{n^{-m}}\int^{n^2}_{n^{-m}} [(\Delta u_\ell)^{2H}+(\Delta v_\ell)^{2H}]^{-\frac{d}{2}}\, d\Delta u_\ell d\Delta v_\ell \\
&\qquad+\limsup\limits_{n\to\infty} \frac{c_7}{n} \int^{e^{nt}}_{e^{nt}/m}\int^{e^{nt}}_{e^{nt}/m} [(\Delta u_\ell)^{2H}+(\Delta v_\ell)^{2H}]^{-\frac{d}{2}}\, d\Delta u_\ell d\Delta v_\ell\\
&=0,
\end{align*}
where in the last equality we use similar arguments as in the proof of Lemma \ref{lemma1}.  

Similarly, when $\ell$ is even, by Lemmas \ref{lemma1} and \ref{lemma3},
\begin{align*}
&\limsup\limits_{n\to\infty} n^{-\frac{m}{2}}  \int_{E^n_{m,\ell} \times E^n_{m,\ell}} \int_{\R^{md}} \prod^{m/2}_{k=1}|\widehat{f}(y_{2k})|^2\exp\Big(-\frac{\kappa}{2}\sum\limits^m_{i=1} |y_i|^2\big[(\Delta u_i)^{2H}+(\Delta v_i)^{2H}\big]\Big)\, dy\, du\, dv\\
&\leq c_8 \limsup\limits_{n\to\infty}   \int_{\R^{d}}  |\widehat{f}(y_{\ell})|^2 \left(  \left(\int^{n^{-1}}_{n^{-m}}+\int^{e^{nt}}_{e^{nt}/m}\right)  \exp\Big(-\frac{\kappa}{2} |y_{\ell}|^2 (\Delta u_\ell)^{2H}\Big) d\Delta u_\ell \right)^2\, dy_\ell\\
&\leq c_9 \limsup\limits_{n\to\infty}  \int^{n^{-1}}_{n^{-m}}\int^{n^{-1}}_{n^{-m}} \int_{\R^{d}} |\widehat{f}(y_{\ell})|^2 \exp\Big(-\frac{\kappa}{2} |y_{\ell}|^2 [(\Delta u_\ell)^{2H}+(\Delta v_\ell)^{2H}]\Big) dy_\ell \, d\Delta u_\ell d\Delta v_\ell \\
&\qquad+c_9 \limsup\limits_{n\to\infty} \int^{e^{nt}}_{n}\int^{e^{nt}}_{n} \int_{\R^{d}} |\widehat{f}(y_{\ell})|^2 \exp\Big(-\frac{\kappa}{2} |y_{\ell}|^2 [(\Delta u_\ell)^{2H}+(\Delta v_\ell)^{2H}]\Big)dy_\ell \, d\Delta u_\ell d\Delta v_\ell\\
&\leq c_{10} \limsup\limits_{n\to\infty} \int^{e^{nt}}_{n}\int^{e^{nt}}_{n} \int_{\R^{d}} |y_{\ell}|^{2\alpha} \exp\Big(-\frac{\kappa}{2} |y_{\ell}|^2 [(\Delta u_\ell)^{2H}+(\Delta v_\ell)^{2H}]\Big)dy_\ell \, d\Delta u_\ell d\Delta v_\ell\\
&\leq c_{11}\limsup\limits_{n\to\infty} \int^{e^{nt}}_{n}\int^{e^{nt}}_{n} [(\Delta u_\ell)^{2H}+(\Delta v_\ell)^{2H}]^{-\frac{d}{2}-\alpha}\, d\Delta u_\ell d\Delta v_\ell\\
&=0,
\end{align*}
where we use Remark \ref{remark1} to get the third inequality and the proof of Lemma \ref{lemma1} in the last equality.  

Therefore, $\limsup\limits_{n\to\infty}|\I^n_m-\widetilde{\I}^n_{m}|=0$. Now for any $\gamma>1$, define
\begin{align} \label{gamma}
\widetilde{\I}^n_{m,\gamma}
&=\frac{m!}{((2\pi)^d\sqrt{n})^m}  \sum_{\sigma\in\mathscr{P}}  \int_{\widetilde{O}^{\gamma}_{m} \times \widetilde{O}^{\gamma}_{m} }  \int_{\R^{md}} \prod^m_{i=1}\widehat{f}(x_i)  \nonumber  \\  
&\qquad\qquad\qquad \times \exp\Big(-\frac{1}{2}\Var\big(\sum\limits^m_{i=1} x_i\cdot (X_{u_i}-\widetilde{X}_{v_{\sigma(i)}})\big)\Big)\, dx\, du\, dv,
\end{align}
where
\[
\widetilde{O}^{\gamma}_{m}=\widetilde{O}_{m}\cap \Big\{\frac{\Delta u_{2i-1}}{\Delta u_{2j-1}}>\gamma\; \text{ or }\; \frac{\Delta u_{2j-1}}{\Delta u_{2i-1}}>\gamma\; \text{for all}\, i,j\in\{1,2,\cdots,m/2\}\; \text{with}\, i\neq j \Big\}.
\]
Then, using similar arguments as in {\bf (I)}  {\bf Symmetrization of $|\I_m^n|$ via Cauchy-Schwarz inequality},
\begin{align*}
|\widetilde{\I}^n_{m}-\widetilde{\I}^n_{m,\gamma}|
&\leq c_{12}\, n^{-\frac{m}{2}}  \sum_{\sigma\in\mathscr{P}} \int_{\widetilde{O}_{m}\times \widetilde{O}_{m}-\widetilde{O}^{\gamma}_{m}\times \widetilde{O}^{\gamma}_{m}} \int_{\R^{md}} \prod^m_{i=1}|\widehat{f}(x_i)|  \nonumber  \\  
&\qquad\qquad\qquad \times \exp\Big(-\frac{1}{2}\Var\big(\sum\limits^m_{i=1} x_i\cdot (X_{u_i}-\widetilde{X}_{v_{\sigma(i)}})\big)\Big)\, dx\, du\, dv\\
&\leq c_{13}\,  n^{-\frac{m}{2}}   \int_{\widetilde{O}_{m}\times \widetilde{O}_{m}-\widetilde{O}^{\gamma}_{m}\times \widetilde{O}^{\gamma}_{m}} \int_{\R^{md}} \prod^m_{i=1}|\widehat{f}(y_i-y_{i+1})|  \nonumber\\  
&\qquad\qquad\qquad \times \exp\Big(-\frac{\kappa}{2}\sum\limits^m_{i=1} |y_i|^2\big[(\Delta u_i)^{2H}+(\Delta v_i)^{2H}\big]\Big)\, dy\, du\, dv\\
&\leq c_{14}\, n^{-\lambda}+c_{13}\,  n^{-\frac{m}{2}}   \int_{\widetilde{O}_{m}\times \widetilde{O}_{m}-\widetilde{O}^{\gamma}_{m}\times \widetilde{O}^{\gamma}_{m}} \int_{\R^{md}} \prod^{\frac{m}{2}}_{j=1}|\widehat{f}(y_{2j})|^2  \nonumber\\  
&\qquad\qquad\qquad \times \exp\Big(-\frac{\kappa}{2}\sum\limits^m_{i=1} |y_i|^2\big[(\Delta u_i)^{2H}+(\Delta v_i)^{2H}\big]\Big)\, dy\, du\, dv,
\end{align*}
where we use Lemma \ref{chain} in the last inequality.

For any odd numbers $k,\ell\in\{1,2,\cdots,m\}$ with $k\neq \ell$, define  
\[
\widetilde{O}^{\gamma}_{m,k,\ell}=\widetilde{O}_{m}\cap \big\{1/\gamma\leq \Delta u_k/\Delta u_{\ell} \leq \gamma\big\}.
\] 
Then, by the symmetry of $u$ and $v$ in the above inequality,
\begin{align*}
&n^{-\frac{m}{2}}   \int_{\widetilde{O}_{m}\times \widetilde{O}_{m}-\widetilde{O}^{\gamma}_{m}\times \widetilde{O}^{\gamma}_{m}} \int_{\R^{md}} \prod^{m/2}_{j=1}|\widehat{f}(y_{2j})|^2\exp\Big(-\frac{\kappa}{2}\sum\limits^m_{i=1} |y_i|^2\big[(\Delta u_i)^{2H}+(\Delta v_i)^{2H}\big]\Big)\, dy\, du\, dv\\
&\leq 2 \sum_{k\neq \ell}   n^{-\frac{m}{2}} \int_{\widetilde{O}^{\gamma}_{m,k,\ell}\times \widetilde{O}_{m}} \int_{\R^{md}} \prod^{m/2}_{j=1}|\widehat{f}(y_{2j})|^2\exp\Big(-\frac{\kappa}{2}\sum\limits^m_{i=1} |y_i|^2\big[(\Delta u_i)^{2H}+(\Delta v_i)^{2H}\big]\Big)\, dy\, du\, dv\\
&\leq c_{15}\sum_{k\neq \ell}  \left[n^{-\frac{m}{2}} \int_{\widetilde{O}^{\gamma}_{m,k,\ell}\times \widetilde{O}^{\gamma}_{m,k,\ell}} \int_{\R^{md}} \prod^{m/2}_{j=1}|\widehat{f}(y_{2j})|^2\exp\Big(-\frac{\kappa}{2}\sum\limits^m_{i=1} |y_i|^2\big[(\Delta u_i)^{2H}+(\Delta v_i)^{2H}\big]\Big)\, dy\, du\, dv\right]^{\frac{1}{2}}.
\end{align*}
Integrating with respect to the $u_i$ and $v_i$ for $i\neq k,\ell$ and  all $y_i$s gives
\begin{align*}
&n^{-\frac{m}{2}} \int_{\widetilde{O}^{\gamma}_{m,k,\ell}\times \widetilde{O}^{\gamma}_{m,k,\ell}} \int_{\R^{md}} \prod^{\frac{m}{2}}_{j=1}|\widehat{f}(y_{2j})|^2\exp\Big(-\frac{\kappa}{2}\sum\limits^m_{i=1} |y_i|^2\big[(\Delta u_i)^{2H}+(\Delta v_i)^{2H}\big]\Big)\, dy\, du\, dv\\
&\leq \frac{c_{16}}{n^2} \int^{e^{nt}}_{n^{-m}} \int^{e^{nt}}_{n^{-m}}  \int^{e^{nt}}_{n^{-m}}  \int^{e^{nt}}_{n^{-m}} \big[(\Delta u_k)^{2H}+(\Delta v_k)^{2H}\big]^{-\frac{d}{2}}\big[(\Delta u_{\ell})^{2H}+(\Delta v_{\ell})^{2H}\big]^{-\frac{d}{2}}\\
&\qquad\qquad\qquad\times 1_{\{1/\gamma\leq \Delta u_k/\Delta u_{\ell} \leq \gamma\}}1_{\{1/\gamma\leq \Delta v_k/\Delta v_{\ell} \leq \gamma\}}\, du_k\, du_{\ell}\, dv_k\, dv_{\ell}.
\end{align*}

Making the change of variables as in the proof of Lemma \ref{lemma1}, we could obtain that the right-hand side of the above inequality is less than a constant multiple of  $(\ln\gamma)/n$. Therefore, 
\[
|\widetilde{\I}^n_{m}-\widetilde{\I}^n_{m,\gamma}|\leq c_{16}\Big(n^{-\lambda}+ \sqrt\frac{\ln \gamma}{ n}\Big).
\]
This implies that $\limsup\limits_{n\to\infty}|\widetilde{\I}^n_m-\widetilde{\I}^n_{m,\gamma}|=0$.

\noindent \textbf{Step 2.} We show that $\widetilde{\I}^n_{m,\gamma}$  is asymptotically equal to $\widetilde{\I}^{n,\varepsilon}_{m,\gamma}$ defined in (\ref{gamma1}).  Making the change of variables $y_i=\sum\limits^m_{j=i}x_j$ for $i=1,2,\cdots, m$ (with the convention $y_{m+1}=0$) gives
\begin{align*}
\widetilde{\I}^n_{m,\gamma}
&=\frac{m!}{((2\pi)^d\sqrt{n})^m}\sum_{\sigma\in\mathscr{P}}  \int_{\widetilde{O}^{\gamma}_{m}\times \widetilde{O}^{\gamma}_{m}} \int_{\R^{md}} \prod^m_{i=1}\widehat{f}(y_i-y_{i+1}) \exp\Big(-\frac{1}{2}\Var\big(\sum\limits^m_{i=1} y_i\cdot (X_{u_i}-X_{u_{i-1}})\big)\Big)\\
&\qquad\qquad\qquad \times \exp\Big(-\frac{1}{2}\Var\big(\sum\limits^m_{i=1} \sum^m_{j=i} ( y_{\sigma(j)}-y_{\sigma(j)+1} )\cdot \big(X_{v_i}-X_{v_{i-1}})\big)\Big)\, dy\, du\, dv.
\end{align*}

For any $\varepsilon\in(0,1)$, define 
\begin{align} \label{gamma1}
\widetilde{\I}^{n,\varepsilon}_{m,\gamma}
&=\frac{m!}{((2\pi)^d\sqrt{n})^m} \sum_{\sigma \in \mathscr{P}} \int_{\widetilde{O}^{\gamma}_{m}\times \widetilde{O}^{\gamma}_{m}} \int_{T^{\sigma}_{\varepsilon}} \prod^m_{i=1}\widehat{f}(y_i-y_{i+1}) \exp\Big(-\frac{1}{2}\Var\big(\sum\limits^m_{i=1} y_i\cdot (X_{u_i}-X_{u_{i-1}})\big)\Big) \nonumber \\
&\qquad\qquad\qquad \times \exp\Big(-\frac{1}{2}\Var\big(\sum\limits^m_{i=1} \sum^m_{j=i} ( y_{\sigma(j)}-y_{\sigma(j)+1} )\cdot \big(X_{v_i}-X_{v_{i-1}})\big)\Big)\, dy\, du\, dv,
\end{align}
where 
\begin{align} \label{tsv}
T^{\sigma}_{\varepsilon}=\R^{md}\cap \bigg\{\left|y_{2k-1}\right|<\varepsilon,  ~ \bigg|\sum^m_{j={2k-1}} ( y_{\sigma(j)}-y_{\sigma(j)+1})\bigg|<\varepsilon,\; k=1,2,\cdots, m/2\bigg\}.
\end{align}

Let
\begin{align*}
T_{\sigma,\varepsilon}=\R^{md}\cap \bigg\{\bigg|\sum^m_{j={2k-1}} ( y_{\sigma(j)}-y_{\sigma(j)+1})\bigg|<\varepsilon,\; k=1,2,\cdots, m/2\bigg\}
\end{align*}
and
\begin{align*}
T_{\varepsilon}=\R^{md}\cap \big\{ |y_{2k-1}|<\varepsilon,\; k=1,2,\cdots, m/2\big\}.
\end{align*}
Then $T^{\sigma}_{\varepsilon}=T_{\varepsilon}\cap T_{\sigma,\varepsilon}$. This implies that
\[
\R^{md}-T^{\sigma}_{\varepsilon}=\left(\R^{md}-T_{\varepsilon} \right)\cup \left(\R^{md}-T_{\sigma,\varepsilon} \right).
\]
Therefore, 
\begin{align*}
&|\widetilde{\I}^n_{m,\gamma}-\widetilde{\I}^{n,\varepsilon}_{m,\gamma}|\\
&\leq c_{18}\, n^{-\frac{m}{2}}  \sum_{\sigma \in \mathscr{P}} \int_{\widetilde{O}^{\gamma}_{m} \times \widetilde{O}^{\gamma}_{m}} \int_{\R^{md}-T^{\sigma}_{\varepsilon}} \prod^m_{i=1}|\widehat{f}(y_i-y_{i+1})| \exp\Big(-\frac{1}{2}\Var\big(\sum\limits^m_{i=1} y_i\cdot (X_{u_i}-X_{u_{i-1}})\big)\Big)\\
&\qquad\qquad\qquad \times \exp\Big(-\frac{1}{2}\Var\big(\sum\limits^m_{i=1} \sum^m_{j=i} ( y_{\sigma(j)}-y_{\sigma(j)+1} )\cdot \big(X_{v_i}-X_{v_{i-1}})\big)\Big)\, dy\, du\, dv\\
&\leq c_{18}\, n^{-\frac{m}{2}}  \sum_{\sigma \in \mathscr{P}} \int_{\widetilde{O}^{\gamma}_{m} \times \widetilde{O}^{\gamma}_{m}} \int_{\R^{md}-T_{\varepsilon}} \prod^m_{i=1}|\widehat{f}(y_i-y_{i+1})| \exp\Big(-\frac{1}{2}\Var\big(\sum\limits^m_{i=1} y_i\cdot (X_{u_i}-X_{u_{i-1}})\big)\Big)\\
&\qquad\qquad\qquad \times \exp\Big(-\frac{1}{2}\Var\big(\sum\limits^m_{i=1} \sum^m_{j=i} ( y_{\sigma(j)}-y_{\sigma(j)+1} )\cdot \big(X_{v_i}-X_{v_{i-1}})\big)\Big)\, dy\, du\, dv\\
&\qquad+c_{18}\, n^{-\frac{m}{2}}  \sum_{\sigma \in \mathscr{P}} \int_{\widetilde{O}^{\gamma}_{m} \times \widetilde{O}^{\gamma}_{m}} \int_{\R^{md}-T_{\sigma,\varepsilon}} \prod^m_{i=1}|\widehat{f}(y_i-y_{i+1})| \exp\Big(-\frac{1}{2}\Var\big(\sum\limits^m_{i=1} y_i\cdot (X_{u_i}-X_{u_{i-1}})\big)\Big)\\
&\qquad\qquad\qquad \times \exp\Big(-\frac{1}{2}\Var\big(\sum\limits^m_{i=1} \sum^m_{j=i} ( y_{\sigma(j)}-y_{\sigma(j)+1} )\cdot \big(X_{v_i}-X_{v_{i-1}})\big)\Big)\, dy\, du\, dv.
\end{align*}
Using similar arguments as in {\bf (I)} {\bf Symmetrization of $|\I_m^n|$ via Cauchy-Schwarz inequality},
\begin{align*}
&|\widetilde{\I}^n_{m,\gamma}-\widetilde{\I}^{n,\varepsilon}_{m,\gamma}|\\
&\leq c_{19}\, n^{-\frac{m}{2}} \int_{\widetilde{O}^{\gamma}_{m} \times \widetilde{O}^{\gamma}_{m}} \int_{\R^{md}-T_{\varepsilon}} \prod^m_{i=1}|\widehat{f}(y_i-y_{i+1})| \exp\Big(-\frac{1}{2}\Var\big(\sum\limits^m_{i=1} y_i\cdot (X_{u_i}-X_{u_{i-1}})\big)\Big) \\
&\qquad\qquad\qquad\qquad\qquad\qquad\qquad \times \exp\Big(-\frac{1}{2}\Var\big(\sum\limits^m_{i=1} y_i \cdot \big(X_{v_i}-X_{v_{i-1}})\big)\Big)\, dy\, du\, dv.
\end{align*}
Now, by Assumption (B) and Lemma  \ref{chain}, $|\widetilde{\I}^n_{m,\gamma}-\widetilde{\I}^{n,\varepsilon}_{m,\gamma}|$ is less than
\begin{align*}
&c_{20}\, n^{-\lambda}+c_{20}\, n^{-\frac{m}{2}} \int_{\widetilde{O}^{\gamma}_{m}\times \widetilde{O}^{\gamma}_{m}}\int_{\R^{md}-T_{\varepsilon}} \prod^{m/2}_{j=1}|\widehat{f}(y_{2j})|^2\\
&\qquad\qquad\qquad\qquad\qquad\qquad \times \exp\Big(-\frac{\kappa}{2}\sum\limits^m_{i=1} |y_i|^2\big[(\Delta u_i)^{2H}+(\Delta v_i)^{2H}\big]\Big)\, dy \, du\, dv\\
&\leq c_{20}\, n^{-\lambda}+c_{21}\, n^{-1} \int^{e^{nt}/m}_{n^2}\int^{e^{nt}/m}_{n^2}\int_{|x|\geq \varepsilon} \exp\Big(-\frac{\kappa}{2}|x|^2(s^{2H}+t^{2H})\Big) \, dx\, ds\, dt\\
&\leq c_{20}\, n^{-\lambda}+c_{22}\, n^{-1} e^{-\frac{\kappa}{2}\varepsilon^2 n^{4H}} \int^{e^{nt}/m}_{n^2}\int^{e^{nt}/m}_{n^2}\int_{|x|\geq \varepsilon} \exp\Big(-\frac{\kappa}{4}|x|^2(s^{2H}+t^{2H})\Big) \, dx\, ds\, dt\\
&\leq c_{20}\, n^{-\lambda}+c_{23}\,e^{-\frac{\kappa}{2}\varepsilon^2 n^{4H}}.  
\end{align*}
This gives $\limsup\limits_{n\to\infty} |\widetilde{\I}^n_{m,\gamma}-\widetilde{\I}^{n,\varepsilon}_{m,\gamma}|=0$.
 
\noindent \textbf{Step 3.} Recall the change of variables $y_i=\sum\limits^m_{j=i}x_j$ for $i=1,2,\cdots, m$. We see that $\widetilde{\I}^{n,\varepsilon}_{m,\gamma}$ can also be written as
\begin{align*} 
\widetilde{\I}^{n,\varepsilon}_{m,\gamma}
&=\frac{m!}{((2\pi)^d\sqrt{n})^m} \sum_{\sigma\in\mathscr{P}}\int_{\widetilde{O}^{\gamma}_{m} \times \widetilde{O}^{\gamma}_{m}} \int_{\overline{T}^{\sigma}_{\varepsilon}}  \prod^m_{i=1}\widehat{f}(x_i) \exp\Big(-\frac{1}{2}\Var\big(\sum\limits^m_{i=1} \sum\limits^m_{j=i}x_j \cdot (X_{u_i}-X_{u_{i-1}})\big)\Big) \nonumber \\
&\qquad\qquad\qquad \times \exp\Big(-\frac{1}{2}\Var\big(\sum\limits^m_{i=1} \sum^m_{j=i} x_{\sigma(j)} \cdot \big(X_{v_i}-X_{v_{i-1}})\big)\Big)\, dx\, du\, dv,
\end{align*}
where 
\begin{align} \label{domain}
\overline{T}^{\sigma}_{\varepsilon}=\R^{md}\cap \Big\{\Big|\sum\limits^m_{j=2k-1}x_j\Big|<\varepsilon,\;  \Big|\sum^m_{j={2k-1}} x_{\sigma(j)}\Big|<\varepsilon,\; k=1,2,\cdots, m/2\Big\}.
\end{align}

For any $\sigma\in\mathscr{P}$, define 
\begin{align} \label{ctr}
\mathscr{P}_1=\{\sigma\in\mathscr{P}:\; \# A(\sigma)=m/2\} ~~\text{ and  } ~~ \mathscr{P}_0=\mathscr{P}-\mathscr{P}_1,
\end{align}
 where
\begin{align*}
A(\sigma)=\Big\{\{2k,2k-1\},\, k=1,2,\cdots, m/2\Big\}\cap \Big\{\{\sigma(2k),\sigma(2k-1)\},\, k=1,2,\cdots, m/2\Big\}.
\end{align*}

For any $\sigma\in\mathscr{P}$, let 
\begin{align*} 
\widetilde{\I}^{n,\varepsilon,\sigma}_{m,\gamma}
&=\frac{m!}{((2\pi)^d\sqrt{n})^m} \int_{\widetilde{O}^{\gamma}_{m}\times \widetilde{O}^{\gamma}_{m}} \int_{\overline{T}^{\sigma}_{\varepsilon}}  \prod^m_{i=1}\widehat{f}(x_i) \exp\Big(-\frac{1}{2}\Var\big(\sum\limits^m_{i=1} \sum\limits^m_{j=i}x_j \cdot (X_{u_i}-X_{u_{i-1}})\big)\Big) \nonumber \\
&\qquad \qquad \qquad\qquad\qquad \times \exp\Big(-\frac{1}{2}\Var\big(\sum\limits^m_{i=1} \sum^m_{j=i} x_{\sigma(j)} \cdot \big(X_{v_i}-X_{v_{i-1}})\big)\Big)\, dx\, du\, dv.
\end{align*}

In the following, we will show the asymptotic behavior  of $\widetilde{\I}^{n,\varepsilon,\sigma}_{m,\gamma}$ when $\sigma\in \mathscr{P}_0$ and $\sigma\in \mathscr{P}_1$, respectively.  

For any $\sigma\in \mathscr{P}_0$, there exist $j, k,\ell\in\{1,2,\cdots, m/2\}$ with $k\neq \ell$ such that 
\begin{align} \label{permutation}
\sigma(2j)\in\{2k,2k-1\}\; \text{and}\; \sigma(2j-1)\in\{2\ell,2\ell-1\}.
\end{align}
By the definition of $\overline{T}^{\sigma}_{\varepsilon}$ in (\ref{domain}). For any $i=1,2,\dots, m/2$,
\begin{align} \label{paring}
|x_{2i}+x_{2i-1}|\leq (m/2-i+1)\varepsilon\quad \text{and}\quad |x_{\sigma(2i)}+x_{\sigma(2i-1)}|\leq (m/2-i+1)\varepsilon.
\end{align}
We claim that 
\begin{align*} 
|x_{2k}-x_{2\ell}|\leq 2m\varepsilon \quad \text{or}\quad |x_{2k}+x_{2\ell}|\leq 2m\varepsilon. 
\end{align*}
In fact, from (\ref{permutation}) there are four possibilities for the values of $\sigma(2j)$ and $\sigma(2j-1)$: (1) $\sigma(2j)=2k$ and $\sigma(2j-1)=2\ell$; (2) $\sigma(2j)=2k$ and $\sigma(2j-1)=2\ell-1$; (3) $\sigma(2j)=2k-1$ and $\sigma(2j-1)=2\ell$;  (4) $\sigma(2j)=2k-1$ and $\sigma(2j-1)=2\ell-1$. In this first case, the claim follows from (\ref{paring}) directly. In the second and the third cases, 
\begin{align*} 
|x_{2k}-x_{2\ell}|\leq |x_{2k}-(-1)^{\sigma(2j)}x_{\sigma(2j)}|+|x_{\sigma(2j)}+x_{\sigma(2j-1)}|+|(-1)^{\sigma(2j-1)}x_{\sigma(2j-1)}-x_{2\ell}|\leq 2m\varepsilon.
\end{align*}
In the last case,
\begin{align*} 
|x_{2k}+x_{2\ell}|\leq |x_{2k}+x_{\sigma(2j)}|+|x_{\sigma(2j)}+x_{\sigma(2j-1)}|+|x_{\sigma(2j-1)}+x_{2\ell}|\leq 2m\varepsilon.
\end{align*}

We next show that 
\[
|y_{2k}-y_{2\ell}|\leq 4m\varepsilon \quad \text{or}\quad |y_{2k}-y_{2\ell}|\leq 4m\varepsilon.
\] 
Without loss of generality, we can assume that $k<\ell$. Then
\begin{align*}
|y_{2k}-y_{2\ell}|
&=\bigg|\sum^{2\ell}_{j=2k+1} x_j+x_{2k}-x_{2\ell}\bigg|\leq 4m\varepsilon
\end{align*}
if $|x_{2k}-x_{2\ell}|\leq 2m\varepsilon$, and 
\begin{align*}
|y_{2k}+y_{2\ell}|
&=|2\sum^{m}_{j=2\ell+1} x_j+\sum^{2\ell}_{j=2k+1} x_j+x_{2k}+x_{2\ell}|\leq 4m\varepsilon
\end{align*}
if $|x_{2k}+x_{2\ell}|\leq 2m\varepsilon$.
   
Using similar arguments as in {\bf (I)} {\bf Symmetrization of $|\I_m^n|$ via Cauchy-Schwarz inequality} and then Lemma \ref{chain},
\begin{align*} 
|\widetilde{\I}^{n,\varepsilon,\sigma}_{m,\gamma}|&\leq c_{24}\,\sum_{1\leq k\neq \ell\leq m/2} n^{-\frac{m}{2}}  \int_{\widetilde{O}^{\gamma}_{m} \times \widetilde{O}^{\gamma}_{m}} \int_{\overline{T}^{\sigma}_{\varepsilon}} \prod^m_{i=1}|\widehat{f}(x_i)| \exp\Big(-\frac{1}{2}\Var\big(\sum\limits^m_{i=1} \sum\limits^m_{j=i}x_j \cdot (X_{u_i}-X_{u_{i-1}})\big)\Big) \nonumber \\
&\qquad\qquad\qquad \times \exp\Big(-\frac{1}{2}\Var\big(\sum\limits^m_{i=1} \sum^m_{j=i} x_{j} \cdot \big(X_{v_i}-X_{v_{i-1}})\big)\Big)\, dx\, du\, dv\\
&\leq c_{25}\, n^{-\lambda}+c_{25}\,\sum_{1\leq k\neq \ell\leq m/2}\, n^{-\frac{m}{2}}  \int_{\widetilde{O}^{\gamma}_{m} \times \widetilde{O}^{\gamma}_{m}} \int_{\R^{md}} \prod^{m/2}_{j=1}|\widehat{f}(y_{2j})|^2\\
&\qquad\qquad\qquad \times \exp\Big(-\frac{\kappa}{2}\sum\limits^m_{i=1} |y_i|^2 [(\Delta u_i)^{2H}+(\Delta v_i)^{2H}]\Big) \, 1_{\{|y_{2k}\pm y_{2\ell}|\leq 4m\varepsilon\}}\, dy\, du\, dv.
\end{align*}
This yields, by Lemmas \ref{lemma1} and \ref{lemma3},
\begin{align} \label{ppp1}
\limsup\limits_{n\to\infty}|\widetilde{\I}^{n,\varepsilon,\sigma}_{m,\gamma}|\leq c_{26}\, \int_{\R^{2d}} |\widehat{f}(x)|^2 |\widehat{f}(y)|^2|x|^{-d}|y|^{-d}\, 1_{\{|x-y|\leq 4m\varepsilon\}}\, dx\, dy,~ \text{ for } \sigma\in \mathscr P_0.
\end{align}
Note that the integral on the right hand side goes to zero as $\varepsilon$ tends to zero.

Observe that 
\begin{align} \label{gamma10}
\widetilde{\I}^{n,\varepsilon,\sigma}_{m,\gamma}
&=\frac{m!}{((2\pi)^d\sqrt{n})^m}  \int_{\widetilde{O}^{\gamma}_{m}\times \widetilde{O}^{\gamma}_{m}} \int_{T^{\sigma}_{\varepsilon}} \prod^m_{i=1}\widehat{f}(y_i-y_{i+1}) \exp\Big(-\frac{1}{2}\Var\big(\sum\limits^m_{i=1} y_i\cdot (X_{u_i}-X_{u_{i-1}})\big)\Big) \nonumber \\
&\qquad\qquad\qquad \times \exp\Big(-\frac{1}{2}\Var\big(\sum\limits^m_{i=1} \sum^m_{j=i} ( y_{\sigma(j)}-y_{\sigma(j)+1} )\cdot \big(X_{v_i}-X_{v_{i-1}})\big)\Big)\, dy\, du\, dv.
\end{align}
Recall $T^{\sigma}_{\varepsilon}$ in (\ref{tsv}). Let $T^{\sigma}_{\varepsilon,1}=T^{\sigma}_{\varepsilon}-T^{\sigma}_{\varepsilon,2}$ where 
\[
T^{\sigma}_{\varepsilon,2}=T^{\sigma}_{\varepsilon}\cap \left\{|y_{2i}|>\gamma\varepsilon: i=1,2,\cdots,m/2\right\}.
\]  

Define 
\begin{align*} 
\widetilde{\I}^{n,\varepsilon,\sigma}_{m,\gamma,1}
&=\frac{m!}{((2\pi)^d\sqrt{n})^m}  \int_{\widetilde{O}^{\gamma}_{m}\times \widetilde{O}^{\gamma}_{m}} \int_{T^{\sigma}_{\varepsilon,1}} \prod^m_{i=1}\widehat{f}(y_i-y_{i+1}) \exp\Big(-\frac{1}{2}\Var\big(\sum\limits^m_{i=1} y_i\cdot (X_{u_i}-X_{u_{i-1}})\big)\Big) \nonumber \\
&\qquad\qquad\qquad \times \exp\Big(-\frac{1}{2}\Var\big(\sum\limits^m_{i=1} \sum^m_{j=i} ( y_{\sigma(j)}-y_{\sigma(j)+1} )\cdot \big(X_{v_i}-X_{v_{i-1}})\big)\Big)\, dy\, du\, dv,
\end{align*}
\begin{align*} 
\widetilde{\I}^{n,\varepsilon,\sigma}_{m,\gamma,2}
&=\frac{m!}{((2\pi)^d\sqrt{n})^m} \int_{\widetilde{O}^{\gamma}_{m}\times \widetilde{O}^{\gamma}_{m}} \int_{T^{\sigma}_{\varepsilon,2}} \prod^m_{i=1}\widehat{f}(y_i-y_{i+1}) \exp\Big(-\frac{1}{2}\Var\big(\sum\limits^m_{i=1} y_i\cdot (X_{u_i}-X_{u_{i-1}})\big)\Big) \nonumber \\
&\qquad\qquad\qquad \times \exp\Big(-\frac{1}{2}\Var\big(\sum\limits^m_{i=1} \sum^m_{j=i} ( y_{\sigma(j)}-y_{\sigma(j)+1} )\cdot \big(X_{v_i}-X_{v_{i-1}})\big)\Big)\, dy\, du\, dv
\end{align*}
and 
\begin{align*} 
\widetilde{\I}^{n,\varepsilon,\sigma}_{m,\gamma,3}
&=\frac{m!}{((2\pi)^d\sqrt{n})^m} \int_{\widetilde{O}^{\gamma}_{m}\times \widetilde{O}^{\gamma}_{m}} \int_{T^{\sigma}_{\varepsilon,2}}  \prod^{m/2}_{j=1}|\widehat{f}(y_{2j})|^2 \exp\Big(-\frac{1}{2}\Var\big(\sum\limits^m_{i=1} y_i\cdot (X_{u_i}-X_{u_{i-1}})\big)\Big) \nonumber \\
&\qquad\qquad\qquad \times \exp\Big(-\frac{1}{2}\Var\big(\sum\limits^m_{i=1} \sum^m_{j=i} ( y_{\sigma(j)}-y_{\sigma(j)+1} )\cdot \big(X_{v_i}-X_{v_{i-1}})\big)\Big)\, dy\, du\, dv.
\end{align*}
Obviously, $\widetilde{\I}^{n,\varepsilon,\sigma}_{m,\gamma}=\widetilde{\I}^{n,\varepsilon,\sigma}_{m,\gamma,1}+\widetilde{\I}^{n,\varepsilon,\sigma}_{m,\gamma,2}$.

\noindent \textbf{Step 4.}  For any $\sigma\in \mathscr{P}_1$, we will show that $\limsup\limits_{n\to\infty} |\widetilde{\I}^{n,\varepsilon,\sigma}_{m,\gamma,1}|$ is less than a constant multiple of $\int_{|y|\leq \gamma\varepsilon}|\widehat{f}(y)|^2|y|^{-d}\, dy$ and $\limsup\limits_{n\to\infty} |\widetilde{\I}^{n,\varepsilon,\sigma}_{m,\gamma,2}-\widetilde{\I}^{n,\varepsilon,\sigma}_{m,\gamma,3}|=0$ when $\gamma$ is large enough.

Recall the definition of $\mathscr{P}_1$ in (\ref{ctr}). It is easy to see that $\# \mathscr{P}_1=2^{\frac{m}{2}}(\frac{m}{2})!$. Moreover, for any $\sigma\in \mathscr{P}_1$,  the expression of summation $\sum^m_{j=i} ( y_{\sigma(j)}-y_{\sigma(j)+1})$ 
on the right-hand side of (\ref{gamma10})  after simplification only has two possibilities. One is that it consists of only variables $y$ with odd indices when $i$ is odd, and the other is that among the variables $y$ in its expression,  there is only one variable $y$ with even index when $i$ is even. Note that all variables $y$ with odd indices are in the ball centered at the origin with radius $\varepsilon$ and $\varepsilon$ is a positive constant which can be chosen arbitrarily small.  

For any $\sigma\in \mathscr{P}_1$, using similar arguments as in {\bf (I)} {\bf Symmetrization of $|\I_m^n|$ via Cauchy-Schwarz inequality} and Lemma \ref{chain},
\begin{align}  \label{ppp2}
\limsup\limits_{n\to\infty} |\widetilde{\I}^{n,\varepsilon,\sigma}_{m,\gamma,1}|
&\leq c_{27} \sum\limits^m_{i=1, i: even} \int_{|y_{\overline{\sigma}(i)}|\leq \gamma \varepsilon}|\widehat{f}(y_{\overline{\sigma}(i)})|^2|y_{\overline{\sigma}(i)}|^{-d}\, dy_{\overline{\sigma}(i)}\nn\\
&\leq c_{28} \int_{|y|\leq \gamma\varepsilon}|\widehat{f}(y)|^2|y|^{-d}\, dy,
\end{align}
where $\overline{\sigma}(i)=\sigma(i)$ if $\sigma(i)$ is even and  $\sigma(i-1)$ otherwise.

Define
\begin{align*} 
\widetilde{\J}^{n,\varepsilon,\sigma}_{m,\gamma,2}
&=\frac{m!}{((2\pi)^d\sqrt{n})^m}  \int_{\widetilde{O}^{\gamma}_{m}\times \widetilde{O}^{\gamma}_{m}} \int_{T^{\sigma}_{\varepsilon,2,\gamma}} \prod^m_{i=1}\widehat{f}(y_i-y_{i+1}) \exp\Big(-\frac{1}{2}\Var\big(\sum\limits^m_{i=1} y_i\cdot (X_{u_i}-X_{u_{i-1}})\big)\Big) \nonumber \\
&\qquad\qquad\qquad \times \exp\Big(-\frac{1}{2}\Var\big(\sum\limits^m_{i=1} \sum^m_{j=i} ( y_{\sigma(j)}-y_{\sigma(j)+1})\cdot \big(X_{v_i}-X_{v_{i-1}})\big)\Big)\, dy\, du\, dv
\end{align*}
and
\begin{align*} 
\widetilde{\J}^{n,\varepsilon,\sigma}_{m,\gamma,3}
&=\frac{m!}{((2\pi)^d\sqrt{n})^m}   \int_{\widetilde{O}^{\gamma}_{m}\times \widetilde{O}^{\gamma}_{m}} \int_{T^{\sigma}_{\varepsilon,2,\gamma}} \prod^{m/2}_{j=1}|\widehat{f}(y_{2j})|^2\exp\Big(-\frac{1}{2}\Var\big(\sum\limits^m_{i=1} y_i\cdot (X_{u_i}-X_{u_{i-1}})\big)\Big) \nonumber \\
&\qquad\qquad\qquad \times \exp\Big(-\frac{1}{2}\Var\big(\sum\limits^m_{i=1} \sum^m_{j=i} ( y_{\sigma(j)}-y_{\sigma(j)+1} )\cdot \big(X_{v_i}-X_{v_{i-1}})\big)\Big)\, dy\, du\, dv,
\end{align*}
where
\[
T^{\sigma}_{\varepsilon,2,\gamma}=T^{\sigma}_{\varepsilon,2}-\bigcup_{i\neq j\in\{2k-1: k=1,2,\cdots,m/2\}}\big\{|y_j|/\gamma<|y_i|<\gamma |y_j|\big\}.
\]

Now, for any $\sigma\in \mathscr{P}_1$, using similar arguments as in obtaining \eref{jm2} with the help of {\bf (I)} {\bf Symmetrization of $|\I_m^n|$ via Cauchy-Schwarz inequality} and Lemma \ref{chain}, we get
\begin{align*}  \label{jjj}
\limsup\limits_{n\to\infty}|\widetilde{\I}^{n,\varepsilon, \sigma}_{m,\gamma,2}-\widetilde{\J}^{n,\varepsilon, \sigma}_{m,\gamma,2}|=0\quad\text{and}\quad \limsup\limits_{n\to\infty}|\widetilde{\I}^{n,\varepsilon, \sigma}_{m,\gamma,3}-\widetilde{\J}^{n,\varepsilon, \sigma}_{m,\gamma,3}|=0
\end{align*}
provided that  $\gamma$ is large enough.

Next we estimate $|\widetilde{\J}^{n,\varepsilon, \sigma}_{m,\gamma,2}-\widetilde{\J}^{n,\varepsilon, \sigma}_{m,\gamma,3}|.$  
\begin{align*}
|\widetilde{\J}^{n,\varepsilon, \sigma}_{m,\gamma,2}-\widetilde{\J}^{n,\varepsilon, \sigma}_{m,\gamma,3}|
&\leq \frac{m!}{((2\pi)^d\sqrt{n})^m} \int_{\widetilde{O}^{\gamma}_{m}\times \widetilde{O}^{\gamma}_{m}} \int_{T^{\sigma}_{\varepsilon,2,\gamma}} \Big|\prod^m_{i=1}\widehat{f}(y_i-y_{i+1})-\prod^{m/2}_{j=1}|\widehat{f}(y_{2j})|^2\Big|\\
&\qquad\times \exp\Big(-\frac{1}{2}\Var\big(\sum\limits^m_{i=1} y_i\cdot (X_{u_i}-X_{u_{i-1}})\big)\Big) \nonumber \\
&\qquad\qquad \times \exp\Big(-\frac{1}{2}\Var\big(\sum\limits^m_{i=1} \sum^m_{j=i} ( y_{\sigma(j)}-y_{\sigma(j)+1} )\cdot \big(X_{v_i}-X_{v_{i-1}})\big)\Big)\, dy\, du\, dv.
\end{align*}

For $\gamma$ large enough,  on $T^{\sigma}_{\varepsilon,2, \gamma}$,
\begin{align*}
(1-\frac{m}{\gamma})\sup_{j\in A_i^\sigma}|y_{j}|\leq \Big|\sum\limits^m_{j=i} ( y_{\sigma(j)}-y_{\sigma(j)+1})\Big|,
\end{align*}
where 
\[
A^{\sigma}_i=\big\{\sigma(i), \cdots, \sigma(m)\big\}\Delta \big\{\sigma(i)+1, \cdots, \sigma(m)+1\big\}.
\]
So, by Assumption (B),
\begin{align*} 
|\widetilde{\J}^{n,\varepsilon, \sigma}_{m,\gamma,2}-\widetilde{\J}^{n,\varepsilon, \sigma}_{m,\gamma,3}|
&\leq \frac{m!}{((2\pi)^d\sqrt{n})^m} \int_{\widetilde{O}^{\gamma}_{m}\times \widetilde{O}^{\gamma}_{m}} \int_{T^{\sigma}_{\varepsilon,2,\gamma}} \Big|\prod^m_{i=1}\widehat{f}(y_i-y_{i+1})-\prod^{m/2}_{j=1}|\widehat{f}(y_{2j})|^2\Big| \nonumber\\
&\qquad\times  \exp\Big(-\frac{\kappa}{2}\sum\limits^m_{i=1} |y_i|^2(u_i-u_{i-1})^{2H}\Big) \nonumber \\
&\qquad\qquad \times \exp\Big(-\frac{\kappa}{2}(1-\frac{m}{\gamma})^2\sum\limits^m_{i=1} \sup_{j\in A^{\sigma}_i} |y_j|^2(v_i-v_{i-1})^{2H}\Big).
\end{align*}

By Lemma \ref{lemma7}, there exists $\widetilde \sigma\in\mathscr P$ such that $\widetilde\sigma(j)$ and $j$ have the same parity for all $j=1,2,\cdots,m$, and
\begin{align*}
|\widetilde{\J}^{n,\varepsilon, \sigma}_{m,\gamma,2}-\widetilde{\J}^{n,\varepsilon, \sigma}_{m,\gamma,3}|
&\leq \frac{m!}{((2\pi)^d\sqrt{n})^m} \int_{\widetilde{O}^{\gamma}_{m}\times \widetilde{O}^{\gamma}_{m}} \int_{T^{\sigma}_{\varepsilon,2,\gamma}} \Big|\prod^m_{i=1}\widehat{f}(y_i-y_{i+1})-\prod^{m/2}_{j=1}|\widehat{f}(y_{2j})|^2\Big|\nn\\
&\qquad\times  \exp\Big(-\frac{\kappa}{8}\sum\limits^m_{i=1} |y_i|^2(u_i-u_{i-1})^{2H}-\frac{\kappa}{8}\sum\limits^m_{i=1} |y_{\widetilde{\sigma}(i)}|^2(v_i-v_{i-1})^{2H}\Big)\, dy\, du\, dv.
\end{align*}

Now, using similar arguments as in the proof of Lemma \ref{chain} to the right hand side of the above inequality, we can get
\[
\limsup\limits_{n\to\infty}|\widetilde{\J}^{n,\varepsilon, \sigma}_{m,\gamma,2}-\widetilde{\J}^{n,\varepsilon, \sigma}_{m,\gamma,3}|=0.
\]
Therefore, $\limsup\limits_{n\to\infty} |\widetilde{\I}^{n,\varepsilon,\sigma}_{m,\gamma,2}-\widetilde{\I}^{n,\varepsilon,\sigma}_{m,\gamma,3}|=0$ when $\gamma$ is large enough.

\noindent \textbf{Step 5.}  We obtain the limit of $\I^{n}_{m}$ as $n$ tends to $\infty$.  By Assumptions (C1) and (C2),
\[
\Var\big(\sum\limits^m_{i=1} y_i\cdot (X_{u_i}-X_{u_{i-1}})\big)
\]
is between $\underline{\beta}(\gamma,n)\sum\limits^m_{i=1} |y_i|^2 \E(X^1_{u_i}-X^1_{u_{i-1}})^2$ and $\overline{\beta}(\gamma,n)\sum\limits^m_{i=1} |y_i|^2 \E(X^1_{u_i}-X^1_{u_{i-1}})^2$, where $\underline{\beta}(\gamma,n)=1-c_{29}\beta_1(\gamma)-c_{29}\beta_2(n)$ and $\overline{\beta}(\gamma,n)=1+c_{30}\beta_1(\gamma)+c_{30}\beta_2(n)$.

In the sequel, we always assume that $\gamma$ is very large. Note that
\begin{align*}  
\limsup\limits_{n\to\infty}\widetilde{\I}^{n,\varepsilon, \sigma}_{m,\gamma,3}
&\leq \limsup\limits_{n\to\infty}\frac{m!}{((2\pi)^d\sqrt{n})^m}  \int_{\widetilde{O}^{\gamma}_{m} \times \widetilde{O}^{\gamma}_{m}} \int_{T^{\sigma}_{\varepsilon,2}} \prod^{m/2}_{j=1}|\widehat{f}(y_{2j})|^2 \nonumber\\
&\qquad\times \exp\Big(-\frac{\underline{\beta}(\gamma,n)}{2}\sum\limits^m_{i=1} |y_i|^2 \E(X^1_{u_i}-X^1_{u_{i-1}})^2\Big) \nonumber \\
&\qquad\times \exp\Big(-\frac{\underline{\beta}(\gamma,n)}{2}\sum\limits^m_{i=1} |\sum^m_{j=i} ( y_{\sigma(j)}-y_{\sigma(j)+1})|^2\E(X^1_{v_i}-X^1_{v_{i-1}})^2\Big)\, dy\, du\, dv\\
&\leq \limsup\limits_{n\to\infty}\frac{m!}{((2\pi)^d\sqrt{n})^m}  \int_{\widetilde{O}^{\gamma}_{m} \times \widetilde{O}^{\gamma}_{m}} \int_{T^{\sigma}_{\varepsilon}} \prod^{m/2}_{j=1}|\widehat{f}(y_{2j})|^2\\
&\qquad  \times \exp\Big(-\frac{\underline{\beta}(\gamma,n)}{2}\sum\limits^m_{i=1} |y_i|^2 \E(X^1_{u_i}-X^1_{u_{i-1}})^2\Big) \nonumber \\
&\qquad \times \exp\Big(-\frac{\underline{\beta}(\gamma,n)}{2}(1-\frac{m}{\gamma}) \sum\limits^m_{i=1, i: even} |y_{\overline{\sigma}(i)}|^2 \E(X^1_{v_i}-X_{v_{i-1}})^2\Big)\\
&\qquad\times \exp\Big(-\frac{\underline{\beta}(\gamma,n)}{2}\sum\limits^m_{i=1, i: odd} \sum^m_{j=i} ( y_{\sigma(j)}-y_{\sigma(j)+1})\E(X^1_{v_i}-X^1_{v_{i-1}})^2\Big)\, dy\, du\, dv.
\end{align*}

By Assumption (A2),  on $\widetilde O_m^\gamma$ and  for even $i$, $\E(X^1_{u_i}-X^1_{u_{i-1}})^2$ and $\E(X^1_{v_i}-X^1_{v_{i-1}})^2$ are greater than $\overline\alpha_2 (u_i-u_{i-1})^{2H}$ and $\overline\alpha_2(v_i-v_{i-1})^{2H}$, respectively, where $\overline\alpha_2$ is given in \eqref{eq-alpha}. Then,
 \begin{align*}
\limsup\limits_{n\to\infty}\sum_{\sigma\in \mathscr P_1} \widetilde{\I}^{n,\varepsilon, \sigma}_{m,\gamma,3}
&\leq \frac{2^m(m-1)!!}{(2\pi)^{md/2}} \left(\frac{1}{(2\pi)^d}\int_{\R^d}|\widehat{f}(z)|^2|z|^{-d}\, dz\right)^{m/2}\\
&\quad\times \left( \int_{[0,+\infty)^2}\exp\left(-\frac{1}{2}(1-c_{29}\beta_1(\gamma))(1-\frac{m}{\gamma})\overline\alpha_2(u^{2H}+v^{2H})\right)\, du\,dv \right)^{m/2}\\
&\quad\quad\times \left(\overline{b}(\gamma) \overline{\alpha}_2(\gamma)\right)^{-\frac{md}{4}} \left[ \frac{\Gamma(\frac{m}{2}+(\frac{\overline{\alpha}_2(\gamma)}{\overline{\alpha}_1(\gamma)})^{\frac{d}{4}})^{\frac{d}{4}})}{(\frac{m}{2})!\Gamma((\frac{\overline{\alpha}_2(\gamma)}{\overline{\alpha}_1(\gamma)})^{\frac{d}{4}})} \right]^2 \left[\frac{2t\pi^{\frac{d}{2}}\Gamma^2(\frac{d+4}{4})}{(1-\frac{m}{\gamma})^{\frac{d}{4}}\Gamma(\frac{d+2}{2})}\right]^{m/2}(m-1)!!,
\end{align*}
where $\overline\alpha_1$ and $\overline\alpha_2$ are given in \eqref{eq-alpha}, and $\overline{b}(\gamma)=\frac{1}{2}-m\beta_1(\gamma)$, the two integrals are from the integration with respect to $\Delta u_i=u_i-u_{i-1}$ and $\Delta v_i=v_i-v_{i-1}$ with even indices $i$ and the fact that $\int_0^\infty e^{-c |z|^2 s^{2H}}ds= |z|^{-d/2} \int_0^\infty e^{-c s^{2H}}ds$, and the terms in the third line follows from the methodology used in  {\bf Step 3} and {\bf Step 4} of the proof for Proposition \ref{prop}.

Recall that $\widetilde{\I}^{n,\varepsilon,\sigma}_{m,\gamma}=\widetilde{\I}^{n,\varepsilon,\sigma}_{m,\gamma,1}+\widetilde{\I}^{n,\varepsilon,\sigma}_{m,\gamma,2}$, inequality (\ref{ppp2}) and $\limsup\limits_{n\to\infty} |\widetilde{\I}^{n,\varepsilon,\sigma}_{m,\gamma,2}-\widetilde{\I}^{n,\varepsilon,\sigma}_{m,\gamma,3}|=0$ for $\sigma\in \mathscr{P}_1$. So $\limsup\limits_{n\to\infty}\sum\limits_{\sigma\in \mathscr{P}_1}\widetilde{\I}^{n,\varepsilon, \sigma}_{m,\gamma}$ is less than
\begin{align*}
&c_{31} \int_{|y|\leq \gamma\varepsilon}|\widehat{f}(y)|^2|y|^{-d}\, dy+\frac{2^m((m-1)!!)^2}{(2\pi)^{md/2}} \left(\frac{1}{(2\pi)^d}\int_{\R^d}|\widehat{f}(z)|^2|z|^{-d}\, dz\right)^{m/2}
&\qquad\qquad\qquad\qquad\qquad\qquad \times \left( \int_{[0,+\infty)^2}\exp\left(-\frac{1}{2}(1-c_{28}\beta_1(\gamma))(1-\frac{m}{\gamma})\alpha_2(u^{2H}+v^{2H})\right)\, du\,dv \right)^{m/2}\\
&\qquad\qquad\qquad\qquad\qquad\times \left( \int_{[0,+\infty)^2}\exp\left(-\frac{1}{2}(1-c_{29}\beta_1(\gamma))(1-\frac{m}{\gamma}) \overline\alpha_2 (u^{2H}+v^{2H})\right)\, du\,dv \right)^{m/2}\\
&\qquad\qquad\qquad\qquad\qquad\qquad\qquad \times \left(\overline{b}(\gamma) \overline{\alpha}_2(\gamma)\right)^{-\frac{md}{4}} \left[ \frac{\Gamma(\frac{m}{2}+(\frac{\overline{\alpha}_2(\gamma)}{\overline{\alpha}_1(\gamma)})^{\frac{d}{4}})}{(\frac{m}{2})!\Gamma((\frac{\overline{\alpha}_2(\gamma)}{\overline{\alpha}_1(\gamma)})^{\frac{d}{4}})} \right]^2 \left[\frac{2t\pi^{\frac{d}{2}}\Gamma^2(\frac{d+4}{4})}{(1-\frac{m}{\gamma})^{\frac{d}{4}}\Gamma(\frac{d+2}{2})}\right]^{m/2}.
\end{align*}
Note that by \textbf{Step 2},
\begin{align*}
\limsup\limits_{n\to\infty}\widetilde{\I}^{n}_{m,\gamma}\leq  \limsup\limits_{n\to\infty}\sum_{\sigma\in \mathscr{P}_0}\widetilde{\I}^{n,\varepsilon, \sigma}_{m,\gamma}+\limsup\limits_{n\to\infty}\sum_{\sigma\in \mathscr{P}_1}\widetilde{\I}^{n,\varepsilon, \sigma}_{m,\gamma}.
\end{align*}
Taking $\varepsilon\to 0$ first and then $\gamma\to+\infty$ on the right hand side of the above inequality, we obtain, by \textbf{Step 1},
 \begin{align*}
\limsup\limits_{n\to\infty}{\I}^{n}_{m}\leq\left[ \frac{\Gamma(\frac{m}{2}+(\frac{\alpha_2}{\alpha_1})^{\frac{d}{4}})}{(\frac{m}{2})!\Gamma((\frac{\alpha_2}{\alpha_1})^{\frac{d}{4}})} \right]^2 (D_{f,d}\, t)^{m/2}((m-1)!!)^2.
\end{align*}

On the other hand,
\begin{align} \label{ppp4}
\liminf\limits_{n\to\infty}\I^{n}_{m}\geq  \liminf\limits_{n\to\infty}\sum_{\sigma\in \mathscr{P}_0}\widetilde{\I}^{n,\varepsilon, \sigma}_{m,\gamma}+\liminf\limits_{n\to\infty}\sum_{\sigma\in \mathscr{P}_1}\widetilde{\I}^{n,\varepsilon, \sigma}_{m,\gamma,1}+\liminf\limits_{n\to\infty}\sum_{\sigma\in \mathscr{P}_1}\widetilde{\I}^{n,\varepsilon, \sigma}_{m,\gamma,3}.
\end{align}
Using similar arguments as above,  $\liminf\limits_{n\to\infty}\sum\limits_{\sigma\in \mathscr{P}_1}\widetilde{\I}^{n,\varepsilon, \sigma}_{m,\gamma,3}$
is greater than
\begin{align*}
&\liminf\limits_{n\to\infty}\frac{m!}{((2\pi)^d\sqrt{n})^m} \sum_{\sigma\in \mathscr{P}_1} \int_{\widetilde{O}^{\gamma}_{m} \times \widetilde{O}^{\gamma}_{m}} \int_{T^{\sigma}_{\varepsilon,2}} \prod^{m/2}_{j=1}|\widehat{f}(y_{2j})|^2 \\
&\qquad\times \exp\Big(-\frac{\overline{\beta}(\gamma,n)}{2}\sum\limits^m_{i=1} |y_i|^2 \E(X_{u_i}-X_{u_{i-1}})^2\Big) \nonumber \\
&\qquad\qquad \times \exp\Big(-\frac{\overline{\beta}(\gamma,n)}{2}\sum\limits^m_{i=1} |\sum^m_{j=i} ( y_{\sigma(j)}-y_{\sigma(j)+1})|^2 \E(X_{v_i}-X_{v_{i-1}})^2\Big)\, dy\, du\, dv\\
&\geq\liminf\limits_{n\to\infty}\frac{m!}{((2\pi)^d\sqrt{n})^m} \sum_{\sigma\in \mathscr{P}_1} \int_{\widetilde{O}^{\gamma}_{m} \times \widetilde{O}^{\gamma}_{m}} \int_{T^{\sigma}_{\varepsilon,2}} \prod^{m/2}_{j=1}\left(|\widehat{f}(y_{2j})|^2\right)\\
&\quad \times \exp\Big(-\frac{\overline{\beta}(\gamma,n)}{2}\sum\limits^m_{i=1} |y_i|^2 \E(X_{u_i}-X_{u_{i-1}})^2\Big) \nonumber \\
&\quad\quad \times \exp\Big(-\frac{\overline{\beta}(\gamma,n)}{2}(1+\frac{m}{\gamma}) \sum\limits^m_{i=1, i: even} |y_{\overline{\sigma}(i)}|^2\E(X_{v_i}-X_{v_{i-1}})^2\Big)\\
&\quad\quad\quad \times \exp\Big(-\frac{\overline{\beta}(\gamma,n)}{2} \sum\limits^m_{i=1, i: odd} |\sum^m_{j=i} ( y_{\sigma(j)}-y_{\sigma(j)+1})|^2\E(X_{v_i}-X_{v_{i-1}})^2\Big)\, dy\, du\, dv.
\end{align*}
Hence,
\begin{align*}
\liminf\limits_{n\to\infty}\sum\limits_{\sigma\in \mathscr{P}_1}\widetilde{\I}^{n,\varepsilon, \sigma}_{m,\gamma,3}
&\geq \frac{2^m((m-1)!!)^2}{(2\pi)^{md/2}}  \left(\frac{1}{(2\pi)^{d}}\int_{\R^d}|\widehat{f}(z)|^2|z|^{-d}1_{\{|z|>\gamma \varepsilon\}}\, dz\right)^{m/2}\\
&\qquad\times \left( \int_{[0,+\infty)^2} \exp\Big(-\frac{1}{2}(1+c_{30}\beta_1(\gamma))(1+\frac{m}{\gamma}) \underline\alpha_2 (u^{2H}+v^{2H})\Big)\, du\,dv \right)^{m/2}\\
&\qquad\qquad\times \left(\underline{b}(\gamma) \underline{\alpha}_2(\gamma)\right)^{-\frac{md}{4}} \left[ \frac{\Gamma(\frac{m}{2}+(\frac{\underline{\alpha}_2(\gamma)}{\underline{\alpha}_1(\gamma)})^{\frac{d}{4}})}{(\frac{m}{2})!\Gamma((\frac{\underline{\alpha}_2(\gamma)}{\underline{\alpha}_1(\gamma)})^{\frac{d}{4}})} \right]^2 \left[\frac{2t\pi^{\frac{d}{2}}\Gamma^2(\frac{d+4}{4})}{(1+\frac{m}{\gamma})^{\frac{d}{4}}\Gamma(\frac{d+2}{2})}\right]^{m/2},
\end{align*}
where $\underline{b}(\gamma)=\frac{1}{2}+m\beta_1(\gamma)$, $\underline{\alpha}_1(\gamma)$ and $\underline{\alpha}_2(\gamma)$ are given in \eqref{eq-alpha}.

Recall inequalities (\ref{ppp1}) and (\ref{ppp2}). Taking $\varepsilon\to 0$ first and then $\gamma\to+\infty$ on the right hand side of (\ref{ppp4}) gives
 \begin{align*}
\liminf\limits_{n\to\infty}\I^{n}_{m}\geq \left[ \frac{\Gamma(\frac{m}{2}+(\frac{\alpha_2}{\alpha_1})^{\frac{d}{4}})}{(\frac{m}{2})!\Gamma((\frac{\alpha_2}{\alpha_1})^{\frac{d}{4}})} \right]^2 (D_{f,d}\, t)^{m/2}((m-1)!!)^2.
\end{align*}
This completes the proof of convergence of even moments. 
\end{proof}

\begin{remark}   When $m$ is an even integer, although the asymptotic $m$-th moment $\I_m^n$ given by \eqref{Inm} involves all permutations of $\{1,2,\dots,m\}$, only permutations in $\mathscr P_1$ given by \eqref{ctr} contribute to the limit when $n\to\infty$. Moreover, the arguments in \textbf{Step 3}, \textbf{Step 4} and \textbf{Step 5} show that the paring defined in \eqref{ctr} indicates a clear relationship between the first-order limit law and the corresponding second-order limit law.
\end{remark}

\bigskip

\noindent
{\bf Proof of Theorem \ref{thm2}}:  This follows from Lemmas \ref{lem1}-\ref{chain} and Proposition \ref{moments}. \hfill  {\scriptsize $\blacksquare$} 

\bigskip

\section{Appendix}
Here we give some lemmas  which are used to estimate moments when $n$ goes to infinity.  Recall that $Hd=2$. The generic constant $c$ is independent of $n$ and varies at different places.

\begin{lemma}\label{lemma0}
Let $a$ and $m$ be  positive constants. Then we have
\begin{align*}
&\int_{\R^d} e^{-a|x|^2}dx=c\, a^{-d/2}, \text{  and hence } \int_{|x|\le m} e^{-a|x|^2}dx\le c\, (1\wedge a^{-d/2});\\
&\int_0^\infty e^{-a^2 u^{2H}} du = c\, a^{-d/2}, \text{ and hence } \int_0^m e^{-a^2 u^{2H}} du \le  c\, (1\wedge a^{-d/2}).
\end{align*}
\end{lemma}
\begin{proof}
The results can be proven using change of variables and the fact that $Hd=2$. 
\end{proof}

\begin{lemma}\label{lemma0'}
For any $a>0$, $b>0$ and $\gamma>1$,
$$\int_{a/\gamma}^{a\gamma}\int_{b/\gamma}^{b\gamma} (u^{2H}+v^{2H})^{-d/2} dudv \le \frac{\pi}{H^2} \ln \gamma.$$
\end{lemma}

\begin{proof}
Let $u^H=r\cos \theta$ and $v^H=r\sin\theta$. Then
\begin{align*}
\int_{a/\gamma}^{a\gamma}\int_{b/\gamma}^{b\gamma} (u^{2H}+v^{2H})^{-d/2} dudv
&\leq\int^{\sqrt{a^{2H}+b^{2H}}\, \gamma^H}_{\sqrt{a^{2H}+b^{2H}}/\gamma^H} \int^{\frac{\pi}{2}}_0  r^{-d}  \frac{1}{H^2} r^{\frac{2}{H}-1} (\cos\theta)^{\frac{1}{H}-1} (\sin\theta)^{\frac{1}{H}-1} d\theta dr\\
&\leq\frac{\pi}{2H^2}\int^{\sqrt{a^{2H}+b^{2H}}\, \gamma^H}_{\sqrt{a^{2H}+b^{2H}}/\gamma^H}  \frac{1}{r}\, dr\\
&=\frac{\pi}{H} \ln \gamma,
\end{align*}
where in the second inequality we use the fact that $\frac{2}{H}=d$ and $H<1$. 
\end{proof}

\begin{lemma}\label{lemma1}
\begin{align*}
&\int^{e^{nt}}_{e^{-2mnt}}\int^{e^{nt}}_{e^{-2mnt}} (u^{2H}+v^{2H})^{-\frac{d}{2}} du\, dv <c\, n,\\
&\int^{e^{nt}}_0\int^{e^{nt}}_0 1\wedge (u^{2H}+v^{2H})^{-\frac{d}{2}} du\, dv <c\, n.
\end{align*}
\end{lemma}
\begin{proof}
It suffices to show the first inequality. Let $u^H=r\cos \theta$ and $v^H=r\sin\theta$. Then 
\begin{align*}
\int^{e^{nt}}_{e^{-2mnt}}\int^{e^{nt}}_{e^{-2mnt}} (u^{2H}+v^{2H})^{-\frac{d}{2}} du\, dv
&\leq \int^{2e^{nHt}}_{e^{-2mnHt}} \int^{\frac{\pi}{2}}_0  r^{-d}  \frac{1}{H^2} r^{\frac{2}{H}-1} (\cos\theta)^{\frac{1}{H}-1} (\sin\theta)^{\frac{1}{H}-1} d\theta dr\\
&\leq \frac{\pi}{2H^2}\int^{2e^{nHt}}_{e^{-2mnHt}}  \frac{1}{r} dr\\
&\leq c\, n,
\end{align*}
where in the second inequality we use the fact that $\frac{2}{H}=d$ and $H<1$. 
\end{proof}

\begin{lemma}\label{lemma2}
For $\alpha>0,$
$$\int_{\mathbb R^d}\int_{[n^{-m}, e^{nt}]^2}  |x|^\alpha e^{-|x|^2(s^{2H}+r^{2H})} dsdrdx\le  c_\alpha\, n^{mH\alpha}.$$
\end{lemma}
\begin{proof}
Integrating with respect to $x$ gives 
\begin{align*}
\int_{\mathbb R^d}\int_{[n^{-m}, e^{nt}]^2}  |x|^\alpha e^{-|x|^2(s^{2H}+r^{2H})} dsdrdx
=& c\int_{[n^{-m}, e^{nt}]^2}  (s^{2H}+r^{2H})^{-\frac{d+\alpha} 2} dsdr\\
\le & c_\alpha\, n^{mH\alpha},
\end{align*}
where the proof of the last inequality is similar to that of Lemma \ref{lemma1}.
\end{proof}

\begin{lemma}\label{lemma3} If $f$ is a real-valued bounded measurable function on $\R^d$ with $\int_{\R^d} f(x)\, dx=0$ and $\int_{\R^d}|f(x)||x|^{\beta}\, dx<\infty$ for some $\beta>0$, then
$$\int_{\mathbb R^d}\int_{[n^{-m}, e^{nt}]^2}  |\widehat f(x)|^2 e^{-|x|^2(s^{2H}+r^{2H})} dsdrdx<\infty.$$
\end{lemma}
\begin{proof} Using the change of variables
$u=|x|^{1/H}s$ and $v=|x|^{1/H}t$,
\begin{align*}
&\int_{\mathbb R^d}\int_{[n^{-m}, e^{nt}]^2}  |\widehat f(x)|^2 e^{-|x|^2(s^{2H}+r^{2H})} dsdrdx\\
\le& \int_{\R^d}|\widehat{f}(x)|^{2}|x|^{-\frac{2}{H}} dx \int^{\infty}_{0} e^{-u^{2H}} du  \int^{\infty}_{0} e^{-v^{2H}} dv\\
\le & c\, \int_{\R^d}|\widehat{f}(x)|^{2}|x|^{-d} dx,
\end{align*}
where the last integral is finite by Remark \ref{remark1}.
\end{proof}

\begin{lemma}\label{lemma4}
For any $A>0$, $$\sum_{\sigma\in \mathscr P_m} A^{m-|\sigma|}=\prod_{i=1}^m \big(A+(i-1)\big),  $$
where $|\sigma|$ is given in \eqref{|sigma|}.
\end{lemma}
\begin{proof}
The result can be proven by the method of induction. 
\end{proof}

\begin{lemma}\label{lemma7}
Let $m$ be an even integer and $\sigma\in\mathscr P_1$, where $\mathscr P_1$ is given in \eqref{ctr}. Recall that \[
A^{\sigma}_i=\big\{\sigma(i), \cdots, \sigma(m)\big\}\Delta \big\{\sigma(i)+1, \cdots, \sigma(m)+1\big\}
\]
and $$T_{\varepsilon, 2, \gamma}^\sigma\subset \Big\{|y_k|>\gamma \varepsilon, \text{ for  even } k;\; |y_k|<\varepsilon, \frac{|y_k|}{|y_l|}\notin (\frac1\gamma, \gamma) \text{ for  odd } k \text{ and } l \Big\}.$$
Then  on $T_{\varepsilon, 2, \gamma}^\sigma$, there exists  $\widetilde{\sigma}\in \mathscr{P}$,  such that $\widetilde{\sigma}(j)$  and $j$ have the same parity for $j=1,2,\dots, m$, and
\begin{align}\label{eq5.1}
 \sup_{j\in A^{\sigma}_i} |y_j| \ge |y_{\widetilde \sigma(i)}|. 
\end{align}
\end{lemma}

\begin{proof}  
When $i$ is even, noting that $\sigma\in \mathscr P_1$,  $A_i^\sigma$ contains only one $y$ with even index which is $y_{\overline{\sigma}(i) }$. Recall that for an even number $i,$ $\overline{\sigma}(i)$ equals $\sigma(i)$ if $\sigma(i)$ is even and $\sigma(i-1)$ otherwise.   Therefore, on  $T_{\varepsilon, 2, \gamma}^{\sigma}$, $\sup_{j\in A^{\sigma}_i} |y_j|=|y_{\overline{\sigma}(i)}|$, and we may just define $\widetilde \sigma(i)=\overline{\sigma}(i)$ for $i$ even.

When $i$ is odd, $\sigma\in \mathscr P_1$ implies that $A_i^\sigma$ only contains $y$ variables with odd indices. Define $\overline\sigma(i)=\sigma(i)$ if $\sigma(i)$ is odd, and $\overline\sigma(i)=\sigma(i+1)$ otherwise.  Clearly $A_{m-1}^\sigma=\{\overline{\sigma}(m-1),\overline{\sigma}(m-1)+2\}$. Here we use the convention $y_k=0$ if $k>m$.  Hence $\overline\sigma(m-1)\in A_{m-1}^{\sigma}.$ For $\overline{\sigma}(m-3)$, if it does not belong to $A_{m-3}^{\sigma}$, then it must coincide with $\overline{\sigma}(m-1)+2$, and hence lies in $A_{m-1}^\sigma$. Therefore, we have $\{\overline \sigma(m-3), \overline\sigma(m-1)\}\subset A_{m-3}^\sigma\bigcup A_{m-1}^\sigma.$  Continuing in this way,  we have
$$\Big\{\overline\sigma(2k-1), \overline\sigma(2k+1), \dots, \overline\sigma(m-1)\Big\}\subset \bigcup_{i=2k-1; ~ i\text{ odd}}^{m-1} A_{i}^\sigma ~~\text{ for } ~~ k=1, 2,\dots, m/2. $$
Noting that $\big\{\overline\sigma(1), \overline\sigma(3), \dots, \overline\sigma(m-1)\big\}=\big\{1,3, \dots, m-1\big\},$ there exists an odd number $k_i $ in each $A_i^\sigma$ with $i$ odd such that $\big\{k_i, i=1, 3, \dots, m-1\big\}=\big\{1, 3, \dots, m-1\big\}$, and thus we may define $ \widetilde \sigma(i)=  k_i$ for $i$ odd.  The proof is concluded. 
\end{proof}

\bigskip

$\begin{array}{cc}
\begin{minipage}[t]{1\textwidth}
{\bf Jian Song}\\
Department of Mathematics, University of Hong Kong, Hong Kong\\
\texttt{txjsong@hku.hk}
\end{minipage}
\hfill
\end{array}$

$\begin{array}{cc}
\begin{minipage}[t]{1\textwidth}
{\bf Fangjun Xu}\\
School of Statistics, East China Normal University, Shanghai 200241, China \\
NYU-ECNU Institute of Mathematical Sciences, NYU Shanghai, Shanghai 200062, China\\
\texttt{fangjunxu@gmail.com, fjxu@finance.ecnu.edu.cn}
\end{minipage}
\hfill
\end{array}$

$\begin{array}{cc}
\begin{minipage}[t]{1\textwidth}
{\bf  Qian Yu }\\
School of Statistics, East China Normal University, Shanghai 200241, China \\
\texttt{qyumath@163.com}
\end{minipage}
\hfill
\end{array}$

\end{document}